\documentclass{article}
\usepackage{amsmath,amsthm,amssymb}
\usepackage[usenames,dvipsnames]{xcolor}
\usepackage{enumerate}
\usepackage{graphicx}
\usepackage{cite}
\usepackage{comment}
\usepackage{oands}
\usepackage{tikz}
\usepackage{changepage}
\usepackage{bbm}
\usepackage{mathtools}
\usepackage[margin=1in]{geometry}
\usepackage[pagewise,mathlines]{lineno}
\usepackage{appendix}
\usepackage{stmaryrd} 
\usepackage{multicol}
\usepackage{microtype}
\usepackage[colorinlistoftodos]{todonotes}
\usepackage{enumitem}

\usepackage[pdftitle={A mating-of-trees approach to graph distances in random planar maps},
  pdfauthor={Ewain Gwynne, Nina Holden, and Xin Sun},
colorlinks=true,linkcolor=NavyBlue,urlcolor=RoyalBlue,citecolor=PineGreen,bookmarks=true,bookmarksopen=true,bookmarksopenlevel=2,unicode=true,linktocpage]{hyperref}

\theoremstyle{plain}
\newtheorem{thm}{Theorem}[section]

\newtheorem{lem}[thm]{Lemma}
\newtheorem{prop}[thm]{Proposition}

\newtheorem{definition}[thm]{Definition}

\def\@rst #1 #2other{#1}
\newcommand\MR[1]{\relax\ifhmode\unskip\spacefactor3000 \space\fi
  \MRhref{\expandafter\@rst #1 other}{#1}}
\newcommand{\MRhref}[2]{\href{http://www.ams.org/mathscinet-getitem?mr=#1}{MR#2}}

\theoremstyle{definition}
\newtheorem{defn}[thm]{Definition}
\newtheorem{remark}[thm]{Remark}

\newtheorem{ques}[thm]{Question}

\numberwithin{equation}{section}

\newcommand{\dsb}{\begin{adjustwidth}{2.5em}{0pt}
\begin{footnotesize}}
\newcommand{\dse}{\end{footnotesize}
\end{adjustwidth}}

\newcommand{\ssb}{\begin{adjustwidth}{2.5em}{0pt}}
\newcommand{\sse}{\end{adjustwidth}}

\newcommand{\aryb}{\begin{eqnarray*}}
\newcommand{\arye}{\end{eqnarray*}}
\def\alb#1\ale{\begin{align*}#1\end{align*}}
\def\allb#1\alle{\begin{align}#1\end{align}}
\newcommand{\eqb}{\begin{equation}}
\newcommand{\eqe}{\end{equation}}
\newcommand{\eqbn}{\begin{equation*}}
\newcommand{\eqen}{\end{equation*}}

\newcommand{\BB}{\mathbbm}
\newcommand{\ol}{\overline}

\newcommand{\op}{\operatorname}

\newcommand{\frk}{\mathfrak}

\newcommand{\ep}{\epsilon}
\newcommand{\rta}{\rightarrow}

\newcommand{\wt}{\widetilde}
\newcommand{\wh}{\widehat} 
\newcommand{\mcl}{\mathcal}

\newcommand{\bdy}{\partial}

\newcommand{\SLE}{\mathrm{SLE}}

\let\originalleft\left
\let\originalright\right
\renewcommand{\left}{\mathopen{}\mathclose\bgroup\originalleft}
\renewcommand{\right}{\aftergroup\egroup\originalright}

\newcommand{\Z}{\BB Z}
\newcommand{\N}{\BB N}
\newcommand{\cL}{\mathcal L}
\newcommand{\cR}{\mathcal R}
\newcommand{\cZ}{\mathcal Z}
\newcommand{\cV}{\mathcal V}

\newcommand{\cH}{\mathcal H}
\newcommand{\cG}{\mathcal G}
\newcommand{\Mn}{M_n}

\usepackage[normalem]{ulem}
\newcommand{\Ar}{A_{\mathrm{r}}}
\newcommand{\Ab}{A_{\mathrm{b}}}
\newcommand{\Ag}{A_{\mathrm{g}}}
\newcommand{\Tr}{T_{\mathrm{r}}}
\newcommand{\Tb}{T_{\mathrm{b}}}
\newcommand{\Tg}{T_{\mathrm{g}}}
\newcommand{\set}{\mathcal T_{\mathrm{SE}}     }
\newcommand{\nwt}{\mathcal T_{\mathrm{NW}}     }
\newcommand{\cE}{\mathcal E}
\newcommand{\cF}{\mathcal F}

\newcommand{\cO}{\mathcal O}
\newcommand{\cQ}{\mathcal Q}
\newcommand{\cP}{\mathcal P}
\newcommand{\map}{M}
\newcommand{\R}{\mathbb R}
\newcommand{\outf}{f_0}
\newcommand{\innV}{\mathring{\cV}}
\newcommand{\innF}{\mathring{\cF}}
\newcommand{\npole}{\mathrm{N}}
\newcommand{\spole}{\mathrm{S}}
\renewcommand{\path}{\lambda}
\newcommand{\UBOM}{\mathrm{UIBOM}}
\newcommand{\UIWT}{\mathrm{UIWT}}

\newcommand{\1}{\mathbf{1}}
\renewcommand{\P}{\mathbb{P}}
\newcommand{\E}{\mathbb{E}}

\newcommand{\nb}{\mathrm{Nb}}

\title{A mating-of-trees approach for graph distances\\ in random planar maps}

\date{  }
\author{
\begin{tabular}{c} Ewain Gwynne\\[-5pt]\small University of Cambridge \end{tabular}
\begin{tabular}{c} Nina Holden\\[-5pt]\small ETH Z\"urich \end{tabular}
\begin{tabular}{c} Xin Sun\\[-5pt]\small Columbia University \end{tabular}
}

\begin{document}

\maketitle
 
\begin{abstract} 
We introduce a general technique for proving estimates for certain random planar maps which belong to the $\gamma$-Liouville quantum gravity (LQG) universality class for $\gamma \in (0,2)$. The family of random planar maps we consider are those which can be encoded by a two-dimensional random walk with i.i.d.\ increments via a mating-of-trees bijection, and includes the uniform infinite planar triangulation (UIPT; $\gamma=\sqrt{8/3}$); and planar maps weighted by the number of different spanning trees ($\gamma=\sqrt 2$), bipolar orientations ($\gamma=\sqrt{4/3}$), or Schnyder woods ($\gamma=1$) that can be put on the map.

Using our technique, we prove estimates for graph distances in the above family of random planar maps. In particular, we obtain non-trivial upper and lower bounds for the cardinality of a graph distance ball consistent with the Watabiki (1993) prediction for the Hausdorff dimension of $\gamma$-LQG and we establish the existence of an exponent for certain distances in the map. 

The basic idea of our approach is to compare a given random planar map $M$ to a \emph{mated-CRT map}---a random planar map constructed from a correlated two-dimensional Brownian motion---using a strong coupling (Zaitsev, 1998) of the encoding walk for $M$ and the Brownian motion used to construct the mated-CRT map. This allows us to deduce estimates for graph distances in $M$ from the estimates for graph distances in the mated-CRT map which we proved (using continuum theory) in a previous work. In the special case when $\gamma=\sqrt{8/3}$, we instead deduce estimates for the $\sqrt{8/3}$-mated-CRT map from known results for the UIPT. 

The arguments of this paper do not directly use SLE/LQG, and can be read without any knowledge of these objects.
\end{abstract}

\tableofcontents

\section{Introduction}
\label{sec-intro}

\subsection{Overview}
\label{sec-overview}

A \emph{planar map} is a graph embedded in the plane, viewed modulo orientation-preserving homeomorphisms\footnote{It is more common in the planar map literature to define a planar map to be a graph embedded in the \emph{$2$-sphere}, rather than in the plane. We choose to embed in the plane since we will only be working with infinite planar maps and finite planar maps with boundary, so it is convenient to have a notion of ``infinity".}.
Random planar maps are a natural model of discrete random surfaces and are of fundamental importance in mathematical physics, probability, and combinatorics. 
Many interesting random planar maps converge in various topologies to $\gamma$-Liouville quantum gravity (LQG) surfaces --- a family of continuum random fractal surfaces --- with the parameter $\gamma \in (0,2]$ depending on the random planar map model. 
Such random planar maps are said to belong to the \emph{$\gamma$-LQG universality class}. 

The $\sqrt{8/3}$-LQG universality class contains uniform random planar maps, including uniform triangulations, quadrangulations, and uniform maps with unconstrained face degree. The $\gamma$-LQG universality class for $\gamma \not=\sqrt{8/3}$ contains random planar maps sampled with probability proportional to the partition function of some statistical mechanics model on the map, e.g., the uniform spanning tree ($\gamma = \sqrt2$),  bipolar orientations ($\gamma=\sqrt{4/3}$), the Ising model ($\gamma=\sqrt 3$), or the Schnyder wood ($\gamma=1$). It is expected (and in some cases known) that the universality class is affected only by changing the statistical mechanics model, not by changing the microscopic features of the planar map such as constraints on face or vertex degree.  
 
The goal of this paper is to set up a general framework for proving estimates for a certain family of random planar maps in the $\gamma$-LQG universality class for $\gamma \in (0,2)$ and apply this framework to obtain estimates for graph distances on these maps (see Section~\ref{sec-main-results} for precise statements).
The family of planar maps we consider are those which can be encoded by a two-dimensional random walk with i.i.d.\ increments via a so-called \emph{mating-of-trees bijection}. It includes the uniform infinite planar triangulation (UIPT) as well as infinite-volume limits of planar maps sampled with probability proportional to the number of different spanning trees, bipolar orientations, or Schnyder woods that can be put on the map. 

Using a strong coupling of the encoding walk with Brownian motion~\cite{kmt,zaitsev-kmt}, we will obtain a coupling of each map in our family with the \emph{$\gamma$-mated-CRT map}---a specific random planar map in the $\gamma$-LQG universality class constructed from a correlated two-dimensional Brownian motion (see Section~\ref{sec-peanosphere} for a precise definition). The mated-CRT map, in turn, is directly connected to SLE and LQG, so can be studied using continuum theory (we will not use this theory directly in the present paper, however).

To estimate graph distances in our family of random planar maps, we will show that in the above coupling the mated-CRT map and the other map are roughly isometric up to a polylogarithmic factor with high probability (Theorem~\ref{thm-map-coupling}). 
We will then deduce estimates for distances in the other map from the estimates for distances in the mated-CRT map which were proven in~\cite{ghs-dist-exponent}.
See Figure~\ref{fig-cont-to-discrete-dist} for a schematic illustration of our approach.

In particular, for each of the maps in our family, we obtain upper and lower bounds for the cardinality of a graph-metric ball consistent with the prediction of Watabiki~\cite{watabiki-lqg} for the dimension of $\gamma$-LQG (Theorem~\ref{thm-map-ball}); and establish the existence of an exponent $\chi$ which describes various distances in the planar map and which, at least for $\gamma \leq \sqrt 2$, we expect be the scaling exponent for the graph distance needed to get a non-trivial scaling limit in the Gromov-Hausdorff topology (Theorem~\ref{thm-map-dist}). 
To our knowledge, our results constitute the first non-trivial bounds for distances for planar maps weighted by spanning trees, bipolar orientations, or Schnyder woods. 
Since distances in uniform maps are already well understood, in the case when $\gamma =\sqrt{8/3}$ we instead obtain sharper bounds for distances in the $\sqrt{8/3}$-mated-CRT map than those in~\cite{ghs-dist-exponent} (Theorem~\ref{thm-6-ball}). 

The results of this paper can also be used to prove other results for random planar maps. For example, our coupling theorems are used in~\cite{gm-spec-dim,gh-displacement} to solve several open problems concerning the simple random walk on the UIPT and on the other random planar maps considered in this paper (including computing the exponents for the return probability and graph distance displacement of the walk). Our results are also used in~\cite{dg-lqg-dim} to relate the metric ball volume exponent for random planar maps to various exponents related to LQG --- arising from the Liouville heat kernel, Liouville graph distance, and Liouville first passage percolation --- and to obtain new bounds for these exponents. More generally, since there is a wide range of tools available for studying mated-CRT maps (due to their connection to SLE/LQG) we expect that the techniques developed in this paper could potentially be useful whenever one needs to prove estimates for random planar maps which can be encoded by mating-of-trees bijections.  

In the rest of this section, we provide some background on distances in random planar maps, define mated-CRT maps, introduce some notation, and state our main results. Section~\ref{sec-dist-comparison} contains the core of our argument, where we apply the strong coupling theorem of~\cite{zaitsev-kmt} to compare graph distances for mated-CRT maps and other random planar maps. In Section~\ref{sec-bijection} we review the mating-of-trees bijections for several particular random planar maps, then use the results of Section~\ref{sec-dist-comparison} to deduce our main results (since the bijections for the various maps are defined in slightly different ways, we need to treat the maps individually). In Section~\ref{sec-questions}, we discuss some open problems.

\begin{figure}[ht!]
 \begin{center}
\includegraphics[scale=1]{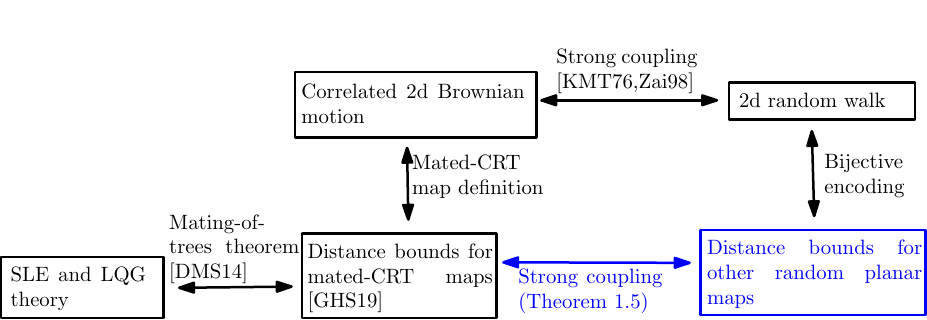} 
\caption[Logical structure of Chapter~\ref{chap-map-dist} and related works]{Schematic illustration of how the results of this paper fit together with other works (blue indicates what is done in the present work). As explained in Section~\ref{sec-peanosphere}, the mated-CRT map can be constructed from a two-dimensional correlated Brownian motion, or equivalently from a $\gamma$-LQG surface decorated by a space-filling SLE$_\kappa$ for $\kappa = 16/\gamma^2$ due to the results of~\cite{wedges}. In~\cite{ghs-dist-exponent}, we proved estimates for distances in mated-CRT maps using a combination of Brownian motion and SLE/LQG techniques. In the present paper, we transfer these estimates to other random planar map models which can be bijectively encoded by a two-dimensional random walk with i.i.d.\ increments. This is done by coupling the encoding walk with the Brownian motion used to define the mated-CRT map using~\cite{zaitsev-kmt}; and arguing that in such a coupling, the other random planar map and the mated-CRT map are roughly isometric up to a polylogarithmic factor. 
}\label{fig-cont-to-discrete-dist}
\end{center}
\end{figure}

\medskip

\noindent\textbf{Acknowledgements.}
We thank an anonymous referee for helpful comments on an earlier version of this article.
We thank Jason Miller for helpful discussions.
E.G.\ was partially funded by NSF grant DMS 1209044.
N.H.\ was supported by a fellowship from the Norwegian Research Council.
X.S.\ was supported by the  Simons Foundation as a Junior Fellow at Simons Society of Fellows.

\subsection{Background and context}
\label{sec-background}

\subsubsection{Mating-of-trees bijections and mated-CRT maps} 
\label{sec-peanosphere}

Many random planar map models in the $\gamma$-LQG universality class for $\gamma \in (0,2)$ can be encoded by a two-dimensional random walk via a bijection of \emph{mating-of-trees} (a.k.a.\ \emph{peanosphere}) type.
The reason for the name is that these bijections can be interpreted as gluing together two discrete trees (corresponding to the two coordinates of a two-dimensional random walk)  to construct a planar map decorated by a space-filling (``peano") curve. 

The first mating-of-trees bijection is that of Mullin~\cite{mullin-maps} (explained more explicitly in~\cite{bernardi-maps}), which encodes a random spanning tree-decorated planar map by means of a  nearest-neighbor random walk in $\BB Z^2$. Sheffield's hamburger-cheeseburger bijection~\cite{shef-burger} (which is a re-formulation of a bijection in~\cite{bernardi-maps}) is a generalization of Mullin's bijection which encodes a planar map decorated by the critical Fortuin Kasteleyn (FK) model~\cite{fk-cluster} with parameter $q\in (0,4)$. There are also mating-of-trees bijections for bipolar-oriented planar maps~\cite{kmsw-bipolar}, uniform random triangulations decorated by site percolation~\cite{bernardi-dfs-bijection,bhs-site-perc}, random planar maps decorated by a Schnyder wood~\cite{lsw-schnyder-wood}, and random planar maps decorated by an activity-weighted spanning tree~\cite{gkmw-burger}.  

For $\gamma \in (0,2)$, the \emph{$\gamma$-mated-CRT map} is a random planar map in the $\gamma$-LQG universality class, constructed by means of a semi-continuous variant of a mating-of-trees bijection with a correlated two-dimensional Brownian motion in place of a random walk. Suppose $\gamma \in (0,2)$ and $Z = (L,R)$ is a two-sided, two-dimensional Brownian motion with variances and covariances 
\eqb \label{eqn-bm-cov}
\op{Var}(L_t) = \op{Var}(R_t) = |t| \quad\op{and} \quad \op{Cov}(L_t,R_t) = -\cos(\pi\gamma^2/4) |t| ,\quad\forall t\in \BB R .
\eqe 
Note that the correlation of $L$ and $R$ ranges over $(-1,1)$ as $\gamma$ ranges over $(0,2)$. 

We define the mated-CRT map $\mcl G$ to be the graph with vertex set $\BB Z$, with two vertices $x_1, x_2 \in \BB Z$ with $x_1<x_2$ connected by an edge if and only if either
\allb \label{eqn-bm-inf-adjacency}
&\left( \inf_{t\in [x_1-1 , x_1]} L_t  \right) \vee \left( \inf_{t\in [x_2 - 1 , x_2]} L_t  \right) \leq \inf_{t\in [x_1  , x_2-1]} L_t \quad\op{or}\quad \notag \\
&\qquad \left( \inf_{t\in [x_1 - 1 , x_1]} R_t  \right) \vee \left( \inf_{t\in [x_2 -1 , x_2]} R_t  \right) \leq \inf_{t\in [x_1  , x_2-1]} R_t  .
\alle
Note that~\eqref{eqn-bm-inf-adjacency} holds in particular if $|x_1-x_2| = 1$.
The vertices $x_1$ and $x_2$ are connected by two edges if both conditions in~\eqref{eqn-bm-inf-adjacency} hold but $|x_1-x_2| > 1$. 
See Figure~\ref{fig-mated-crt-map} for an illustration of this definition and an explanation of why $\mcl G$ is a planar map (in fact, a triangulation).\footnote{The papers~\cite{ghs-dist-exponent,gms-tutte} consider the mated-CRT maps $\mcl G^\ep$ with vertex set $\ep\BB Z$, with each vertex $x$ corresponding to an interval of the form $[x-\ep,x]$. In this notation $\mcl G = \mcl G^1$. By Brownian scaling the law of $\mcl G^\ep$ does not depend on $\ep$.} 

\begin{figure}[ht!]
 \begin{center}
\includegraphics[scale=.7]{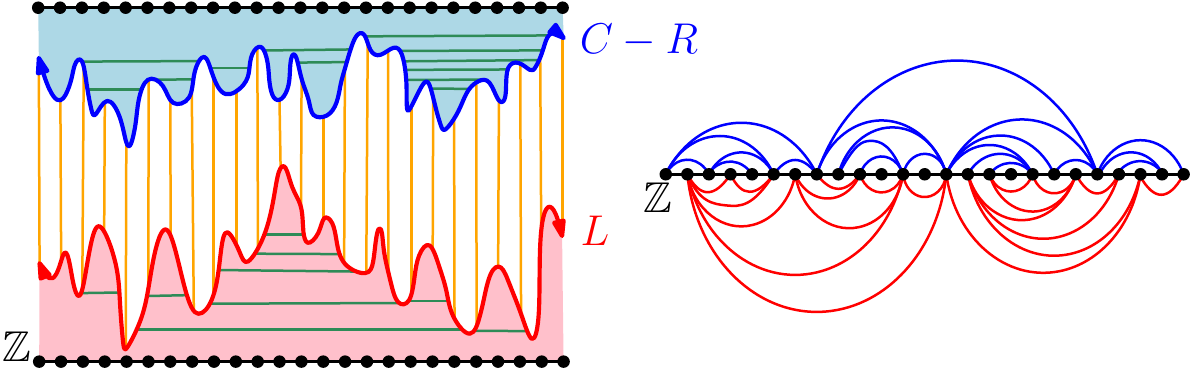} 
\caption[Definition of the mated-CRT map]{\textbf{Left:} To construct the mated-CRT map $\mcl G $ geometrically, one can draw the graph of $L$ (red) and the graph of $C-R$ (blue) for some large constant $C > 0$ chosen so that the parts of the graphs over some time interval of interest do not intersect. One then divides the region between the graphs into vertical strips (boundaries shown in orange) and identifies each strip with the horizontal coordinate $x\in \BB Z$ of its rightmost point. Vertices $x_1,x_2\in \BB Z$ are connected by an edge if and only if the corresponding strips either share a boundary segment (i.e., $|x_1-x_2| = 1$) or are connected by a horizontal line segment which lies under the graph of $L$ or above the graph of $C-R$. One such segment is shown in green in the figure for each pair of vertices for which this latter condition holds.
\textbf{Right:} One can draw the graph $\mcl G $ in the plane by connecting two vertices $x_1,x_2 \in  \BB Z$ by an edge by an arc above (resp.\ below) the blue line if the corresponding strips are connected by a horizontal segment above (resp.\ below) the graph of $L$ (resp.\ $C-R$); and connecting each pair of consecutive vertices of $ \BB Z$ by an edge. This gives $\mcl G $ a planar map structure under which it is a triangulation. Note that the two pictures do not correspond to the same mated-CRT map realization.
}\label{fig-mated-crt-map}
\end{center}
\end{figure}

The mated-CRT map is a discretized mating of the continuum random trees (CRT's)~\cite{aldous-crt1,aldous-crt2,aldous-crt3} associated with $L$ and $R$. As we will see in Section~\ref{sec-bijection}, this adjacency condition~\eqref{eqn-bm-inf-adjacency} is an exact continuum analogue of the condition for two vertices to be adjacent in terms of the encoding walk in various discrete mating-of-trees bijections.

It follows from~\cite[Theorem 1.9]{wedges} (and is explained in more detail in~\cite{ghs-dist-exponent}) that the mated-CRT map is directly connected to SLE-decorated Liouville quantum gravity. We will not directly use this relationship here, but we briefly mention it for the sake of context.
Suppose $h$ is the distribution corresponding to a $\gamma$-quantum cone (a type of $\gamma$-LQG surface) and $\eta$ is an independent whole-plane space-filling SLE$_\kappa$ from $\infty$ to $\infty$ for $\kappa =16/\gamma^2$, as constructed in~\cite[Sections 1.2.3 and 4.3]{ig4}. 
If we parametrize $\eta : \BB R\rta \BB C$ in such a way that the $\gamma$-LQG mass of $\eta([s,t])$ equals $t-s$ whenever $s < t$, then two vertices $x_1,x_2\in  \BB Z$ of $\mcl G$ are joined by an edge in $\mcl G$, i.e.,~\eqref{eqn-bm-inf-adjacency} holds, if and only if the ``cells" $\eta([x_1-1, x_1])$ and $\eta([x_2-1,x_2])$ intersect along a non-trivial connected boundary arc (this adjacency graph of cells is sometimes called the \emph{structure graph} associated with $(h,\eta)$). 
This gives an embedding of the mated-CRT map into $\BB C$ which can be analyzed using SLE and LQG estimates. 
 
Using a combination of the above relationship between mated-CRT maps and SLE/LQG and Brownian motion techniques, the paper~\cite{ghs-dist-exponent} proves a number of estimates for graph distances in the mated-CRT map $\mcl G $, valid for all $\gamma \in (0,2)$. This paper will transfer these estimates to other random planar maps.
We also mention the recent paper~\cite{gms-tutte}, which proves that the mated-CRT map under the so-called \emph{Tutte embedding} converges to $\gamma$-LQG by showing that it is close to the a priori embedding (coming from SLE-decorated LQG) discussed above. We will not need the results of this latter paper here.

\subsubsection{Graph distances in random planar maps}
\label{sec-dist-background}

To put our results in context, we provide here some background about graph distances in random planar maps.

Uniformly random planar maps (i.e., those in the $\sqrt{8/3}$-LQG universality class) can be analyzed by means of the Schaeffer bijection~\cite{schaeffer-bijection} and variants thereof (e.g.,~\cite{bdg-bijection}), which, unlike mating-of-trees bijections, describe graph distances in a simple way. In particular, for a number of different types of uniformly random planar maps (including $k$-angulations for $k=3$ or $k \geq 4$ even) it is known that if we sample a map of size $n$ and re-scale graph distances by $n^{-1/4}$ then the re-scaled metric converges in law with respect to the Gromov-Hausdorff topology to the \emph{Brownian map}~\cite{legall-uniqueness,miermont-brownian-map}, a continuum metric measure space constructed via a continuum analogue of the Schaeffer bijection. 
It is shown in~\cite{lqg-tbm1,lqg-tbm2,lqg-tbm3} that the Brownian map is equivalent, as a metric measure space, to a certain type of $\sqrt{8/3}$-LQG surface. 
Hence the metric space structure of uniformly random planar maps is well understood. 

For random planar maps in the $\gamma$-LQG universality class for $\gamma\not=\sqrt{8/3}$, however, the graph distance is not at all well understood. In this case there is no known bijective encoding analogous to the Schaeffer bijection which describes distances in a simple way.
However, a candidate for the Gromov-Hausdorff limit, the so-called \emph{$\gamma$-LQG metric} has recently been constructed in~\cite{gm-uniqueness}.

In fact, even the basic exponents for distances in random planar maps outside the $\sqrt{8/3}$-LQG universality class are unknown. 
Such exponents include the exponent for the diameter of a random planar map of size $n$ or the exponent for the cardinality of a metric ball of radius $n$ in an infinite-volume random planar map (which is expected to be its reciprocal). This latter exponent should be the same as the Hausdorff dimension $d_\gamma$ of the $\gamma$-LQG metric. It was predicted by Watabiki in~\cite{watabiki-lqg} that
\eqb \label{eqn-watabiki}
d_\gamma = 1 + \frac{\gamma^2}{4} + \frac14 \sqrt{(4+\gamma^2)^2 + 16\gamma^2} .
\eqe
The prediction~\eqref{eqn-watabiki} appears to match closely with numerical simulations~\cite{ambjorn-budd-lqg-dist} and there are several rigorous upper and lower bounds for quantities associated with $\gamma$-LQG which are consistent with this prediction~\cite{ghs-dist-exponent,mrvz-heat-kernel,andres-heat-kernel,dzz-heat-kernel,dg-lqg-dim,gp-lfpp-bounds,ang-discrete-lfpp}. 

However,~\eqref{eqn-watabiki} is known to be false for small values of $\gamma$. Indeed, Ding and Goswami~\cite{ding-goswami-watabiki} proved an estimate for various discrete approximations of the LQG metric which shows that $2/d_\gamma \leq 1 - c \frac{\gamma^{4/3}}{\log\gamma^{-1}}$ for small enough $\gamma$, whereas~\eqref{eqn-watabiki} gives $2/d_\gamma = 1 - O_\gamma(\gamma^2)$. 
It follows from the results of~\cite{dzz-heat-kernel,dg-lqg-dim} that one has an analogous estimate for the $\gamma$-mated CRT map when $\gamma$ is small, which then extends to other random planar maps (in particular, the biased bipolar-oriented maps in~\ref{item-kappa>8} below) using the results of the present paper. This then contradicts~\eqref{eqn-watabiki} for random planar maps.   
It is not clear whether~\eqref{eqn-watabiki} is true in the regime when $\gamma$ is not close to 0. 
See~\cite{ghs-dist-exponent,ding-goswami-watabiki,dg-lqg-dim} for further discussion of Watabiki's prediction and the problem of computing $d_\gamma$.

\subsection{Basic notation} 
\label{sec-basic-notation}

\noindent
We write $\BB N$ for the set of positive integers and $\BB N_0 = \BB N\cup \{0\}$. 
\vspace{6pt}

\noindent
For $a,b \in \BB R$ with $a<b$, we define the discrete intervals $[a,b]_{ \BB Z} := [a, b]\cap  \BB Z$ and $(a,b)_{ \BB Z} := (a,b)\cap \BB Z$.
\vspace{6pt}

\noindent
If $a$ and $b$ are two quantities depending on a variable $x$, we write $a = O_x(b)$ (resp.\ $a = o_x(b)$) if $a/b$ remains bounded (resp.\ tends to 0) as $x\rta 0$ or as $x\rta\infty$ (the regime we are considering will be clear from the context). We write $a = o_x^\infty(b)$ if $a = o_x(b^s)$ for every $s\in\BB R$.  
\vspace{6pt}

 \subsubsection{Notation for graphs} 
 \label{sec-graph-notation}
 
\noindent
For a planar graph $G$, we write $\mcl V(G)$, $\mcl E(G)$, and $\mcl F(G)$, respectively, for the set of vertices, edges, and faces of $G$, respectively. 
We write $v_1\sim v_2$ in $G$ if $v_1,v_2\in \cV(G)$ is connected by an edge in $G$.
For $v\in\mcl V(G)$, we write $\op{deg}(v;G)$ for the degree of $v$ (i.e., the number of edges with $v$ as an endpoint).
\vspace{6pt}

\noindent
A \emph{path} in $G$ is a function $P : [0,n]_{\BB Z }\rta \mcl V(G)$ for some $n\in\BB N$ such that $P(i)\sim P(i+1)$ for $i=0,\dots,n-1$. The \emph{length} $|P|$ of $P$ is the integer $n$. 
\vspace{6pt}

\noindent
For sets $A , B$ consisting of vertices and/or edges of $G$, we write $\op{dist}(A,B ; G)$ for the graph distance from $A$ to $B$ in $G$, i.e., the length of the shortest path from a vertex in $A$ or the endpoint of an edge in $A$ to a vertex or an endpoint of an edge in $B$. If $A = \{x\}$ and/or $B = \{y\}$ is a singleton, we drop the set brackets.
\begin{comment}
We note that if $G'\subset G$ is a subgraph and $A,B \subset G'$, then
\eqbn
\op{dist}(A,B ; G) \leq \op{dist}(A,B ; G') 
\eqen
since the quantity on the left is an infimum over a larger set of paths. 
\end{comment}
\vspace{6pt}

\noindent
For a set $A$ consisting of vertices and/or edges of $G$ and
$r>0$, we define the \emph{graph metric ball} $B_r(A ; G)$ to be the subgraph of $G$ consisting of the set of vertices of $G$ which lie at graph distance at most $r$ from $A$ and the set of edges of $G$ which join two such vertices (the set brackets are omitted if $A = \{x\}$ is a singleton).  
\vspace{6pt}

\noindent
We write $\op{diam}(A ; G ) :=\max_{x,y\in \mcl V(A)} \op{dist}(x,y;G) $ for the graph-distance diameter of $A$ with respect to the graph metric on $G$. We abbreviate $\op{diam}(G) := \op{diam}(G ; G)$. 
 \vspace{6pt}

\subsection{Main results} 
\label{sec-main-results}

We will consider the following infinite random planar maps $M$ with a distinguished oriented root edge $e_0$, each decorated by statistical mechanics models, which are the local limits of corresponding finite random planar maps in the Benjamini-Schramm topology~\cite{benjamini-schramm-topology}.  
\begin{enumerate}
\item $\gamma=\sqrt 2$ and $(M,e_0,T)$ is the uniform infinite spanning-tree decorated map, which can be constructed via the infinite-volume version of Mullin's bijection~\cite{mullin-maps,shef-burger,bernardi-maps,chen-fk}. \label{item-kappa8}

\item $\gamma=\sqrt{8/3}$ and $(M,e_0 , \theta)$ is the uniform infinite planar triangulation of type II~\cite{angel-schramm-uipt} (i.e., it allows multiple edges but not self-loops) decorated by a critical ($p=1/2$~\cite{angel-uihpq-perc}) site percolation configuration.\label{item-kappa6}

\item\label{item-kappa12}  $\gamma=\sqrt{4/3}$ and $(M,e_0, \mcl O)$ is the uniform infinite bipolar-oriented map, constructed via the infinite-volume version of the bijection in~\cite{kmsw-bipolar}\footnote{We give a careful proof that this and other infinite-volume limits of bipolar-oriented planar maps exist in Section~\ref{subsub:ubom}. See Proposition~\ref{prop:BS}.}.

\item $\gamma < \sqrt 2$ and $(M,e_0,\mcl O)$ is an infinite bipolar-oriented map with one of the other face degree distributions considered in~\cite[Section 2.3]{kmsw-bipolar} for which the face degree distribution has an exponential tail and the correlation between the coordinates of the encoding walk is $-\cos(\pi\gamma^2/4)$ (e.g., an infinite bipolar-oriented $k$-angulation for $k\geq 3$---in which case $\gamma=\sqrt{4/3}$---or one of the bipolar-oriented maps with biased face degree distributions considered in~\cite[Remark 1]{kmsw-bipolar}). \label{item-kappa>8}

\item $\gamma = 1$ and $(M,e_0 , T)$ is the uniform infinite Schnyder wood-decorated triangulation, as constructed in~\cite{lsw-schnyder-wood}.  \label{item-kappa16} 
\end{enumerate}
The reason why we consider these five cases is that in each of the above cases, the map $(M,e_0)$ and its associated statistical mechanics model can be encoded by a two-sided two-dimensional random walk $\mcl Z : \BB Z\rta\BB Z^2$ with i.i.d.\ increments via a mating-of-trees bijection (these bijections are reviewed in Section~\ref{sec-bijection}), so we can apply the two-dimensional version of the KMT strong coupling theorem~\cite{kmt,zaitsev-kmt} to couple $\mcl Z$ with the correlated Brownian motion $Z$ of~\eqref{eqn-bm-cov}. 

We emphasize that all of our results can be extended to any other random planar map model which can be encoded by a mating-of-trees bijection wherein the encoding walk can be strongly approximated by Brownian motion---the models listed above are just the most natural ones for which we presently know this to be the case. Section~\ref{sec-dist-comparison} considers a larger (but perhaps less natural) class of random planar map models for which slightly stronger versions of our results hold. See Section~\ref{sec-questions} for open problems and additional discussion concerning expanding the set of models for which the results of this paper apply. 

As in Section~\ref{sec-peanosphere}, in each of the above four cases we let $Z = (L,R)$ be a correlated two-dimensional Brownian motion correlation $-\cos(\pi\gamma^2/4)$, as in~\eqref{eqn-bm-cov}, and we let $\mcl G $ be the mated-CRT map constructed from $Z$ as in~\eqref{eqn-bm-inf-adjacency}. 
  
Just below, we will state a coupling result (Theorem~\ref{thm-map-coupling}) which allows us to compare $M$ and $\mcl G$.
The strong coupling theorem~\cite{zaitsev-kmt} only allows us to compare random walk and Brownian motion on a finite time interval, so to state our result we need to describe finite (approximate) subgraphs of $M$ and $\mcl G$ which correspond to finite time intervals. We start by defining such graphs for $\mcl G$. 

\begin{defn} \label{def-sg-restrict}
For $n\in\BB N$ we write $\mcl G_{n}$ for the subgraph of $\mcl G$ whose vertex set is $[-n,n]_{\BB Z}$ and whose edge set consists of all of the edges of $\mcl G$ between two such vertices ($\mcl G_n$ is called $\mcl G^1|_{[-n,n]}$ in~\cite{ghs-dist-exponent}). We also write $\bdy\mcl G_{n}$ for the subgraph of $\mcl G_{n}$ consisting of all vertices of $\mcl G_{n}$ which are connected by an edge to a vertex of $\mcl G\setminus \mcl G_{n}$; and all edges of $\mcl G_{n}$ which join two such vertices. 
\end{defn}

We now discuss the analogue of Definition~\ref{def-sg-restrict} for the map $M$. 
In each of the above situations, the corresponding bijection gives for each $n\in\BB N$ a planar map $\Mn$ associated with the random walk increment $\mcl Z|_{[-n,n]_{\BB Z}}$.
The maps $\Mn$ are defined slightly different in each of the four cases and will be defined precisely in Section~\ref{sec-bijection}. 

The map $\Mn$ is not necessarily a subgraph of $M$ (see, however, Remark~\ref{remark-subgraph}), but there is a distinguished boundary\footnote{Recall that a planar map with boundary is a planar map $G$ with a distinguished face (the \emph{external face}), in which case the boundary $\bdy G$ is the set of vertices and edges on the boundary of this face.}
$\bdy \Mn$ and an ``almost" inclusion function
\eqb \label{eqn-inclusion-function} 
\iota_n : \Mn \rta M \quad \text{which is injective on $\Mn \setminus \bdy \Mn$},
\eqe 
i.e., a function from $\mcl V(\Mn) \cup \mcl E(\Mn) $ to $\mcl V(M) \cup \mcl E(M)  $. 
The injectivity of $\iota_n$ on $\Mn\setminus \bdy \Mn$ means that we can canonically identify $\Mn \setminus \bdy \Mn$ with a subgraph of $M $. We emphasize that $\iota_n$ is not injective on all of $\Mn$ in case~\ref{item-kappa6}: see Remark~\ref{remark-subgraph}. The map $\Mn$ possesses a canonical root vertex which is mapped to $\BB v$ by $\iota_n$, and which (by a slight abuse of notation) we identify with $\BB v$. 
 
To set up a correspondence between the vertex sets $\mcl V(\map_n)$ and $\mcl V( \mcl G_n) = [-n,n]_{\BB Z}$, we will define in Section~\ref{sec-bijection} for each $n \in\BB N$ functions
\eqb \label{eqn-peano-functions}
\phi_n : \mcl V(\Mn) \rta [-n,n]_{\BB Z}  \quad \op{and} \quad \psi_n : [-n,n]_{\BB Z} \rta \mcl V(\Mn) , \quad \op{with} \quad \phi_n(\BB v) = 0 \quad \op{and} \quad \psi_n(0) = \BB v , 
\eqe   
where $\BB v$ denotes the initial endpoint of $e_0$.  
Roughly speaking, the vertex $\psi_n(i)$ corresponds to the $i$th step of the walk $\mcl Z|_{[-n,n]_{\BB Z}}$ in the bijective construction of $(M,e_0,T)$ from $\mcl Z$ and $\phi_n$ is ``close" to being the inverse of $\psi_n$. 
However, the construction of $\Mn$ from $\mcl Z|_{[-n,n]_{\BB Z}}$ does not set up an exact bijection between $[-n, n]_{\BB Z}$ and the vertex set of $\Mn $, so the functions $\phi_n$ and $\psi_n$ are neither injective nor surjective. See Figure~\ref{fig-coupling-setup} for an illustration of the above definitions.

\begin{remark} \label{remark-subgraph}
In cases~\ref{item-kappa8} and \ref{item-kappa12} through~\ref{item-kappa16} above, each of the planar maps $\Mn$ can be canonically identified with a subgraph of $M$ and one can take $\iota_n$ to be the identity map. Moreover, there are functions $\phi : \mcl V(M) \rta \BB Z$ and $\psi : \BB Z \rta \mcl V(M)$ for which $\phi_n = \phi|_{\mcl V(\Mn)}$ and $\psi_n = \psi|_{[-n,n]_{\BB Z}}$. However, in case~\ref{item-kappa6} $\Mn$ cannot be identified with a subgraph of $M$ since it may be be the case that one or more pairs of vertices or edges of $\Mn$ get identified together to a single edge of $M$ at a later step of the bijection; see Figure~\ref{fig4}. 
\end{remark}

\begin{remark} \label{remark-edge-map}
For each of the random planar maps considered here, for $i\in\BB Z$ the translated encoding walk $ j\mapsto \mcl Z_{j + i} -\mcl Z_i$ has the same law as $\mcl Z$, so encodes rooted planar map with the same law as $(M,e_0)$, decorated by a statistical mechanics model. It is easy to see from the definitions of the bijections, as reviewed in Section~\ref{sec-bijection}, that the random planar map encoded by the translated walk is the same as $M$ but with a different choice of root edge which we denote by $\lambda(i)$ (the statistical mechanics model may not be the same). The function $ \lambda : \BB Z \rta \mcl E(M)$ is a bijection in cases~\ref{item-kappa6} through~\ref{item-kappa16} and is two-to-one in case~\ref{item-kappa8}.
It is tempting to try to define $\Mn$ to be the subgraph of $M$ whose edge set is $ \lambda([-n,n]_{\BB Z})$ and whose vertex set consists of the endpoints of edges in $ \lambda([-n,n]_{\BB Z})$. This definition does not work in general, however, since the resulting graph is not necessarily connected in cases~\ref{item-kappa12} and~\ref{item-kappa>8} and the resulting graph may have pairs of vertices which get identified at a later step in case~\ref{item-kappa6}, as in Remark~\ref{remark-subgraph}. However, with our definitions of $\Mn$ one can check that in each of cases~\ref{item-kappa8} through~\ref{item-kappa>8} above,
\eqb \label{eqn-edge-map}
\lambda^{-1} \left( \iota_n\left( \mcl E (\Mn)  \setminus \mcl E ( \bdy \Mn )   \right) \right) \subset  [-n,n]_{\BB Z}  \quad \op{and} \quad \lambda([-n,n-1]_{\BB Z} )\subset    \iota_n \left( \mcl E (\Mn) \right)   .
\eqe
In the case~\ref{item-kappa16} of the Schnyder wood-decorated map,~\eqref{eqn-edge-map} does not hold with our definitions since we treat such maps as a special case of bipolar-oriented maps for the sake of convenience (see Section~\ref{subsub:Schnyder}). However, we expect that one can use the bijection from~\cite{lsw-schnyder-wood} to give a different definition of $\Mn$ for which~\eqref{eqn-edge-map} holds.
\end{remark}

\begin{figure}[ht!]
 \begin{center}
\includegraphics[scale=.7]{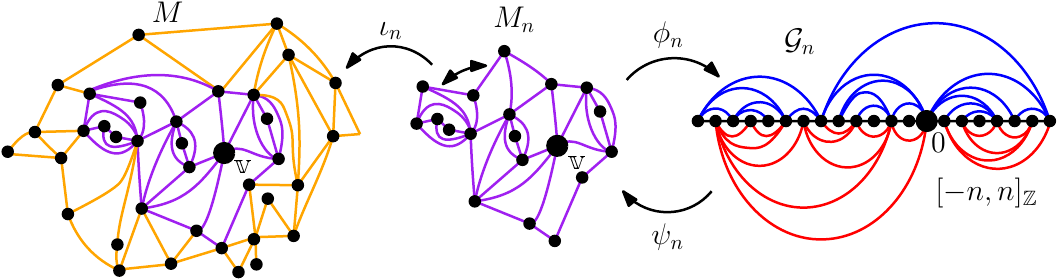} 
\caption[The setup for Section~\ref{sec-main-results}]{Illustration of the map $\Mn$, the ``almost inclusion" map $\iota_n : \Mn\rta M$, and the maps $\phi_n : \mcl V(\Mn) \rta \mcl V(\mcl G_n)$ and $\psi_n : \mcl V(\mcl G_n) \rta \mcl V(\Mn)$. Note that two edges of $\bdy \Mn$ get identified when we apply $\iota_n$. Theorem~\ref{thm-map-coupling} asserts that $\phi_n$ and $\psi_n$ are rough isometries up to polylogarithmic factors.
}\label{fig-coupling-setup}
\end{center}
\end{figure}

The main result which allows us to compare the graph metrics on $M$ and $\mcl G $ is Theorem~\ref{thm-map-coupling} just below, which says that these two maps are roughly isometric (up to a polylogarithmic factor), in the following sense.

\begin{defn} \label{def-rough-isometry}
Let $(X_1,d_1)$ and $(X_2,d_2)$ be metric spaces. For $a,b,c > 0$, a map $\phi : X_1 \rta X_2$ is called a \emph{rough isometry} with parameters $(a,b,c)$ if
\eqb \label{eqn-rough-iso-compare}
a^{-1} d_1(x , y) - b \leq  d_2(\phi(x) , \phi(y))   \leq  a d_1(x,y) + b ,\quad \forall x,y \in X_1 
\eqe
and for each $z\in X_2$, there is an $x\in X_1$ such that
\eqb
d_2(z, \phi(x)) \leq c .
\eqe
\end{defn}
 
\begin{thm} \label{thm-map-coupling}
Suppose we are in one of the five settings listed above. 
For each $A> 0$,  there is a constant $C=  C(A) > 0$ such that for each $n\in\BB N$, there is a coupling of $Z$ and $(M,e_0, T)$ such that with probability  $1-O_n(n^{-A})$, the map $\phi_{n }$ of~\eqref{eqn-peano-functions} is a rough isometry from $\Mn$ to $\mcl G_{n}$ (each equipped with their respective graph distances) with parameters $a = C (\log n)^4$, $b = 2$, and $c = C (\log n)^4$; and the map $\psi_{n}$ is a rough isometry from $\mcl G_{n}$ to $\Mn$ with these same values of $a,b,$ and $c$. 
\end{thm}
 
In cases~\ref{item-kappa8} and~\ref{item-kappa12} through~\ref{item-kappa16} above, Theorem~\ref{thm-map-coupling} will be used to deduce estimates for the graph metric on $M$ from the estimates for the graph metric on the mated-CRT map $\mcl G $ from~\cite{ghs-dist-exponent} (in case~\ref{item-kappa6}, these estimates for $M$ are weaker than the ones obtained from a Schaeffer-type bijection or from peeling). We now state the results we can obtain for the graph metric on $M$. 
Our first main result concerns the cardinality of graph metric balls in $M$, and is a planar map analogue of~\cite[Theorem 1.10]{ghs-dist-exponent}. 
 
\begin{thm} \label{thm-map-ball}
In each of the settings listed above, the following is true. Let
\allb \label{eqn-ball-scaling-exponent}
d_- :=  \frac{2\gamma^2}{4 + \gamma^2 - \sqrt{16 + \gamma^4}}   
\quad \op{and}\quad 
d_+ :=  2 + \frac{\gamma^2}{2} + \sqrt2 \gamma .
\alle
Then for each $u> 0$, there exists $c =c(u) > 0$ such that with probability $1-O_n(n^{-c})$, we have the following bounds for the number of vertices in graph metric balls of $M$ centered at the root edge:  
\eqb \label{eqn-map-ball}
  n^{d_-   - u} \leq  \# \mcl V\left( B_n \left( e_0 ; M \right)\right)  \leq n^{d_+ + u}    .
\eqe 
\end{thm}

As noted in~\cite{ghs-dist-exponent}, it is expected that typically $\# \mcl V( B_n( e_0 ; M) ) = n^{d_\gamma + o_n(1)}$, where $d_\gamma$ is the Hausdorff dimension of the $\gamma$-LQG metric (in fact, it can be proven that this is the case by combining~\cite[Theorem 1.6]{dg-lqg-dim} and~\cite[Corollary 1.7]{gp-kpz}, both of which were established after this paper). Hence Theorem~\ref{thm-map-ball} gives upper and lower bounds for $d_\gamma$. See Figure~\ref{fig-ball-bounds} for a graph of these bounds and a table giving their values for various special planar map models. 

The papers~\cite{dg-lqg-dim,gp-lfpp-bounds} build on the results of the present paper to prove new upper and lower bounds for $\# \mcl V( B_n( e_0 ; M) )$ which are sharper than those of Theorem~\ref{thm-map-ball} for most values of $\gamma$. For example, the bound in the case of spanning-tree weighted maps is improved to $3.550408 \leq d_{\sqrt 2} \leq 3.63299$~\cite[Corollary 2.5]{gp-lfpp-bounds}. The proof of these new bounds relies crucially on the results of the present paper both to transfer from the mated-CRT map to other random planar maps and to obtain that the volume growth exponent for the $\sqrt{8/3}$-mated-CRT map is 4 (Theorem~\ref{thm-6-ball}), which is used in combination with the monotonicity (in $\gamma$) of various exponents to obtain the new bounds.

\begin{figure}[ht!]
 \begin{center}
\includegraphics[scale=.6]{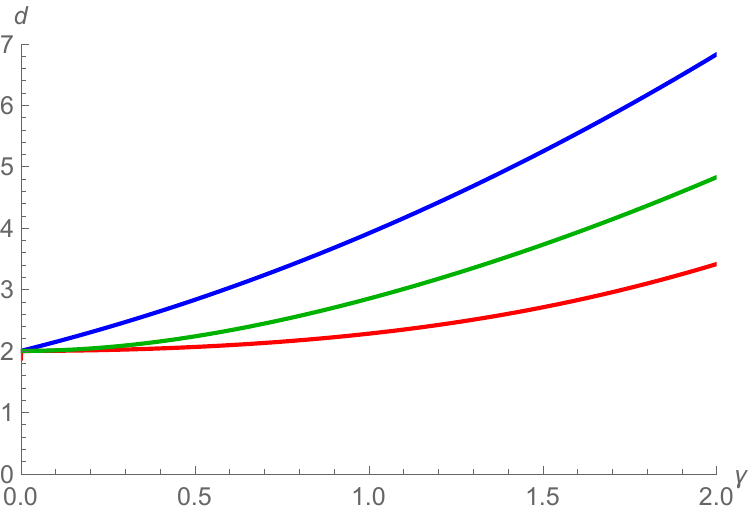} 
\hspace{6pt}
\includegraphics[scale=.9]{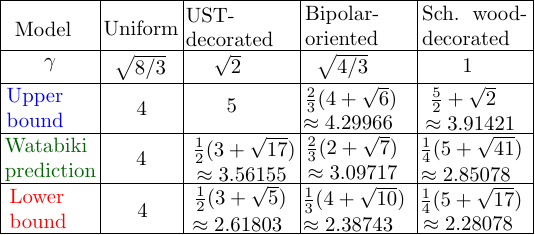} 
\caption[Upper and lower bounds for the cardinality of a metric ball in $M$]{\textbf{Left:} graph of the upper and lower bounds for the cardinality of a graph metric ball from Theorem~\ref{thm-map-ball} (blue and red) and the Watabiki prediction~\eqref{eqn-watabiki} for the Hausdorff dimension of $\gamma$-LQG (green). 
This graph was produced using Mathematica, and also appears in~\cite{ghs-dist-exponent}. \textbf{Right:} Table of the exact and approximate values of the the known upper/lower bounds and the Watabiki prediction for several special planar map models.
}\label{fig-ball-bounds}
\end{center}
\end{figure}

\begin{comment}
wat[gamma_] := 1 + gamma^2/4 + (1/4) Sqrt[(4 + gamma^2)^2 + 16 gamma^2]
low[gamma_] := 2 gamma^2/(4 + gamma^2 - Sqrt[16 + gamma^4])
up[gamma_] := 2 + gamma^2/2 + Sqrt[2] gamma

Sqrt[2]

FullSimplify[low[Sqrt[2]]]
N[%]

FullSimplify[wat[Sqrt[2]]]
N[%]

FullSimplify[up[Sqrt[2]]]
N[%]

Sqrt[4/3]

FullSimplify[low[Sqrt[4/3]]]
N[%]

FullSimplify[wat[Sqrt[4/3]]]
N[%]

FullSimplify[up[Sqrt[4/3]]]
N[%]

1

FullSimplify[low[1]]
N[%]

FullSimplify[wat[1]]
N[%]

FullSimplify[up[1]]
N[%]
\end{comment}

In~\cite[Theorem 1.12]{ghs-dist-exponent}, it is shown that for $\gamma \in (0,2)$, there exists an exponent $\chi=\chi(\gamma) > 0$ which can be defined as the limit
\eqb \label{eqn-chi-def}
\chi = \lim_{ n \rta \infty} \frac{\log \BB E\left[ \op{diam}\left(\mcl G_{n} \right) \right]}{\log n}
\eqe
and which satisfies
\eqb \label{eqn-chi-bounds}
\frac{1}{ 2 + \gamma^2/2 + \sqrt2 \gamma}   \vee \left(1 - \frac{2}{\gamma^2} \right) \leq \chi \leq  \frac12  .
\eqe 
Our next result shows that this same exponent $\chi$ also describes distances in the other random planar maps considered above. 

\begin{thm} \label{thm-map-dist}
Suppose we are in one of the settings listed above and let $\chi$ be as in~\eqref{eqn-chi-def} (for our given choice of $\gamma$). 
For $u > 0$ we have (with the notation $o_n^\infty(n)$ as in Section~\ref{sec-basic-notation}),
\eqb \label{eqn-dist-upper0}
\BB P\left[ \op{diam}\left( \Mn \right) \leq  n^{\chi + u} \right] \geq 1  - o_n^\infty(n)  
\eqe 
and 
\eqb \label{eqn-dist-lower0}
\BB P\left[ \op{diam} \left( \Mn \right)      \geq   n^{\chi - u}  \right] \geq n^{-o_n(1)} .
\eqe
In particular, for each $p > 0$,
\eqb \label{eqn-map-moment}
\BB E\left[ \op{diam} \left( \Mn \right)^p \right] = n^{\chi p  + o_n(1)} .
\eqe 
\end{thm}

We note that~\eqref{eqn-dist-lower0} does not tell us that $\op{diam} \left( \Mn \right)      \geq   n^{\chi - o_n(1)}$ with probability tending to 1 as $n\rta\infty$, although we expect this to be the case. In the special case of spanning-tree weighted planar maps ($\gamma = \sqrt 2$), a much stronger version of this statement is proven in~\cite[Theorem 3.1]{gp-dla}.

In case~\ref{item-kappa6}, certain exponents for graph distances in the UIPT are already known (see, e.g.,~\cite{angel-peeling}), so we can deduce estimates for graph distances in the $\sqrt{8/3}$-mated-CRT map from our coupling result Theorem~\ref{thm-map-coupling}. For example, we get the correct exponent for the cardinality of a metric ball. 

\begin{thm} \label{thm-6-ball}
Suppose $\gamma=\sqrt{8/3}$, so the correlation of $L$ and $R$ is $1/2$. For each $u> 0$, the graph metric ball in the mated-CRT map satisfies
\eqb \label{eqn-map-ball6}
\lim_{n\rta \infty} \BB P\left[ n^{4   - u} \leq  \# \mcl V\left( B_n \left( 0 ; \mcl G  \right)  \right) \leq n^{4 + u} \right] = 1  .
\eqe 
\end{thm}

\subsection{A stronger coupling theorem}
\label{sec-stronger-coupling}

Our proof of Theorem~\ref{thm-map-coupling} yields a stronger statement than just the existence of rough isometries between the maps $\mcl G_{n}$ and $\Mn$, which we state here.

\begin{thm} \label{thm-map-count} 
Suppose we are in one of the five settings listed at the beginning of Section~\ref{sec-main-results}
For each $A> 0$,  there is a constant $C = C(A) >0$ such that for each $n\in\BB N$, there is a coupling of $Z$ and $(M,e_0, T)$ such that with probability $1-O_n(n^{-A})$, the following is true (with $\phi_n$ and $\psi_n$ as in~\eqref{eqn-peano-functions}).  
\begin{enumerate} 
\item For each $v_1,v_2 \in  \mcl V(\Mn )$ with $v_1 \sim v_2$ in $\Mn$, there is a path $P_{v_1,v_2}^{\mcl G}$ from $\phi_{n}(v_1)$ to $\phi_{n}(v_2)$ in $\mcl G_{n}$ with $|P_{v_1,v_2}^{\mcl G}| \leq C (\log n)^4$; and each $i\in [-n , n]_{\BB Z}$ is hit by a total of at most $ C (\log n)^7 $ of the paths $P_{v_1,v_2}^{\mcl G}$ for $v_1,v_2 \in  \mcl V(M_{ [-n , n] } )$ with $v_1 \sim v_2$.  \label{item-map-count-G}
\item For each $i_1,i_2 \in    [-n , n]_{\BB Z}$ with $i_1 \sim i_2$ in $\mcl G_n$, there is a path $P_{i_1,i_2}^M$ from $\psi_n(i_1)$ to $\psi_n(i_2)$ in $\Mn$ with $|P_{i_1,i_2}^{M}| \leq C (\log n)^4$; and each $v \in \mcl V(\Mn)$ is hit by a total of at most $ C (\log n)^7$ of the paths $P_{i_1,i_2}^M$ for $i_1,i_2 \in    [-n , n]_{\BB Z}$ with $i_1 \sim i_2$.  \label{item-map-count-M} 
\item We have $\op{dist}\left(\psi_{n}(\phi_{n}(v)) , v ; \Mn \right) \leq C ( \log n)^4$ for each $v\in \mcl V(\Mn)$ and $ \op{dist}\left(\phi_{n}(\psi_{n}(i)) , i ; \mcl G_{n} \right) \leq C (\log n)^4 $ for each $i \in [-n , n]_{\BB Z}$. \label{item-map-count-close}
\end{enumerate}
\end{thm}

Theorem~\ref{thm-map-count} is strictly stronger than Theorem~\ref{thm-map-coupling}. Indeed, Theorem~\ref{thm-map-count} trivially implies Theorem~\ref{thm-map-coupling} (we record this fact as Lemma~\ref{lem-map-thms} below for the sake of reference). However, the upper bounds on the total number of the paths $P_{v_1,v_2}^{\mcl G}$ or $P_{i_1,i_2}^M$ which hit a given vertex in Theorem~\ref{thm-map-count} are not implied by Theorem~\ref{thm-map-coupling}. These bounds will be important in~\cite{gm-spec-dim}. 

\begin{lem} \label{lem-map-thms}
Let $A > 0$ and $n\in\BB N$. Any coupling of $Z$ with $(M,e_0 , T)$ which satisfies the conditions of Theorem~\ref{thm-map-count} for this choice of $A$ and $n$ also satisfies the conditions of Theorem~\ref{thm-map-coupling} for this choice of $A$ and $n$, with the same value of $C$.
\end{lem}
\begin{proof} 
The second condition from Definition~\ref{def-rough-isometry} with $c = C (\log n)^4$ for either $\phi_{n}$ or $\psi_{n}$ is immediate from condition~\ref{item-map-count-close} of Theorem~\ref{thm-map-count}, so we just need to check the first condition (concerning the amount by which each of $\phi_{n}$ and $\psi_{n}$ distort distances).

We first argue that the conditions of Theorem~\ref{thm-map-count} imply that
\eqb \label{eqn-map-coupling-G}
 \op{dist}\left( \phi_{n}(v_1) , \phi_{n}(v_2) ; \mcl G_{n} \right)  
  \leq  C (\log n)^4 \op{dist}\left( v_1 , v_2  ;  \Mn \right)   ,\quad \forall v_1,v_2 \in \mcl V(\Mn) .
\eqe 
Indeed, suppose we are given $v_1,v_2 \in \mcl V(\Mn)$ and let $P$ be a geodesic in $\mcl V(\Mn)$ from $v_1$ to $v_2$. 
Condition~\ref{item-map-count-G} from Theorem~\ref{thm-map-count} implies that 
\eqbn
\op{dist}\left( \phi_{n}(P(k-1)) , \phi_{n}(P(k)) ; \mcl G_{n} \right) \leq C (\log n)^4 ,
\eqen
for each $k \in [1,|P|]_{\BB Z}$. Summing over all such $k$ and applying the triangle inequality yields~\eqref{eqn-map-coupling-G}. Similarly, condition~\ref{item-map-count-M} from Theorem~\ref{thm-map-count} implies that
\eqb \label{eqn-map-coupling-M} 
 \op{dist}\left( \psi_{n}(i_1) , \psi_{n}(i_2) ; \Mn \right)  
  \leq  C (\log n)^4 \op{dist}\left( i_1 , i_2  ; \mcl G_{n} \right)    , \quad \forall i_1,i_2 \in [-n , n]_{\BB Z} .
\eqe   

Combining~\eqref{eqn-map-coupling-M} (applied with $i_1 = \phi_{n}(v_1)$ and $i_2 = \phi_{n}(v_2)$) and condition~\ref{item-map-count-close} from Theorem~\ref{thm-map-count} and using the triangle inequality gives that for $v_1,v_2 \in \mcl V(\Mn)$, 
\eqb \label{eqn-map-coupling-lower-M}
 \op{dist}\left( \phi_{n}(v_1) , \phi_{n}(v_2) ; \mcl G_{n} \right)  
 %\geq C^{-1} (\log n)^{-4} \op{dist}\left( \psi_n(\phi_n(v_1))  , \psi_n(\phi_n(v_2) ) ; \Mn \right) 
 \geq C^{-1} (\log n)^{-4} \op{dist}\left(v_1,v_2 ; \Mn \right)  - 2   .
\eqe 
Similarly, combining~\eqref{eqn-map-coupling-G} and condition~\ref{item-map-count-close} from Theorem~\ref{thm-map-count} gives
\eqb \label{eqn-map-coupling-lower-G}
 \op{dist}\left( \psi_{n}(i_1) , \psi_{n}(i_2) ; \Mn \right)   
 %\geq  C^{-1} (\log n)^{-4} \op{dist}\left( \phi_{n}(\psi_{n}(i_1))  ,\phi_{n}(\psi_{n}(i_2))   ;  \mcl G_{n} \right) 
 \geq C^{-1} (\log n)^{-4} \op{dist}\left(i_1 , i_2  ;  \mcl G_{n} \right)  - 2 .
\eqe
Combining~\eqref{eqn-map-coupling-G},~\eqref{eqn-map-coupling-M},~\eqref{eqn-map-coupling-lower-M}, and~\eqref{eqn-map-coupling-lower-G} gives the first condition from Definition~\ref{def-rough-isometry} with $a = C (\log n)^4$ and $b=2$.
\end{proof}

\subsection{Comparison of graph metric balls}
\label{sec-ball-compare}

In order to use Theorem~\ref{thm-map-coupling} (or Theorem~\ref{thm-map-count}) to compare graph metric balls in $M$ and $\mcl G$, we need to make sure that such balls are contained in $\Mn$ and $\mcl G_n$, respectively, with high probability. This is the purpose of the present subsection. In particular, we will establish the following lemma, which is sufficient for our purposes.

\begin{lem} \label{lem-ball-iso}
Suppose we are in one of the five settings listed at the beginning of Section~\ref{sec-main-results}. 
For each $A > 0$, there exists $K = K(A) > 0$ such that with probability at least $1-O_n(n^{-A})$, we have $B_n(0;\mcl G) \subset \mcl G_{n^K}$ and the map $\iota_{n^K}$ of~\eqref{eqn-inclusion-function} restricts to a graph isomorphism from $B_n(\BB v ; M_{n^K})$ to $B_n(\BB v; M)$. 
\end{lem} 

To prove Lemma~\ref{lem-ball-iso} we will need the following lemma.

\begin{lem} \label{lem-ball-layers}
Suppose $m < n$ are positive integers such that 
\allb \label{eqn-walk-inf-condition0}
 \max\left\{ \min_{i \in \left[  m +1 , n \right]_{\BB Z} } \mcl L_i , \min_{i \in \left[ - n  ,  - m -1 \right]_{\BB Z} } \mcl L_i \right\}
&< \min_{i \in \left[-m -1  , m+ 1 \right]_{\BB Z} } \mcl L_i   
\alle
and the same holds with $\mcl R$ in place of $\mcl L$. 
Then the map $\phi_{n}$ of~\eqref{eqn-peano-functions} satisfies
\eqb
\phi_{n}(\bdy M_n) \cap [-m , m ]_{\BB Z} = \emptyset  .
\label{eq:bdy-disjoint}
\eqe
\end{lem} 

In each of the five cases we consider, Lemma~\ref{lem-ball-layers} is an easy consequence of the definitions in Section~\ref{sec-bijection}. 
We will check the lemma separately for each case in the appropriate subsection of Section~\ref{sec-bijection}.

\begin{proof}[Proof of Lemma~\ref{lem-ball-iso}]
In light of~\eqref{eqn-inclusion-function}, we only need to find $K =K(A)$ as in the statement of the lemma such that
\eqb \label{eqn-ball-include}
\BB P\left[ B_n(0;\mcl G) \subset \mcl G_{n^K}  \:\op{and} \:   B_n(\BB v ; M_{n^K} ) \subset M_{n^K}\setminus \bdy M_{n^K} \right] \geq 1 - O_n(n^{-A}) .
\eqe 

By~\cite[Corollary 3.2]{ghs-dist-exponent}, there exists $K   > 2$ such that with probability at least $1-O_n(n^{-A})$, 
\eqb \label{eqn-sg-ball-bdy}
B_{n^{2}}(0; \mcl G)  \subset \mcl G_{n^{K }}  . % \quad \text{which implies} \quad B_{n^2}(0; \mcl G)  \cap \bdy \mcl G_{n^{K }} = \emptyset.
\eqe 
We will now deduce from~\eqref{eqn-sg-ball-bdy}, Lemma~\ref{lem-ball-layers}, and Theorem~\ref{thm-map-count} that after possibly increasing $K$, we also have $\BB P\left[B_n(\BB v ; M_{n^{K^2} } )  \subset M_{n^{K^2}}\setminus \bdy M_{n^{K^2}} \right] \geq 1 - O_n(n^{-A})$. 
Since~\eqref{eqn-sg-ball-bdy} implies the analogous statement with $K^2$ in place of $K$, this will imply~\eqref{eqn-ball-include} with $K^2$ in place of $K$. 

Recall that each of the two coordinates of $\mcl Z$ is a one-dimensional random walk started at zero with i.i.d.\ increments having an exponential tail at $\infty$.
Just below, we will explain using basic random walk estimates that if $K$ is chosen to be sufficiently large, in a manner depending only on $A$, then with probability at least $1-O_n(n^{-A})$, 
\allb \label{eqn-walk-inf-condition}
 \max\left\{ \min_{i \in \left[  n^{K} + 1 , n^{K^2} \right]_{\BB Z} } \mcl L_i , \min_{i \in \left[ - n^{K^2} - 1  ,  - n^{K } \right]_{\BB Z} } \mcl L_i \right\}
 < \min_{i \in \left[-n^{K } -1 , n^{K }+ 1 \right]_{\BB Z} } \mcl L_i   
\alle
and the same holds with $\mcl R$ in place of $\mcl L$. One way to justify this is as follows. Let $B$ be a standard linear Brownian motion. Using the reflection principle to estimate the running minima of $B$, one gets that if $K$ is chosen to be sufficiently large, in a manner depending only on $A$, then with probability at least $1-O_n(n^{-A})$, 
\allb \label{eqn-bm-inf-condition}
\max\left\{ \inf_{t \in \left[  n^{K} + 1 , n^{K^2} \right] } B_t  , \inf_{t \in \left[ - n^{K^2} - 1  ,  - n^{K } \right]  } B_t \right\}
 < \inf_{t \in \left[-n^{K } -1 , n^{K }+ 1 \right]_{\BB Z} } B_t   - n .
\alle
We then deduce~\eqref{eqn-walk-inf-condition} from~\eqref{eqn-bm-inf-condition} and the KMT coupling theorem~\cite{kmt} (see~\cite[Theorem 7.1.1]{lawler-limic-walks} for the precise version which we use here). 
 
By Lemma~\ref{lem-ball-layers}, if~\eqref{eqn-walk-inf-condition} holds, then  
\eqb \label{eqn-no-bdy-intersect}
\phi_{n^{K^2}}(\bdy M_{n^{K^2}}) \cap \left[-n^{K }, n^{K }\right]_{\BB Z} = \emptyset .
\eqe
By Theorem~\ref{thm-map-coupling}, there is a $C = C(A) > 1$ such that with probability at least $1- O_n(n^{-A})$, 
\eqb \label{eqn-bdy-intersect-dist} 
\op{dist}\left( v , \BB v ; M_{n^{K^2}}\right) \geq C^{-1} (\log n)^{-4} \op{dist}\left( \phi_{n^{K^2}}(v) , 0 ; \mcl G_{n^{K^2}} \right)  -  2 ,
\quad \forall v \in \mcl V(\mcl M_{n^{K^2}}) .
\eqe

Henceforth assume that~\eqref{eqn-sg-ball-bdy},~\eqref{eqn-no-bdy-intersect}, and~\eqref{eqn-bdy-intersect-dist} all hold, which happens with probability at least $1-O_n(n^{-A})$. 
If $v \in \bdy M_{n^{K^2}}$, then by~\eqref{eqn-no-bdy-intersect}, $\phi_{n^{K^2}}(v) \notin \left[-n^{K }, n^{K }\right]_{\BB Z} = \mcl V(\mcl G_{n^{K }})$. 
By~\eqref{eqn-sg-ball-bdy}, this means that 
\eqb
\op{dist}\left( \phi_{n^{K^2}}(v) , 0 ; \mcl G_{n^{K^2}} \right) \geq \op{dist}\left( \phi_{n^{K^2}}(v) , 0 ; \mcl G  \right) \geq n^2. 
\eqe
By~\eqref{eqn-bdy-intersect-dist}, it therefore follows that for large enough $n$, 
\eqb
\op{dist}\left( v , \BB v ; M_{n^{K^2}} \right)
\geq  C^{-1} (\log n)^{-4} n^2 - 2 \geq n +1 .
\eqe
That is, $v\notin B_n(\BB v ; M_{n^{K^2}})$.
This gives the lemma statement with $K^2$ in place of $K$. 
\end{proof}

\begin{remark} \label{remark-dual}
All of the results in this and the previous two subsections remain true with $M$ replaced by its dual map $M_*$. One can even use the same coupling in Theorems~\ref{thm-map-coupling} and~\ref{thm-map-count} for both $M$ and $M_*$ simultaneously. This is because one can formulate the bijections used in this paper in terms of $M_*$ rather than $M$, and then use similar arguments to the ones in Section~\ref{sec-bijection} to transfer from Theorem~\ref{thm-sg-map-dist} below to estimates for $M_*$ instead of estimates for $M$. 
Similar considerations hold if, instead of $M_*$, we consider, e.g., the so-called \emph{radial quadrangulation} $\mcl Q = \mcl Q(M)$ whose vertex set is $\mcl V(M) \cup \mcl V(M_*)$ with two vertices joined by an edge if and only if they correspond to a face of $M$ and a vertex on the boundary of this face; or the dual map $\mcl Q_*$ of $\mcl Q$. 
\end{remark}

\section{Comparing distances via strong coupling}
\label{sec-dist-comparison}

In this section, we prove a variant of Theorem~\ref{thm-map-count} which will be the main technical input in the proofs of our main theorems. This theorem compares the mated-CRT map $\mcl G $ to a random planar map $\mcl H$ constructed from a general two-sided random walk in $\BB R^2$ with i.i.d.\ increments. 
Roughly speaking, this planar map is constructed via a simpler bijection than the ones corresponding to the four cases in Section~\ref{sec-main-results}, where the functions $\phi_n$ and $\psi_n$ of~\eqref{eqn-peano-functions} can be taken to be restrictions of globally defined inverse bijections $\phi : \mcl V(\mcl H) \rta \BB Z$ and $\psi : \BB Z\rta \mcl V(\mcl H)$, so the vertex set of $\mcl H$ can be identified with $\BB Z = \mcl V(\mcl G)$. 
As we will explain in Section~\ref{sec-bijection}, for several particular choices of $\mcl Z$ the graph $\mcl H$ is a close approximation of one of the random planar maps considered in Section~\ref{sec-main-results}. 

Let $\mcl Z = (\mcl L , \mcl R) : \BB Z\rta \BB R^2$ be a two-sided two-dimensional random walk normalized so that $\mcl Z_0 = 0$. We assume that the increments $\mcl Z_j - \mcl Z_{j-1}$ of $\mcl Z$ are i.i.d.\ with mean-zero; that there is a constant $c > 0$ such that the increment distribution satisfies
\eqb \label{eqn-moment-hypothesis}
\BB E\left[ \exp\left( (\mcl Z_j - \mcl Z_{j-1}) \cdot \xi \right) \right] < \infty ,\quad \forall \xi \in \BB R^2 \: \op{with} \: |\xi| \leq c ;
\eqe 
and that the walk is truly two-dimensional in the sense that the correlation $\rho := \op{Corr}(\mcl L_j-\mcl L_{j-1} , \mcl R_j - \mcl R_{j-1} )$ belongs to $(-1,1)$. 
We note that the hypothesis~\eqref{eqn-moment-hypothesis} is precisely the condition needed to apply the strong coupling result~\cite[Theorem 1.3]{zaitsev-kmt}. 

We define an infinite random planar map $\mcl H = \mcl H(\mcl Z)$ via the following discrete analogue of the formula~\eqref{eqn-bm-inf-adjacency} defining the mated-CRT map. 
The vertex set of $\mcl H$ is $\BB Z$, and for $i_1,i_2\in \BB Z$ with $i_1 < i_2$ we declare that $i_1$ and $i_2$ are connected by an edge in $\mcl H$ if and only if either
\eqb \label{eqn-walk-adjacency}
\mcl L_{i_1-1} \vee \mcl L_{i_2} < \min_{j\in [i_1 ,i_2-1]_{\BB Z}} \mcl L_j \quad \op{or} \quad
\mcl R_{i_1-1} \vee  \mcl R_{i_2} < \min_{j\in [i_1  ,i_2-1]_{\BB Z}} \mcl R_j \quad \op{or} \quad i_2-i_1 = 1 .
\eqe
(The arguments in this subsection are robust with respect to small modifications of~\eqref{eqn-walk-adjacency}; see Remark~\ref{remark-other-adjacency}).

As in Definition~\ref{def-sg-restrict}, for $n\in\BB N$ we define $\mcl H_{n}$ to be the subgraph of $\mcl H$ whose vertex set is $[-n,n]_{\BB Z}$, with two vertices connected by an edge in $\mcl H_{n}$ if and only if they are connected by an edge in $\mcl H$. 

Let $\gamma \in (0,2)$ be chosen so that $\rho = -\cos(\pi\gamma^2/4)$, so that the Brownian motion $Z$ of~\eqref{eqn-bm-cov} for this choice of $\gamma$ has correlation $\rho$.

\begin{thm} \label{thm-sg-map-dist}
For each $n\in\BB N$ and each $A > 0$, there is a $C   > 0$ depending on $A$ and the particular law of $\mcl Z$ and a coupling of $\mcl Z$ with the Brownian motion $Z$ (and thereby the mated-CRT map $\mcl G$) such that for each $n\in\BB N$, the following is true with probability at least $1-O_n(n^{-A})$. 
For each $i_1,i_2 \in  [-n , n]_{\BB Z}$ with $i_1 \sim i_2$ in $\mcl H$, there is a path $P_{i_1,i_2}^{\mcl G}$ from $i_1$ to $i_2$ in $\mcl G_{n}$ with $|P_{i_1,i_2}^{\mcl G}| \leq C (\log n)^3$; and each $i\in [-n , n]_{\BB Z}$ is hit by a total of at most $ C (\log n)^6 $ of the paths $P_{i_1,i_2}^{\mcl G}$. Moreover, the same is true with $\mcl H$ and $\mcl G$ interchanged. 
In particular, for each $i_1,i_2 \in [-n , n]_{\BB Z}$, 
\allb \label{eqn-sg-map-dist}
 C^{-1} (\log n)^{-3} \op{dist}\left(  i_1 , i_2 ; \mcl G_{n} \right)  
\leq \op{dist}\left( i_1,i_2 ; \mcl H_{n} \right)  
 \leq  C (\log n)^3 \op{dist}\left( i_1 , i_2; \mcl G_{n} \right) .
\alle
\end{thm}

The proof of Theorem~\ref{thm-sg-map-dist} proceeds by coupling $Z$ and $\mcl Z$ using the two-dimensional variant of the KMT coupling theorem~\cite[Theorem 1.3]{zaitsev-kmt} then comparing the adjacency conditions~\ref{eqn-bm-inf-adjacency} and~\eqref{eqn-walk-adjacency}. See Figure~\ref{fig-dist-comparison} for an illustration.

\begin{figure}[ht!]
\begin{center}
\includegraphics[scale=1.1]{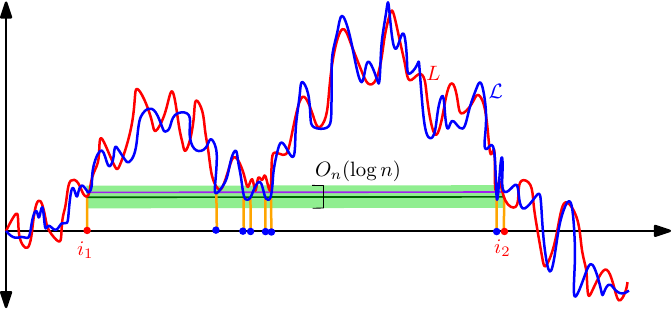} 
\caption[Illustration of the proof of Theorem~\ref{thm-sg-map-dist}]{\label{fig-dist-comparison} Illustration of the proof of Theorem~\ref{thm-sg-map-dist}. If $i_1 , i_2\in \BB Z$ are adjacent in $\mcl G $, then by~\eqref{eqn-bm-inf-adjacency} either we can draw a horizontal line segment (dark green) under the graph of $L$ connecting $(t_1 , L_{t_1})$ and $(t_2, L_{t_2})$ for some $t_1 \in [i_1 - 1, i_1]$ and $t_2 \in [i_2-1,i_2]$; or the same holds with $R$ in place of $L$. By the strong coupling result~\cite[Theorem 1.3]{zaitsev-kmt} we can couple $Z$ and $\mcl Z$ such that they differ by at most $O_n(\log n)$ on $[-n , n]_{\BB Z}$ with high probability. The purple segment is obtained by translating the dark green segment upward by $O_n(\log n)$. The blue and red dots correspond to times $j\in [i_1,i_2]_{\BB Z}$ for which $(j, \mcl L_j)$ lies below the purple line. For each such $j$, $(j,L_j)$ lies within distance $O_n(\log n)$ of the dark green segment (in particular, in the light green region), so Lemma~\ref{lem-bm-bad-cell} implies that there are at most $O_n( (\log n)^3)$ blue dots.  
By~\eqref{eqn-walk-adjacency}, successive blue or red dots are connected by edges of $\mcl H$, which gives a path in $\mcl H$ from $i_1$ to $i_2$ of length at most $O_n((\log n)^3)$. 
Similar considerations hold with the roles of $\mcl H$ and $\mcl G$ interchanged.
}
\end{center}
\end{figure}
 
The formula~\eqref{eqn-walk-adjacency} is unaffected if we rescale each of the coordinates $\mcl L$ and $\mcl R$ by a (possibly different) constant, so we can assume without loss of generality that these coordinates are normalized so that $\op{Var}(\mcl L_j  -\mcl L_{j-1}) = \op{Var}(\mcl R_j - \mcl R_{j-1} ) = 1$.  

By the multi-dimensional strong coupling theorem~\cite[Theorem 1.3]{zaitsev-kmt} (the higher dimensional analogue of~\cite{kmt}), there are constants $b_0,b_1>0$, depending only on the law of the increments of $\mcl Z$, and a coupling of $\mcl Z$ with $Z $ such that
\eqb \label{eqn-kmt-moment}
\BB E\left[ \exp\left( b_0 \max_{j \in [-n , n]_{\BB Z} } |\mcl Z_j - Z_j | \right) \right] \leq O_n(n^{b_1 }) .
\eqe 
By the Chebyshev inequality,~\eqref{eqn-kmt-moment} implies that there is a constant $C_0  > 0$, depending on $A$ and the law of the increments of $\mcl Z$, such that except on an event of probability $O_n(n^{-A})$ we have $\max_{j\in [-n,n]_{\BB Z}} |\mcl Z_j - Z_j| \leq C_0   \log n$. 
Henceforth fix such a coupling and constants $C_1 , C_2 \geq C_0\vee 1$ to be chosen later in a manner depending only on $\gamma$, $A$, and the law of the increments of $\mcl Z$. 

We will show that the conditions in the statement of Theorem~\ref{thm-sg-map-dist} are satisfied on an event $E^n$ depending on $Z$ and $\mcl Z$ (c.f.\ Lemma~\ref{lem-sg-map-path}).
In particular, we let $E^n = E^n(C_0,C_1,C_2)$ be the event that the following is true.  
\begin{enumerate}
\item $\max_{j\in [-n,n]_{\BB Z}} |\mcl Z_j - Z_j| \leq C_0   \log n$. \label{item-sg-map-kmt}
\item For each $i \in [-n,n]_{ \BB Z}$, we have $\sup_{s,t\in [i-1  , i]} |Z_t - Z_s| \leq   \log n$. \label{item-sg-map-gaussian}
\item For each pair of integers $(i_1,i_2)$ satisfying $ -n \leq i_1 < i_2 \leq n $ and
\eqb \label{eqn-sg-map-hyp}
\inf_{t\in [i_1,i_2]} (L_t - L_{i_1}) \geq - 6C_1  \log n \quad \op{and} \quad  |L_{i_2 }  - L_{i_1 } | \leq 6 C_1   \log n ,
\eqe
we have
\eqb \label{eqn-sg-map-set}
 \#\left\{ j  \in [i_1,i_2]_{ \BB Z} : \inf_{t\in [j-1 , j]} (L_t - L_{i_1}) \leq 7 C_1  \log n \right\} \leq C_1^3 (\log n)^3 .
\eqe  
and the same holds with $R$ in place of $L$. \label{item-sg-map-set}
\item For each $j \in [-n , n]_{\BB Z}$, the number of pairs $(i_1,i_2) \in [-n , n]_{\BB Z}^2$ with $i_1 < i_2$ for which~\eqref{eqn-sg-map-hyp} holds and $j$ belongs to the set in~\eqref{eqn-sg-map-set} is at most $C_2 (\log n)^6$; and the same holds with $R$ in place of $L$.  \label{item-sg-map-count}
\end{enumerate}
The condition~\eqref{eqn-sg-map-hyp} says that $[i_1,i_2]$ is in some sense ``close" to being an excursion interval for $L$, in the sense that the minimum of $L_t - L_{i_1}$ over this interval is not much less than zero and the difference $|L_{i_2} - L_{i_1}|$ is small. The condition~\eqref{eqn-sg-map-set} requires the $L$ does not get close to zero too many times during any such interval. See also Figure~\ref{fig-dist-comparison}. 
Before checking that the conditions in the theorem statement are satisfied on $E^n$, we show that $E^n$ occurs with high probability.

\begin{lem} \label{lem-sg-map-event}
If the constants $C_1 , C_2 \geq C_0 \vee 1$ are chosen sufficiently large, in a manner depending only on $\gamma$, $A$, and the law of $\mcl Z$, then 
\eqb \label{eqn-sg-map-event-prob}
\BB P\left[ E^n \right] \geq 1- O_n(n^{-A})  .
\eqe 
\end{lem}

For the proof of Lemma~\ref{lem-sg-map-event}, we will need the following elementary lemma about Brownian motion. 

\begin{lem} \label{lem-bm-bad-cell}
Let $B$ be a standard linear Brownian motion and for $r  \geq 1$, let $T_r := \inf\{t\geq 0 : B_t\leq -r\}$. For each $c>0$, there are constants $a_0 , a_1 > 0$ depending only on $c$ such that for $ r \geq 1$ and $s > 0$,  
\eqb \label{eqn-bm-bad-cell}
\BB P\left[  \#\left\{ j \in (0, T_r]_{ \BB Z} : \inf_{t\in [j-1,j]} B_t \leq c r \right\}  > s r^2  \right] \leq  a_0 e^{-a_1 s} .
\eqe 
\end{lem}
\begin{proof}
Let $\tau_0 = 0$ and for $k\in\BB N$ inductively let $\tau_k$ be the smallest $t\geq\tau_{k-1} + r^2$ for which $B_t \leq c r$. Also let $K$ be the smallest $k\in\BB N$ for which $\tau_k \geq T_r$, i.e., $K = \inf\left\{k\in\BB N : \inf_{t\in [\tau_{k-1} ,\tau_k]} B_t\leq -r \right\}$.
Then for $r\geq 1$, 
\eqbn
\#\left\{ j \in (0, T_r]_{ \BB Z} : \inf_{t\in [j-1 ,j]} B_t \leq r \right\} \leq   K  ( r^2 + 2) \leq 3 K r^2  .
\eqen
By the strong Markov property and Brownian scaling, there is a $p = p(c) \in (0,1)$ such that for each $k\in \BB N$,
\eqbn
\BB P\left[ \inf_{t \in [\tau_{k-1} ,\tau_k]} B_t \leq - r \,|\, B|_{[0,\tau_{k-1}]} \right] \geq p .
\eqen
Hence for $s  > 0$, $\BB P\left[ K \geq s \right] \leq (1-p)^{\lfloor s\rfloor}$, whence~\eqref{eqn-bm-bad-cell} holds.
\end{proof}

\begin{proof}[Proof of Lemma~\ref{lem-sg-map-event}]
Condition~\ref{item-sg-map-kmt} holds except on an event of probability $O_n(n^{-A})$ by our choice of coupling and condition~\ref{item-sg-map-gaussian} holds except on an event of sub-polynomial probability in $n$ by the Gaussian tail bound, the reflection principle, and a union bound over all $i\in [-n,n]_{ \BB Z}$.  

By Lemma~\ref{lem-bm-bad-cell} applied with $r = 6 C_1   \log n$, $c = 7/6$, and $s= \frac{1}{36} C_1 \log n$, there is a constant $a_1 = a_1(\gamma) > 0$ such that for each $i_1,i_2\in [-n,n]_{ \BB Z}$ with $i_1\leq i_2$ and each choice of $C_1 >0$, the probability of that~\eqref{eqn-sg-map-hyp} holds but~\eqref{eqn-sg-map-set} fails is at most $O_n(n^{-a_1 C_1 })$. If we choose $C_1$ sufficiently large, then by a union bound over all such pairs $(i_1,i_2)$, we find that condition~\ref{item-sg-map-set} holds except on an event of probability $O_n(n^{-A})$. 

Now we turn our attention to condition~\ref{item-sg-map-count}, which will also be obtained using Lemma~\ref{lem-bm-bad-cell}. If $(i_1,i_2) \in [-n , n]_{\BB Z}^2$ for which~\eqref{eqn-sg-map-hyp} holds and $j\in [i_1,i_2]_{\BB Z}$, then for some universal constant $K > 0$, 
\eqb \label{eqn-sg-map-count-adjacent}
\inf_{t\in [i_1 ,  j ]}  (L_t - L_{ i_1 } ) \geq -K C_1  \log n  \quad \op{and} \quad
\inf_{t\in [  j  , i_2 ]}  (L_t - L_{i_2} ) \geq -K C_1  \log n  .
\eqe
If also $j$ belongs to the set in~\eqref{eqn-sg-map-set} for this choice of $(i_1,i_2)$, then also $\inf_{t\in [j-1 , j]} (L_t - L_{i_1}) \leq 7 C_1  \log n $.
If these properties are satisfied and furthermore condition~\ref{item-sg-map-gaussian} in the definition of $E^n$ holds, then $L_j$ differs from $\inf_{t\in [j-1 , j]}  L_t$ by at most $ \log n$ so $|L_j - L_{i_1}| \vee |L_j - L_{i_2}|$ is bounded above by a universal constant times $C_1\log n$. Hence, for a possibly larger universal constant $K$, 
\allb \label{eqn-sg-map-count-adjacent'}
&\inf_{t\in [i_1 ,  j ]}  (L_t - L_j ) \geq -K C_1  \log n ,\quad 
\inf_{t\in [  j  , i_2 ]}  (L_t - L_j ) \geq -K C_1  \log n  , \notag \\
&\qquad \op{and} \quad 
(L_{i_1} - L_j) \vee (L_{i_2} - L_j) \leq     K C_1 \log n .
\alle
By Lemma~\ref{lem-bm-bad-cell}, applied to each of the Brownian motions $t\mapsto L_{t + j} - L_j$ and $t\mapsto L_{-t+j} - L_j$ and with $s $ equal to a large enough constant times $ \log n$ and $r =K C_1 \log n$, there is a constant $C_2 = C_2(A) >0$ such for each fixed $j \in \BB Z$, the probability that the number of pairs $(i_1,i_2) \in \BB Z^2$ for which~\eqref{eqn-sg-map-count-adjacent'} holds is larger than $C_2 (\log n)^6$ is at most $O_n(n^{-A-1})$. By a union bound over all $j \in [-n , n]_{\BB Z}$, we see that the probability that condition~\ref{item-sg-map-gaussian} holds but condition~\ref{item-sg-map-count} fails is at most $O_n(n^{-A})$. 
Combining the four preceding paragraphs shows that~\eqref{eqn-sg-map-event-prob} holds.
\end{proof}

\begin{lem} \label{lem-sg-map-path}
Let $C_1 ,C_2 \geq C_0 \vee 1$ be the constants from Lemma~\ref{lem-sg-map-event}. There is a constant $C \geq 1$, depending only on $C_1,C_2,$ and $C_0$ such that if the event $E^n$ defined just above Lemma~\ref{lem-sg-map-event} occurs, then the first part of the conclusion of Theorem~\ref{thm-sg-map-dist} is satisfied. That is, for each $i_1,i_2 \in  [-n , n]_{\BB Z}$ with $i_1 \sim i_2$ in $\mcl H$, there is a path $P_{i_1,i_2}^{\mcl G}$ from $i_1$ to $i_2$ in $\mcl G_{n}$ with $|P_{i_1,i_2}^{\mcl G}| \leq C (\log n)^3$; and each $i\in [-n , n]_{\BB Z}$ is contained in at most $ C (\log n)^6 $ of the paths $P_{i_1,i_2}^{\mcl G}$. Moreover, the same is true with $\mcl H$ and $\mcl G$ interchanged.
\end{lem}
\begin{proof}
Throughout the proof we assume that $E^n$ occurs. 

First consider a pair of vertices $i_1,i_2 \in [-n , n]_{\BB Z}$ with $i_1 < i_2$ and $i_1\sim i_2 $ in $\mcl H$. 
We will construct a path $P_{i_1,i_2}^{\mcl G}$ from $i_1$ to $i_2$ in $\mcl G_{n}$ with length at most $C_1^3 (\log n)^3$.

By~\eqref{eqn-walk-adjacency} either $|i_1-i_2| = 1$, $\mcl L_{i_1-1} \vee \mcl L_{i_2}  <  \min_{j \in [i_1  , i_2-1 ]_{\BB Z} } \mcl L_j $, or the same holds with $\mcl R$ in place of $\mcl L$. If $|i_1-i_2| = 1$ we take $P_{i_1,i_2}^{\mcl G}$ to be the length-1 path in $\mcl G$ from $i_1$ to $i_2$. 
We will construct $P_{i_1,i_2}^{\mcl G}$ in the case when $\mcl L_{i_1-1} \vee \mcl L_{i_2}  <  \min_{j \in [i_1  , i_2-1 ]_{\BB Z}} \mcl L_j $; the construction when this holds with $\mcl R$ in place of $\mcl L$ is similar. By conditions~\ref{item-sg-map-kmt} and~\ref{item-sg-map-gaussian} in the definition of $E^n$ and since $C_0 \leq C_1$,
\eqb \label{eqn-map-to-sg-condition}
\inf_{t\in [i_1 ,i_2 ] }  (L_t - L_{i_1 } ) \geq - 3 C_1  \log n \quad \op{and} \quad |L_{i_2 }  - L_{i_1 } | \leq 3 C_1   \log n.
\eqe  
By condition~\ref{item-sg-map-set} in the definition of $E^n$, the set
\eqb \label{eqn-use-bm-bad-cell'}
 \left\{ j \in [i_1  , i_2 ]_{  \BB Z} : \inf_{t\in [j- 1  ,j]} (L_t - L_{i_1}) \leq   3 C_1   \log n \right\}   
\eqe 
has cardinality at most $C_1^3 (\log n)^3$. 
By~\eqref{eqn-bm-inf-adjacency}, any two consecutive elements of the set~\eqref{eqn-use-bm-bad-cell'} are connected by an edge in $\mcl G $. Hence this set is connected in $\mcl G $. Since the set~\eqref{eqn-use-bm-bad-cell'} contains $  i_1$ and $i_2$, the vertices $i_1$ and $i_2$ can be connected by a path $P_{i_1,i_2}^{\mcl G}$ contained in this set which has length at most $C_1^3 (\log n)^3$. 

Using condition~\ref{item-sg-map-count} in the definition of $E^n$, we see that each $j\in [-n , n]_{\BB Z}$ is contained in at most $C_2 (\log n)^6$ of the sets~\eqref{eqn-use-bm-bad-cell'} for $i_1  < i_2$ such that $i_1\sim i_2$ in $\mcl H$, so each such $j$ is hit by at most $C_2 (\log n)^6$ of the paths $P_{i_1,i_2}^{\mcl G}$. 
This gives the statement of the lemma for adjacent vertices in $\mcl H$. 

We next prove the reverse relationship. Assume $i_1 , i_2 \in [-n , n]_{\BB Z}$ with $i_1 < i_2$ and $i_1 \sim i_2$ in $\mcl G$. We will construct a path $ P_{i_1,i_2}^{\mcl H}$ in $\mcl H_n$ from $i_1$ to $i_2$ with length at most $C_1^3 (\log n)^3$. The argument is similar to the one for adjacent vertices of $\mcl H$ given above.
 
By~\eqref{eqn-bm-inf-adjacency}, the condition that $i_1\sim i_2$ in $\mcl G$ implies that either
\eqb \label{eqn-use-bm-inf-adjacency}
\left( \inf_{t\in [  i_1-1  ,  i_1]} L_t \right) \vee    \left( \inf_{t\in [ i_2-1  ,  i_2]} L_t \right) \leq   \inf_{t\in [ i_1 ,  i_2-1 ]} L_t  
\eqe 
or the same holds with $R$ in place of $L$. Assume without loss of generality that we are in the former setting. 
By~\eqref{eqn-use-bm-inf-adjacency} and condition~\ref{item-sg-map-gaussian} in the definition of $E^n$, 
\eqb \label{eqn-sg-to-map-connected0}
 \inf_{t\in [i_1 , i_2]} (L_t -L_{i_1}) \geq -2 \log n \quad\op{and} \quad    |L_{i_1} - L_{i_2}| \leq   2 \log n .  
\eqe 
By this and condition~\ref{item-sg-map-kmt} in the definition of $E^n$, 
\eqb \label{eqn-sg-to-map-connected}
%L_{i_1 } \vee L_{i_2 } \leq \inf_{t\in [ i_1  ,   i_2 ]} L_t  + 2C  (\log n)^3. 
\inf_{t\in [i_1  , i_2 ]} (\mcl L_t - \mcl L_{i_1 } ) \geq - 4C_1   \log n \quad \op{and} \quad |\mcl L_{i_2}  - \mcl L_{i_1}| \leq 4C_1  \log n.
\eqe 
Moreover, by condition~\ref{item-sg-map-gaussian} in the definition of $E^n$ we have $|L_j - L_{j-1}| \leq \log n$ for each $j \in [-n, n]_{\BB Z}$
so by condition~\ref{item-sg-map-kmt}  in the definition of $E^n$, 
\eqb \label{eqn-sg-to-map-contain}
 \left\{ j \in [i_1,i_2]_{\BB Z} : \mcl L_j \wedge \mcl L_{j-1} - \mcl L_{i_1} \leq 4 C_1   \log n \right\} 
 \subset \left\{ j \in [i_1 ,i_2]_{\BB Z} : L_j - L_{i_1 } \leq 7 C_1  \log n \right\} .
\eqe 
By~\eqref{eqn-sg-to-map-connected0} and condition~\ref{item-sg-map-set} in the definition of $E^n$, the set on the right side of~\eqref{eqn-sg-to-map-contain} has cardinality at most $C_1^3 (\log n)^3$. 
By~\eqref{eqn-sg-to-map-connected} the set on the left side in~\eqref{eqn-sg-to-map-contain} contains $i_1$ and $i_2$. 
If $j_1$ and $j_2$ are consecutive elements of the left set in~\eqref{eqn-sg-to-map-contain} then either $j_2 = j_1+1$ or 
\eqbn
\mcl L_{j_1-1}  \leq 4 C_1   \log n  +  \mcl L_{i_1}  ,\quad
\mcl L_{j_2}  \leq 4 C_1   \log n  +  \mcl L_{i_1} ,\quad \op{and} \quad
\min_{j \in [j_1 , j_2-1]_{\BB Z}} \mcl L_j > 4 C_1   \log n  +  \mcl L_{i_1} ,
\eqen
so by~\eqref{eqn-walk-adjacency} $j_1$ and $j_2$ are connected by an edge in $\mcl H$. 
Hence $i_1$ and $i_2$ can be connected by a path $P_{i_1,i_2}^{\mcl H}$ in $\mcl H$ with length at most $C_1^3 (\log n)^3$ which is contained in the set on the left side of~\eqref{eqn-sg-to-map-contain}. 

Using condition~\ref{item-sg-map-count} in the definition of $E^n$, we see that each $j\in [-n , n]_{\BB Z}$ is contained in at most $C_2 (\log n)^6$ of the sets on the right side of~\eqref{eqn-sg-to-map-contain} for $i_1  < i_2$ such that~\eqref{eqn-sg-to-map-connected0} holds, so each such $j$ is contained in at most $C_2 (\log n)^6$ of the paths $P_{i_1,i_2}^{\mcl H}$. 

Consequently, the statement of the lemma holds with $C = C_2 \vee C_1^3$. 
\end{proof}

\begin{proof}[Proof of Theorem~\ref{thm-sg-map-dist}]
The first statement of the lemma (concerning the existence of paths satisfying the desired properties) is immediate from Lemmas~\ref{lem-sg-map-event} and~\ref{lem-sg-map-path}.
To deduce~\eqref{eqn-sg-map-dist} from this statement, suppose $i_1,i_2 \in [-n , n]_{\BB Z}$ and let $\wh P$ be a $\mcl H_{n}$-geodesic from $i_1$ to $i_2$. 
By concatenating paths of length at most $C (\log n)^3$ in $\mcl G_{n}$ between the vertices of $\mcl G_{n}$  corresponding to the vertices traversed by $\wh P$, we obtain a path of length at most $C (\log n)^3 |\wh P| $ in $\mcl G_{n}$ from $i_1$ to $i_2$, which gives the lower bound in~\eqref{eqn-sg-map-dist}. We similarly obtain the upper bound in~\eqref{eqn-sg-map-dist}. 
\end{proof}

\begin{remark} \label{remark-other-adjacency}
The arguments in this subsection still work almost verbatim if we slightly modify the adjacency condition~\eqref{eqn-walk-adjacency}, e.g., by inserting $\pm 1$ in various places. The reason for using~\eqref{eqn-walk-adjacency} is that this particular adjacency condition is closely connected to the mating-of-trees bijections for spanning tree-decorated maps and for site percolation on the UIPT (see Sections~\ref{sec-kappa8} and~\ref{sec:perc}). 
In the case of bipolar-oriented and Schnyder wood-decorated maps (Section~\ref{sec:bipolar}), we will apply Theorem~\ref{thm-sg-map-dist} in the case when $\mcl H$ is defined with~\eqref{eqn-walk-adjacency} replaced by \eqref{eq:walk-adjacency2} below.
\end{remark}

\section{Combinatorial arguments}
\label{sec-bijection}

In this section we review the bijective encodings of each of the random planar maps $M$ listed in Section~\ref{sec-main-results} by means of a certain two-dimensional random walk. 
We then compare $M$ to the random planar map $\mcl H$ constructed in Section~\ref{sec-dist-comparison} from this same two-dimensional random walk; and deduce the strong coupling result Theorem~\ref{thm-map-count} from this and Theorem~\ref{thm-sg-map-dist}. Our other main results will be easy consequences of Theorem~\ref{thm-map-count} and the results of~\cite{ghs-dist-exponent}.

Since the comparison between $M$ and $\mcl H$ relies on the fine geometric properties of the bijection, each of the cases needs to be treated separately. 
The reader may wish to read only one of the subsections of this section to get a general idea of the sort of arguments involved.  

We start in Section~\ref{sec-kappa8} by treating the simplest case---that of the infinite spanning tree-decorated map (case~\ref{item-kappa8}), which is encoded by a simple random walk on $\BB Z^2$ via the Mullin bijection. 
We review this encoding, show how the map $\mcl H$ from Section~\ref{sec-dist-comparison} arises from the bijection (Proposition~\ref{prop-kappa-8-tri}) then prove our main results in this case. In Section~\ref{sec:perc}, we treat the case of site percolation on the UIPT (case~\ref{item-kappa6}) by relating the bijection from~\cite{bernardi-dfs-bijection,bhs-site-perc} to the Mullin bijection. 

In Section~\ref{sec:bipolar}, we treat the case of the uniform infinite bipolar-oriented map (case~\ref{item-kappa12}) using the bijection of~\cite{kmsw-bipolar}. Along the way, we check carefully that this infinite map exists and is the local limit of the finite bipolar-oriented maps considered in~\cite{kmsw-bipolar} (unlike in the other cases, the existence of this local limit has not previously been established rigorously). We then treat the case of more general bipolar-oriented maps (case~\ref{item-kappa>8}) and deduce the case of Schnyder-wood decorated maps (case~\ref{item-kappa16}) 
from the result for bipolar-oriented maps plus a bijection relating Schnyder wood with a special type of bipolar-oriented maps~\cite{Felner}.
Section~\ref{sec:bipolar} can be read independently of Sections~\ref{sec-kappa8} and~\ref{sec:perc}. 

Several places in this section we will use the following notion of \emph{submap} of a planar map.
\begin{defn}
	A planar map $M'$ is a \emph{submap} of a planar map $M$ if, with $F_\infty\subset\cF(M')$ denoting the set of faces of $M'$ containing $\infty$, we have
	\eqbn
	\cV(M')\subset \cV(M),\qquad
	\cE(M')\subset \cE(M),\qquad
	\cF(M')\setminus F_\infty \subset \cF(M).
	\eqen
	The \emph{boundary} of $M'$ is the set of $v\in\cV(M)$ and $e\in\cE(M')$ which is incident to a face in $F_\infty$.
\end{defn}

\subsection{Spanning-tree decorated planar maps}
\label{sec-kappa8}

\subsubsection{Mullin bijection}
\label{sec-kappa8-bijection}

The first mating-of-trees bijection to be discovered is the Mullin bijection~\cite{mullin-maps}, which encodes a spanning-tree decorated map by a nearest-neighbor walk in $\BB Z^2$. This bijection is explained in more detail in~\cite{bernardi-maps}, and is also equivalent to the $p=0$ case of Sheffield's hamburger-cheeseburger bijection~\cite{shef-burger}. 
Here we will review the infinite-volume version of the Mullin bijection, which is also explained in~\cite{chen-fk,gms-burger-cone,blr-exponents} in the more general setting of the hamburger-cheeseburger bijection with arbitrary $p \in [0,1]$. See Figure~\ref{fig-mullin-bijection} for an illustration.

\begin{figure}[ht!]
\begin{center}
\includegraphics[scale=.8]{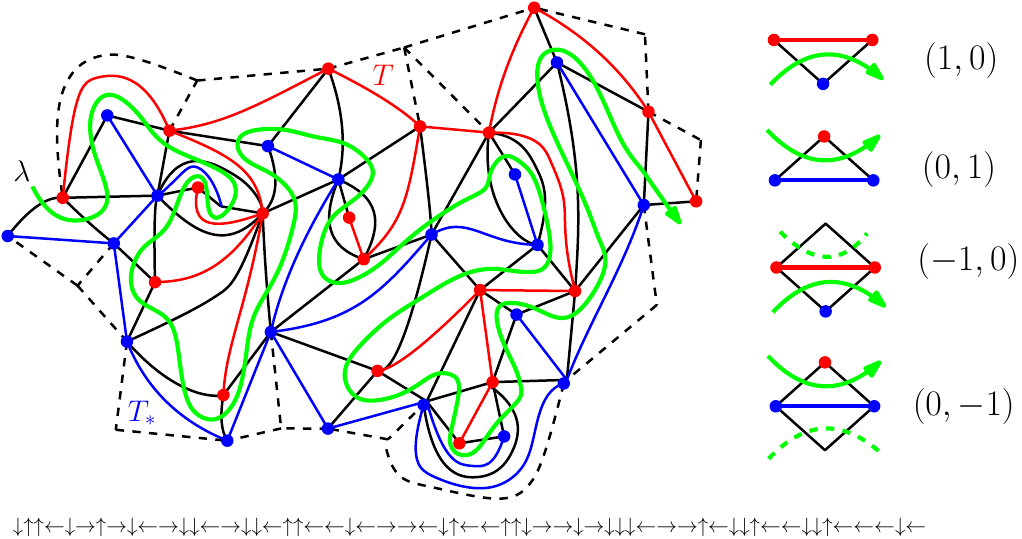} 
\caption[Infinite-volume Mullin bijection]{\label{fig-mullin-bijection} \textbf{Left:} The subset of the triangulation $\mcl Q \cup T\cup T_*$ consisting of the triangles in $\frk t([a,b]_{\BB Z})$, with edges of $\mcl Q$ (resp.\ $T$, $T_*$) shown in black (resp.\ red, blue). The green curve indicates the order in which the triangles are hit by $\frk t$. The edges of triangles of $\mcl T$ which do not belong to $\frk t([a,b]_{\BB Z})$, but which are adjacent to triangles in $\frk t([a,b]_{\BB Z})$, are shown as dotted lines. If $a=-n$ and $b=n$, then the vertices of $\Mn$ are the red vertices and the edges of $\Mn$ are the edges of $M$ which join these vertices (this includes the red edges in the figure plus some additional edges of $M$ which cross blue edges in the figure). \textbf{Right:} The correspondence between $M$ and $\mcl Z$. If $i\in \BB Z$ and $\frk t(i)$ has an edge in the red tree $ T$, then $\mcl Z_i - \mcl Z_{i-1}$ is equal to $(1,0)$ or $(-1,0)$ according to whether the other triangle which shares this same edge of $T$ is hit by $\frk t$ before or after time $i$. 
The other coordinate of $\mcl Z$ is defined symmetrically. 
\textbf{Bottom:} The steps of $\mcl Z$ corresponding to this segment of $\frk t$ if we assume that each of the exterior triangles (i.e., those with dotted edges) is hit by $\frk t$ before each of the non-dotted triangles. 
}
\end{center}
\end{figure}

Let $(M,e_0 , T)$ be the \emph{uniform infinite spanning-tree decorated map}, which is the Benjamini-Schramm~\cite{benjamini-schramm-topology} local limit of uniformly random triples consisting of a planar map with an oriented root edge and a distinguished spanning tree. 
This infinite-volume limit is shown to exist in~\cite{shef-burger,chen-fk}.  

Let $M_*$ be the dual map of $M$ and let $T_*$ be the dual spanning tree of $T$, so that $\mcl E(T_*)$ is the set of edges of $M_*$ which do not cross edges of $T$. 
Also let $\mcl Q = \mcl Q(M)$ be the \emph{radial quadrangulation}, whose vertex set is $\mcl V(M)\cup \mcl V(M_*)$, with two vertices of $\mcl Q$ connected by an edge if and only if they correspond to a face of $M$, i.e., a vertex of $M_*$, and a vertex of $M$ incident to that face (the number of edges is equal to the multiplicity of the vertex as a prime end on the boundary of the face). We declare that the root edge of $\mcl Q$ is the edge $\BB e_0 \in \mcl E(\mcl Q)$ whose primal endpoint coincides with the initial endpoint of $e_0$ and which is the first edge in clockwise order after $e_0$ with this property.

Each face of $\mcl Q$ is crossed diagonally by an edge of $M$ and an edge of $M_*$, exactly one of which belongs to $T\cup T_*$. Hence the graph $\mcl Q\cup T\cup T_*$ is a triangulation with the same vertex set as $\mcl Q$. Let $\mcl T$ be the planar dual of this triangulation, so that $\mcl T$ is the adjacency graph of triangles of $\mcl Q\cup T\cup T_*$, with two triangles considered adjacent if they share an edge. 
We declare that the root edge of $\mcl T$ is the edge of $\mcl T$ which crosses $\BB e_0$, oriented so that the primal (resp.\ dual) endpoint of $\BB e_0$ is to its left (resp.\ right). 

There is a unique path $\frk t : \BB Z \rta \mcl V(\mcl T)$ which hits each triangle (vertex) of $\mcl T$ exactly once; hits the initial and terminal points of the root edge of $\mcl T$ at times 0 and 1, respectively; and does not cross $T$ or $T_*$ in the sense that $\frk t(i)$ and $\frk t(i-1)$ share an edge belonging to $\mcl Q$ for each $i\in \BB Z$.
 
We define a walk $\mcl Z  = (\mcl L , \mcl R) : \BB Z \rta \BB R^2$ with increments in $\{(0,1) , (1,0) , (-1,0) , (0,-1)\}$ as follows. Define $\mcl Z_0 := 0$. Suppose $i\in\BB Z$ and the triangle $\frk t(i) \in \mcl V(\mcl T)$ has an edge in $T$ on its boundary (the other two boundary edges are in $\mcl Q$). 
There is one other triangle of $\mcl T$ with this same edge of $T$ on its boundary. If this triangle is hit by $\frk t$ before (resp.\ after) time $i$, we set $\mcl Z_i - \mcl Z_{i-1} = (1,0)$ (resp.\ $(-1,0)$). 
Symmetrically, if $\frk t(i)$ has a boundary edge in $T_*$ and the other triangle of $\mcl T$ sharing this boundary edge is hit before (resp.\ after) time $i$, we set $\mcl Z_i - \mcl Z_{i-1} = (0,1)$ (resp.\ $(0,-1)$). Then $\mcl Z$ has the law of a standard nearest-neighbor simple random walk in $\BB Z^2$. 

We next state and prove a proposition which will allow us to transfer from the results of Section~\ref{sec-dist-comparison} to results for the infinite spanning-tree weighted map. For the statement, we will use the following definition.

\begin{defn} \label{def-iso-mod-multiplicity}
Let $G$ and $G'$ be graphs and let $f : \mcl V(G) \rta\mcl V(G')$. We say that $f$ is a \emph{graph isomorphism modulo multiplicity} if $f$ is a bijection and two distinct vertices $v,w \in \mcl V(G)$ are connected by at least one edge in $G$ if and only if $f(v)$ and $f(w)$ are connected by at least one edge in $G'$ (i.e., adjacency, but not necessarily the number of edges between two vertices, is preserved). 
\end{defn}

Since we are primarily interested in graph distances, there is no difference for our purposes between a true graph isomorphism and a graph isomorphism modulo multiplicity. However, some of the graphs we consider naturally have multiple edges whereas the graphs $\mcl H$ from Section~\ref{sec-dist-comparison} by definition do not, so we often get isomorphisms modulo multiplicity instead of true isomorphisms. 

\begin{prop} \label{prop-kappa-8-tri}
The function $i\mapsto \frk t(i)$ is an isomorphism modulo multiplicity from the graph $\mcl H$ of Section~\ref{sec-dist-comparison} with the walk $\mcl Z$ as above and adjacency defined as in~\eqref{eqn-walk-adjacency}; to the graph of triangles $\mcl T$. 
\end{prop}
\begin{proof}
For $i\in\BB Z$, the triangle $\frk t(i)$ has two edges in $\mcl Q$ which are shared by $\frk t(i-1)$ and $\frk t(i+1)$, respectively, and one edge in either $T$ or $T_*$. 
Let $e_i$ be this third edge.
For each edge $e$ of $T$, there are precisely two values of $i\in\BB Z$ for which $e_i = e$, and the corresponding triangles $\frk t(i)$ are adjacent in $\mcl T$. 
The ordering of the edges of the tree $T$ obtained from $i\mapsto e_i$ by ignoring the values of $i$ for which $e_i \in T_*$ is the same as the contour (depth-first) ordering of $T$: indeed, this follows immediately from the definition of the path of triangles $i \mapsto \frk t(i)$.
The same holds with $T_*$ in place of $T$. By combining this with the above definitions of $\mcl L$ and $\mcl R$, we find that $\mcl L$ and $\mcl R$ coincide with the contour functions of the trees $T$ and $T_*$~\cite[Section 1]{legall-tree-survey} except that they have extra constant steps (which do not affect the trees). From this, it follows that for $i_1 < i_2$, the condition that $e_{i_1} = e_{i_2}$ is equivalent to the adjacency condition~\eqref{eqn-walk-adjacency} for $\mcl H$. 
Since the three neighbors of $\frk t(i)$ in $\mcl T$ are $\frk t(i-1)$, $\frk t(i+1)$, and the other triangle with $e_i$ on its boundary, we obtain the statement of the lemma.
\end{proof}

\subsubsection{The inverse construction and the sewing procedure}
\label{sec-mullin-inverse}  

In the above discussion, we only explained how to produce the walk $\mcl Z$ from the decorated map $(M,e_0,T)$. We can also go in the reverse direction and produce $(M,e_0,T)$ from $\mcl Z$ by first constructing $\mcl T$ from $\mcl Z$ via~\eqref{eqn-walk-adjacency} (see Proposition~\ref{prop-kappa-8-tri}) then constructing $(M,e_0,T)$ from $\mcl T$. Note that this uses the fact that we can tell which edges of the dual map $\mcl Q\cup T\cup T_*$ of $\mcl T$ belong to each of $\mcl Q$, $T$, and $T_*$ since edges of $\mcl Q$ correspond to edges of $\mcl T$ between consecutive vertices. We emphasize that the procedure for constructing $(M,e_0,T)$ works for any bi-infinite walk $\mcl Z = (\mcl L , \mcl R)$ in $\BB Z^2$ with nearest neighbor steps such that 
\eqb
\liminf_{n\rta -\infty} \mcl L_n = \liminf_{n\rta  \infty} \mcl L_n = \liminf_{n\rta -\infty} \mcl R_n = \liminf_{n\rta  \infty} \mcl R_n = -\infty ,
\eqe
not just a.s. We refer to~\cite[Section 4.1]{chen-fk} for more details on the inverse bijection. 

 The construction of $(M,e_0,T)$ from $\mcl Z$ is uniquely characterized by an inductive relationship between certain submaps of the triangulation $\mcl Q\cup T\cup T_*$, called the \emph{sewing procedure}, which we now describe (see Figure~\ref{fig-mullin-sewing} for an illustration).  
 For $i \in\BB Z$, let $\wh M_{-\infty,i}$ be the submap of $\mcl Q\cup T\cup T_*$ consisting of the vertices and edges on the boundaries of the triangles in $\frk t((-\infty,i]_{\BB Z})$. 
For $i\in\BB Z$, let $e^h_i$ (the ``head") be the edge shared by the triangles $\frk t(i-1)$ and $\frk t(i)$. 
Then $\wh M_{-\infty,i}$ is an infinite triangulation with boundary, and the external face has infinitely many edges lying to the left and right of $e^h_i$. 
We can recover $(\wh M_{-\infty,i+1} , e^h_{i+1})$ from $(\wh M_{-\infty,i } , e^h_{i })$ by attaching a certain triangle (corresponding to $\frk t(i+1)$) with one edge equal to $e^h_i$ and one edge equal to $e_{i+1}^h$. If $\mcl Z_{i+1} - \mcl Z_i = (1,0)$ (resp.\ $(0,1)$), this triangle has two edges in the external face of $\wh M_{-\infty,i}$ and $e^h_{i+1}$ is the right (resp.\ left) one of these two edges if we face outward toward the external face. If $\mcl Z_{i+1} - \mcl Z_i = (-1,0)$ (resp.\ $(0,-1)$), then one of the edges of the new triangle coincides with the edge of $\wh M_{-\infty,i}$ immediately to the left (resp.\ right) of $e^h_i$, and the remaining edge lies in the external face of $\wh M_{-\infty,i}$ and equals $e^h_{i+1}$. The edges of the tree $T$ (resp.\ $T_*$) are precisely the edges of $\mcl Q\cup T\cup T_*$ which are not crossed by $\frk t$ and which lie to the left (resp.\ right) of the path $\frk t$. It follows from the inverse bijection described above~\cite{shef-burger,chen-fk} that, given a walk $\cZ$ with i.i.d.\ increments there a.s.\ exists a unique tuple $(M,e_0,T)$ which can be constructed via the sewing procedure just described.

\begin{figure}[ht!]
\begin{center}
\includegraphics[scale=.8]{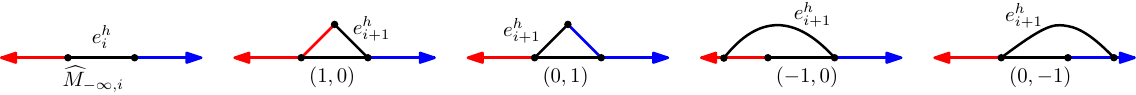} 
\caption[Sewing procedure for the Mullin bijection]{\label{fig-mullin-sewing} Illustration of the inverse Mullin bijection as described in Section~\ref{sec-mullin-inverse}. The leftmost panel shows the outer boundary of the map $\wh M_{-\infty,i}$, which is a subgraph of the triangulation $\mcl Q\cup T\cup T_*$ (with the external face lying above the line). The next four panels show how to construct $\wh M_{-\infty,i+1}$ from $\wh M_{-\infty,i}$ depending on the walk increment $\mcl Z_{i+1} -\mcl Z_i$. 
}
\end{center}
\end{figure}

\subsubsection{Proofs of main theorems in the case of spanning-tree decorated maps}
\label{sec-kappa8-proof}

In this subsection we prove our main results in case~\ref{item-kappa8}, when $\gamma = \sqrt 2$ and $(M,e_0)$ is an infinite-volume spanning-tree weighted planar map. Throughout, we define the quadrangulation $\mcl Q$, the graph of triangles $\mcl T$, and the bijective path $\frk t : \BB Z\rta \mcl V(\mcl T)$ as in Section~\ref{sec-kappa8-bijection}. 
We also let $Z = (L,R)$ be the Brownian motion as in~\eqref{eqn-bm-cov} for $\gamma = \sqrt 2$, so that $L$ and $R$ are independent; and we let $\mcl G $ be the associated mated-CRT map. 

Our first task is to define the maps $\Mn$ and the functions $\phi_n$ and $\psi_n$ from Section~\ref{sec-main-results} in this setting. 
In keeping with Remark~\ref{remark-subgraph}, we will first define functions $\phi : \mcl V(M) \rta \BB Z$ and $\psi : \BB Z\rta\mcl V(M)$. 
For $v\in\mcl V(M)$, we choose $\phi(v) \in \BB Z$ to be the integer with the smallest absolute value for which the triangle $\frk t(\phi(v)) \in \mcl V(\mcl T)$ has the vertex $v$ on its boundary (in the case of a tie, we choose $i$ to have a positive sign). 
For $i\in\BB Z$, we define $\psi(i) \in \mcl V(M)$ to be one of the vertices of $M$ on the boundary of the triangle $\frk t(i)$, chosen via some arbitrary deterministic convention in the case when there are two such vertices. We require $\psi(0) $ to be the initial endpoint of $e_0$ (i.e., the primal endpoint of $\BB e_0$), in order to ensure that the last condition in~\eqref{eqn-peano-functions} is satisfied.

For $n\in\BB N$, we define $\Mn$ the vertex set of $\Mn$ to be $\phi^{-1}([-n,n]_{\BB Z})$.  Equivalently, $\mcl V(\Mn)$ consists of all of the vertices of $M$ on the boundaries of the triangles in $\frk t([-n,n]_{\BB Z})$. The edge set of $\Mn$ is defined to be the set of edges $e$ of $M$ such that either $e$ lies on the boundary of a triangle in $\frk t([-n,n]_{\BB Z})$; or $e$ is contained in the union of two triangles of $\frk t([-n,n]_{\BB Z})$ (if the endpoints of $e$ are in $\mcl V(M_n)$, this latter condition is equivalent to requiring that the dual edge of $e$ lies on the boundary of a triangle in $\frk t([-n,n]_{\BB Z})$).
We define $\bdy \Mn$ to be the subgraph consisting of those vertices and edges of $\Mn$ which lie on the boundary of the unbounded connected component of $\Mn$ (viewed as a subgraph of $M$) and we define $\iota_n : \Mn \rta M$ to be the inclusion map. Set
\eqb
\phi_n := \phi|_{\mcl V(\Mn)} \quad \op{and} \quad \psi_n := \psi|_{[-n,n]_{\BB Z}  } .
\eqe 
Note that $\phi_n$ maps $\mcl V(\Mn)$ into $[-n,n]_{\BB Z}$ by definition and $\psi_n$ maps $[-n,n]_{\BB Z}$ into $\mcl V(\Mn)$ since the above definitions of $\phi$ and $\psi$ show that $|\phi(\psi(i))| \leq |i|$ for each $i\in\BB Z$. 
 
Let us now prove our main coupling theorems with this choice of $\phi_n$ and $\psi_n$. By Proposition~\ref{prop-kappa-8-tri}, if we let $\mcl Z$ be the simple random walk on $\BB Z^2$ which encodes $(M,e_0,T)$ via Mullin's bijection, then the map $\mcl H$ of Theorem~\ref{thm-sg-map-dist} with this choice of $\mcl Z$ is isomorphic to $\mcl T$ via $i\mapsto \frk t(i)$. 
The basic idea of the proof is to deduce Theorem~\ref{thm-map-count} from the analogous statements for $\mcl T$ from Theorem~\ref{thm-sg-map-dist} and straightforward geometric arguments. We first need the following degree bound, which is the reason why we have $(\log n)^4$ in Theorem~\ref{thm-map-count} as compared to the $(\log n)^3$ in Theorem~\ref{thm-sg-map-dist}.

\begin{lem} \label{lem-8-deg}
For $A>0$, there exists $C  = C (A) > 0$ such that for each $n\in\BB N$, it holds with probability at least $1-O_n(n^{-A})$ that each vertex of $\Mn$ lies on the boundary of at most $C \log n$ triangles of $\mcl T$.  
\end{lem}
\begin{proof}
It follows from~\cite[Lemma 6]{chen-fk} (see the proof of recurrence in~\cite[Section 4.2]{chen-fk}) and the stationary increments property of the walk $\mcl Z$ that there are universal constants $c_0,c_1>0$ such that for $i \in \BB Z$ and $k \in \BB N$, the probability that the triangle $\frk t(i)$ has a vertex of $M$ with degree (in $M$) at least $k$ on its boundary is at most $c_0 e^{-c_1 k}$. By a union bound over all $i\in [-n , n]_{\BB Z}$, we can find $C  = C(A) > 0$ as in the statement of the lemma such that 
\eqb  \label{eqn-8-deg0}
 \BB P\left[ \max_{v \in \mcl V(\Mn )} \op{deg}\left( v ; M \right) \leq  \frac12 C  \log n  \right] \geq   1 - O_n(n^{-A})  .
\eqe 
If $v\in \mcl V(M)$, then each quadrilateral of $\mcl Q$ with $v$ on its boundary is bisected by a unique edge of $M$ with $v $ as an endpoint. 
Consequently, the number of triangles of $\mcl T$ with $v $ on their boundaries is at most $2 \op{deg}(v  ; M)$.   
Combining this with~\eqref{eqn-8-deg0} concludes the proof.
\end{proof}

\begin{proof}[Proof of Theorems~\ref{thm-map-coupling} and~\ref{thm-map-count} in case~\ref{item-kappa8}]
In light of Lemma~\ref{lem-map-thms}, we only need to prove Theorem~\ref{thm-map-count}. 

We first define a high-probability event on which we will show that the conditions in the theorem statement are satisfied.
Let $C_0  >  0$ be the constant from Theorem~\ref{thm-sg-map-dist} for our given choice of $A$ and for $\mcl Z$ the simple random walk on $\BB Z^2$. Recall from Proposition~\ref{prop-kappa-8-tri} that the corresponding graph $\mcl H$ from that proposition is isomorphic to the graph of triangles $\mcl T$ via $i\mapsto \frk t(i)$. Hence Theorem~\ref{thm-sg-map-dist} shows that with probability at least $1-O_n(n^{-A})$, the following is true (where here, in obvious notation, we define $\mcl T_{n}$ to be the graph $ \frk t(\mcl H_{n})$).
\begin{enumerate}[label=(\Alph{enumi})]
\item For each $i_1,i_2 \in  [-n , n]_{\BB Z}$ with $\frk t(i_1) \sim \frk t(i_2)$ in $\mcl T$, there is a path $\wt P_{i_1,i_2}^{\mcl G}$ from $i_1$ to $i_2$ in $\mcl G _{n}$ with $| \wt P_{i_1,i_2}^{\mcl G}| \leq C_0 (\log n)^3$; and each $j \in [-n , n]_{\BB Z}$ is contained in at most $ C_0 (\log n)^6 $ of the paths $\wt P_{i_1,i_2}^{\mcl G}$. \label{item-8-path-G}
\item For each $i_1,i_2 \in  [-n , n]_{\BB Z}$ with $i_1 \sim i_2$ in $\mcl G$, there is a path $\wt P_{i_1,i_2}^{\mcl T}  $ from $\frk t(i_1)$ to $\frk t(i_2)$ in $\mcl T_{n}$ with $| \wt P_{i_1,i_2}^{\mcl T}| \leq C_0 (\log n)^3$; and each $j  \in [-n , n]_{\BB Z}$ is contained in at most $ C_0 (\log n)^6 $ of the paths $\wt P_{i_1,i_2}^{\mcl T}$. \label{item-8-path-T}
\end{enumerate} 
By Lemma~\ref{lem-8-deg}, there exists $C_1 =C_1(A)  >0$ such that with probability at least $1-O_n(n^{-A})$, 
\begin{enumerate}[label=(\Alph{enumi})]
 \setcounter{enumi}{2}
\item Each vertex of $\Mn$ lies on the boundary of at most $C_1 \log n$ triangles of $\mcl T$.  \label{item-8-deg}
\end{enumerate} 
Henceforth assume that conditions~\ref{item-8-path-G},~\ref{item-8-path-T}, and~\ref{item-8-deg}  above are all satisfied. 
We will check the conditions of the theorem statement in order. 
\medskip

\noindent\textit{Proof of condition~\ref{item-map-count-G}.} Suppose $v_1 , v_2 \in \mcl V(\Mn)$ with $v_1 \sim v_2$. We will construct a path $P_{v_1,v_2}^{\mcl G}$ in $\mcl G_{n}$ between $\phi(v_1)$ and $\phi(v_2)$. 

By the definition of $\mcl E(\Mn)$, either the edge of $\Mn$ joining $v_1$ and $v_2$ lies on the boundary of a triangle of $\mcl T_n$ or this edge is contained in the union of two triangles of $\mcl T_n$. In either case, we can find $i_1,i_2 \in [-n,n]_{\BB Z}$ such that $v_1$ is a vertex of $\frk t(i_1)$, $v_2$ is a vertex of $\frk t(i_2)$, and the triangles $\frk t(i_1)$ and $\frk t(i_2)$ are either identical or they share an edge (i.e., they are adjacent in $\mcl T$). 

By the definition of $\phi$, the triangle $\frk t(\phi(v_1))$ also has $v_1$ as a vertex. 
The set of triangles in $\frk t([-n , n]_{\BB Z})$ with $v_1$ as a vertex is connected in $\mcl T_n$.
Consequently, there exists a path in $\mcl T_{n}$ from $\frk t(i_1)$ to $\frk t(\phi(v_1))$ with length at most the total number of triangles of $\mcl T$ with $v$ on their boundaries. 
Similar considerations hold with $v_2$ in place of $v_1$.
 
By the preceding paragraph and condition~\ref{item-8-deg} above, there is a $d\leq 2 C_1 \log n$ and integers $j_0 , \dots , j_d  \in [-n , n]_{\BB Z}$ such that $j_0 = \phi(v_1)$, $j_d = \phi(v_2)$, and $\frk t(j_k) \sim \frk t(j_{k-1})$ in $\mcl T$ for each $k \in [1,d]_{\BB Z}$. 
For each such $k$, let $\wt P_{j_{k-1} , j_k}^{\mcl G}$ be the path from $j_{k-1}$ to $j_k$ afforded by condition~\ref{item-8-path-G} above and let $P_{v_1,v_2}^{\mcl G}$ be the concatenation of all $d$ of these paths. 
Then $ P_{v_1,v_2}^{\mcl G} $ is a path from $\phi(v_1)$ to $\phi(v_2)$ in $\mcl G_{n}$ with length at most $C_0 d (\log n)^3 \leq 2 C_0 C_1 (\log n)^4$.

Each of the triangles $\frk t(j_k)$ for $k\in [0,d]_{\BB Z}$ has either $v_1$ or $v_2$ on its boundary. 
Therefore, for each $j\in [-n ,n]_{\BB Z}$ the number of pairs $(v_1,v_2) \in [-n,n]_{\BB Z}\times [-n,n]_{\BB Z}$ for which $v_1\sim v_2$ in $M$ and $j$ is one of the corresponding $j_k$'s is at most the sum of the degrees of the (at most two) vertices of $M$ on the boundary of $\frk t(j)$, which is at most $2 C_1 \log n$ by condition~\ref{item-8-deg} above. 
 
Consequently, each of the paths $\wt P_{j,j'}^{\mcl G}$ for $j,j' \in [-n,n]_{\BB Z}$ with $\frk t(j) \sim \frk t(j')$ in $\mcl T_n$ is a sub-path of at most $2 C_0 \log n$ of the paths $P_{v_1,v_2}^{\mcl G}$ for $v_1,v_2 \in \mcl V(M)$. 
Since the total number of the $\wt P_{j,j'}^{\mcl G}$'s which hit each specified vertex of $\mcl G_{n}$ is at most $C_0 (\log n)^6$, the total number of the $P_{v_1,v_2}^{\mcl G}$'s which hit each such vertex is at most $2C_0 C_1 (\log n)^7$. This gives condition~\ref{item-map-count-G} in the theorem statement with $C = 2C_0 C_1$. 
\medskip

\noindent\textit{Proof of condition~\ref{item-map-count-M}.} Suppose $i_1,i_2 \in [-n  ,n]_{\BB Z}$ with $i_1\sim i_2$ in $\mcl G$. 
Let $\wt P  = \wt P_{i_1,i_2}^{\mcl T}$ be the path from $\frk t(i_1)$ to $\frk t(i_2)$ in $\frk t([-n , n]_{\BB Z})$ as in condition~\ref{item-8-path-T} above.
We will construct a path $P_{i_1,i_2}^M$ in $\Mn$ between the vertices $\psi(i_1)$ and $\psi(i_2)$ of $M$. 

Set $v_0 = \psi(i_1)$, let $v_{|\wt P|} = \psi(i_2)$, and for $k\in [1,|\wt P|-1]_{\BB Z}$ let $v_k$ be a vertex of the triangle $\wt P(k)$ which belongs to $\mcl V(\Mn)$.
For $k\in [1,|\wt P|]_{\BB Z}$, the triangles $\wt P(k-1)$ and $\wt P(k)$ share an edge. This implies that the vertices $v_{k-1}$ and $v_k$ can be connected by a path of length at most 2 in $\Mn$ which hits only vertices of $\wt P(k-1)$ or $\wt P(k)$ (see Figure~\ref{fig-triangle-types}).
If we let $P_{i_1,i_2}^M$ be the concatenation of these $|\wt P|$ paths, then $P_{i_1,i_2}^M$ is a path from $v_1$ to $v_2$ in $\Mn$ with length at most $2|\wt P| \leq 2C_0 (\log n)^3$.

It remains to bound the number of paths $P_{i_1,i_2}^M$ for $i_1,i_2 \in [-n ,n]_{\BB Z}$ with $i_1\sim i_2$ in $\mcl G$ which hit a fixed vertex $v$ of $\Mn$. Indeed, if $P_{i_1,i_2}^M$ hits $v$ then the above construction shows that $\wt P_{i_1,i_2}^{\mcl T}$ must hit one of the triangles of $\mcl T$ which has $v$ as a vertex. By  condition~\ref{item-8-deg}, the number of such triangles is at most $2 C_1 \log n$. 
By condition~\ref{item-8-path-T} above, the total number of the paths $P_{i_1,i_2}^M$ which hit each such triangle is at most $C_0 (\log n)^6$. 
Therefore, $v$ is hit by at most $2C_0 C_1 (\log n)^7$ of the $P_{i_1,i_2}^M$'s, which gives condition~\ref{item-map-count-M} in the theorem statement with $C = 2C_0 C_1$.
\medskip

\noindent\textit{Proof of condition~\ref{item-map-count-close}.} Since the vertices $v$ and $\psi(\phi(v))$ lie on the boundary of the same triangle of $\mcl T$ by definition, we obviously have $\op{dist}\left(\psi(\phi(v)) , v ; \Mn \right) \leq 1$. On the other hand, the triangles $\frk t(i)$ and $\frk t(\phi(\psi(i)))$ share a common vertex of $M$, so by condition~\ref{item-8-deg} the $\mcl T_{n}$-graph distance between these triangles is at most
 $C_1 \log n$. By considering a path in $\mcl T_{n}$ between these two triangles and applying condition~\ref{item-8-path-G} to the successive pairs of triangles which it hits, we get $ \op{dist}\left(\phi(\psi(i)) , i ; \mcl G_{n} \right) \leq C_0 C_1 (\log n)^4 $.  
\end{proof}
%\footnote{Not all of the triangles which share this common vertex $v$ of $M$ necessarily belong to $\mcl T_n$. However, there are two distinct simple paths in $\mcl T$ between $\frk t(i)$ and $\frk t(\phi(\psi(i)))$ consisting of triangles which have $v$ as a vertex, one going in the counterclockwise direction around $v$ and the other going in the clockwise direction around $v$. It is easily seen from the definition of $\mcl T_n$ that at least one of these two paths is contained in $\mcl T_n$.}

\begin{figure}[ht!]
\begin{center}
\includegraphics[scale=.8]{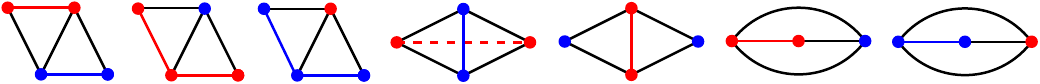} 
\caption[Adjacent triangles in the Mullin bijection]{\label{fig-triangle-types} Illustration of the seven different situations that can arise when two triangles of $\mcl T$ share a vertex. Vertices and edges of $M$ (resp.\ $M_*$) are shown in red (resp.\ blue). Edges of $\mcl Q$ are shown in black. The dotted red edge is an edge of $M$ which does not belong to the tree $T$. Note that in each case, there is a path of length at most 2 in $M$ from any given red vertex of the first triangle to any given red vertex of the second triangle. Note that if the two triangles are in $\mcl T_n$, then by definition the red vertices and edges in the figure are in $\Mn$.}
\end{center}
\end{figure}

\begin{figure}[ht!]
\begin{center}
\includegraphics[scale=.8]{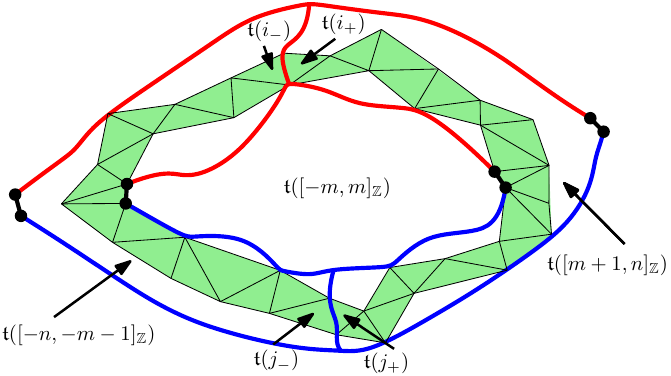} 
\caption{\label{fig-mullin-circuit} Illustration of the proof of Lemma~\ref{lem-ball-layers} in case~\ref{item-kappa8}. The triangles in the path $P$ in $\mcl T$ which separates $\frk t([-m,m]_{\BB Z})$ from $\mcl T\setminus \frk t([-n,n]_{\BB Z})$ are shown in light green. The edges on the boundary of the union of the triangles in each of $\frk t([-m,m]_{\BB Z})$, $\frk t([-n-1,-m]_{\BB Z})$, and $\frk t([m+1,n]_{\BB Z})$ are shown in red, blue, or black according to whether they below to $T$, $T_*$, or $Q$. Other edgeds are not shown. 
}
\end{center}
\end{figure}

\begin{proof}[Proof of Lemma~\ref{lem-ball-layers} in case~\ref{item-kappa8}]
See Figure~\ref{fig-mullin-circuit} for an illustration of the proof.
Recall that the adjacency graph of triangles $\mcl T$ is isomorphic to the graph $\mcl H$ of~\eqref{eqn-walk-adjacency} via $i\mapsto \frk t(i)$. 
Furthermore, the condition on the minimum of $\mcl L$ (resp.\ $\mcl R$) in~\eqref{eqn-walk-adjacency} holds if and only if $\frk t(i_1)$ and $\frk t(i_2)$ share an edge of $T$ (resp.\ $T_*$). 
From this, we infer that if the random walk condition~\eqref{eqn-walk-inf-condition0} and its analog for $\mcl R$ hold, then there exists $i_+ , j_+ \in [m+1,n]_{\BB Z}$ and $i_- , j_- \in [-n , -m-1]_{\BB Z}$ such that $\frk t(i_+) , \frk t(i_-)$ share an edge of $T$ and $\frk t(j_+) , \frk t(j_-)$ share an edge of $T_*$.  
Since consecutive triangles of $\mcl T$ share a side, each of $\frk t([m+1,n]_{\BB Z})$  and $\frk t([-n,-m-1]_{\BB Z})$ is connected in $\mcl T$. 
Hence we can find a path of triangles $P_+$ in $\frk t([m+1,n]_{\BB Z})$ from $\frk t(i_+)$ to $\frk t(j_+)$ and  a path of triangles $P_-$ in $\frk t([-n,-m-1]_{\BB Z})$ from $\frk t(j_-)$ to $\frk t(i_-)$. 
The concatenation $P$ of the paths $P_+$ and $P_-$ is a cycle in $\frk t([-n,n]_{\BB Z})$. 

We claim that $P$ necessarily disconnects $\frk t([-m,m]_{\BB Z})$ from $\mcl T\setminus \frk t([-n,n]_{\BB Z})$. Indeed, consider the set of edges of triangles in $\mcl T$ which lie on the boundary of both a triangle of $\frk t([m+1,n]_{\BB Z})$ and a triangle of $ \frk t([-n,-m-1]_{\BB Z})$. This set has at most two connected components, one of which is a subset of $T$ and the other of which is a subset of $T_*$. By definition, $P$ crosses an edge of each of these two connected components (in particular, they are both non-empty). It follows that the union of the triangles in $\frk t([-n,n]_{\BB Z}) \setminus \frk t([-m,m]_{\BB Z})$ has the topology of a Euclidean annulus whose inner and outer boundaries are disconnected by $P$. Our claim thus follows. 

The vertex set $\mcl V(M_m)$ consists of all vertices on the boundaries of triangles in $\frk t([-m,m]_{\BB Z})$, so the preceding paragraph implies that no vertex of $M_m$ can lie on the boundary of a triangle which is not in $\frk t([-n,n]_{\BB Z})$ (otherwise, such a triangle would have to cross $P$). If $v\in \bdy M_n$ then $v$ lies on the boundary of a triangle which is not in $\frk t([-n,n]_{\BB Z})$, so $\phi(v)$ does not lie on the boundary of a triangle in $\frk t([-m,m]_{\BB Z})$.  
Since $\phi(v) \in \BB Z$ is chosen so that $v$ lies on the boundary of $\frk t(\phi(v))$, we therefore have $\phi(v) \notin [-m,m]_{\BB Z}$. 
Hence $\phi(\bdy M_n) \cap [-m,m]_{\BB Z} = \emptyset$. 
\end{proof}

\begin{comment}
Recall that the adjacency graph of triangles $\mcl T$ is isomorphic to the graph $\mcl H$ of~\eqref{eqn-walk-adjacency} via $i\mapsto \frk t(i)$. 
From this, we infer that if the random walk condition~\eqref{eqn-walk-inf-condition0} and its analog for $\mcl R$ hold, then no triangle in $\frk t([-m,m]_{\BB Z})$ can share a side with a triangle in $\mcl T\setminus \frk t([-n,n]_{\BB Z})$.  
Since consecutive triangles share a side, each of $\frk t([-n,-m-1]_{\BB Z})$ and $\frk t([m+1,n]_{\BB Z})$ is connected in $\mcl T$, from which it follows that no triangle in $\frk t([-m,m]_{\BB Z})$ can share a vertex with a triangle in $\mcl T\setminus \frk t([-n,n]_{\BB Z})$, either.  
The vertex set $\mcl V(M_m)$ consists of all vertices on the boundaries of triangles in $\frk t([-m,m]_{\BB Z})$, so the preceding sentence implies that no vertex of $M_m$ can lie on the boundary of a triangle which is not in $\frk t([-n,n]_{\BB Z})$. If $v\in \bdy M_n$ then $v$ lies on the boundary of a triangle which is not in $\frk t([-n,n]_{\BB Z})$, so $\phi(v)$ does not lie on the boundary of a triangle in $\frk t([-m,m]_{\BB Z})$.  
Since $\phi(v) \in \BB Z$ is chosen so that $v$ lies on the boundary of $\frk t(\phi(v))$, we therefore have $\phi(v) \notin [-m,m]_{\BB Z}$. 
Hence $\phi(\bdy M_n) \cap [-m,m]_{\BB Z} = \emptyset$. 
\end{comment}

\begin{proof}[Proof of Theorems~\ref{thm-map-ball} and~\ref{thm-map-dist} in case~\ref{item-kappa8}]
Theorem~\ref{thm-map-ball} in the case of the spanning-tree weighted map is an immediate consequence of Theorem~\ref{thm-map-coupling}, Lemma~\ref{lem-ball-iso}, and~\cite[Theorem~1.10]{ghs-dist-exponent}. Theorem~\ref{thm-map-dist} follows from Theorem~\ref{thm-map-coupling} and~\cite[Theorem 1.15]{ghs-dist-exponent}.
\end{proof}

\subsection{Site percolation on the UIPT}
\label{sec:perc}

\subsubsection{Bijection for site percolation on the UIPT}
\label{sec-perc-bijection}
In this subsection we review the infinite-volume version of the bijection between site-percolated triangulations of type II (no self-loops, but multiple edges allowed) and \emph{Kreweras walks}---those with steps in $\{(1,0) , (0,1) , (-1,-1)\}$---which was first described in \cite{bhs-site-perc}. This bijection is based on an earlier bijection by Bernardi \cite{bernardi-dfs-bijection} between such walks and trivalent maps decorated by a depth-first search tree. 
We emphasize that the only part of~\cite{bhs-site-perc} needed here is the definition of the bijection in~\cite[Section 2]{bhs-site-perc}.

Let $(M,e_0)$ be an instance of the uniform infinite planar triangulation (UIPT) of type II, rooted at a directed edge $e_0$. Let $P$ be an instance of critical site percolation on the map, i.e., $P$ associates each $v\in \mcl \cV(M)$ with either red or blue, uniformly and independently at random. In \cite[Section 2.5]{bhs-site-perc} (based on \cite{bernardi-dfs-bijection}) it was proved that $(M,e_0,P)$ can a.s.\ be encoded by a bi-infinite walk with increments $a=(1,0)$, $b=(0,1)$, $c=(-1,-1)$\footnote{To be precise, the correspondence between walks and triples $(M,e_0,P)$ in \cite{bhs-site-perc} is for the case where the edge $e_0$ is undirected. The direction of $e_0$ does not play an important role in this section, and we may for example assume that the direction of $e_0$ is determined by an additional binary random variable.}. Conversely, a bi-infinite walk $\cZ=(\cZ_k)_{k\in\Z}$ with $\cZ_0=0$ and steps $a,b,c$ chosen uniformly and independently at random, a.s.\ encodes an instance of $(M,e_0,P)$. Throughout this section we let 
$(\Delta \cZ_k)_{k\in\N}$ denote the steps of $\cZ$, i.e., 
\eqbn
\Delta \cZ_k:=\cZ_k-\cZ_{k-1}\in\{a,b,c \} .
\eqen

As in the case of the Mullin bijection (see Section~\ref{sec-mullin-inverse}), the construction of  $(M,e_0,P)$ from the walk $\cZ$ is uniquely characterized by a sewing procedure, which we now describe. Each $k\in\Z$ corresponds to a map $M_{-\infty,k}$ with a percolation configuration $P_{-\infty,k}$ on the inner vertices and a marked edge $e_k^h$. Again the sewing procedure describes how to obtain $(M_{-\infty,k},P_{-\infty,k},e_k^h)$ from $(M_{-\infty,k-1},P_{-\infty,k-1},e_{k-1}^h)$ by observing the step $\Delta \cZ_k$ of the walk. The map $M$ rooted at $e_0$ is the local limit of $M_{-\infty,k}$ rooted at $e_0$ as $k$ goes to $\infty$ in the Benjamin-Schramm sense. 
%We will see that each integer $k\in\Z$ corresponds to a map $M_{-\infty,k}$ with a percolation configuration and a marked edge $e_k^h$, 
We will see that each $k$ can be associated with a unique element $\wt\lambda(k)\in\cV(M)\cup\cF(M)$. More precisely, we have $\wt\lambda(k)\in\cV(M)$ if $\Delta \cZ_k=c$ and $\wt\lambda(k)\in\cF(M)$ if $\Delta \cZ_k\in\{a,b \}$. The root edge $e_0$ of the map is $e_{-1}^h$. For any $k\in\Z$ the map $M_{-\infty,k}$ has an infinite left frontier and an infinite right frontier. Note that the sewing procedure is only an inductive procedure, but, as for the Mullin bijection, it can be shown that given a walk $\cZ$ with i.i.d.\ increments there is a.s.\ a unique collection of triples $(M_{-\infty,k},P_{-\infty,k},e_k^h)$ such that $(M_{-\infty,k},P_{-\infty,k},e_k^h)$ is obtained from $(M_{-\infty,k-1},P_{-\infty,k},e_{k-1}^h)$ via the sewing procedure.  We refer to \cite[Section 2.5]{bhs-site-perc} for a more complete description, and to Figure \ref{fig1} for an illustration.

\begin{figure}
	\centering
	\includegraphics[scale=1.0]{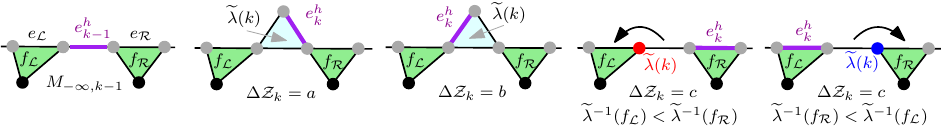}
	\caption[The sewing procedure for site percolation on the UIPT]{Illustration of the sewing procedure for the measure-preserving map between percolation-decorated triangulations and Kreweras walks (i.e., walks with steps $a,b,c$).
	The black arrows indicate that we glue $e^h_{k-1}$ to an adjacent edge, i.e., the two edges are identified. See the main text for the precise description. Vertices are colored gray if their color is not yet determined, while they are colored black if their color has been determined, so they are either blue or red, but the color cannot be determined from the figure.
	}
	\label{fig1}
\end{figure}

Given $M_{-\infty,k-1}$ and $e_{k-1}^h$, we will now explain how to construct $M_{-\infty,k}$ via the sewing procedure. See Figure \ref{fig1}. If $\Delta \cZ_{k}=a$ (resp.\ $\Delta \cZ_{k}=b$), we add a triangle to the edge  $e^h_{k-1}$. We let the new head $e^h_{k}$ be the right (resp.\ left) side of the triangle. We define $\wt\lambda(k)\in\cF(M)$ to be the triangle we added. Next consider the case $\Delta \cZ_{k}=c$. For an edge $e$ on the frontier of the map we say that a face $f$ is \emph{behind} $e$ if it is the unique face of the map for which $e$ is on its boundary. Let $e_\cL$ (resp.\ $e_\cR$) be the edge on the frontier of $M_{-\infty,k-1}$ immediately to the left (resp.\ right) of $e_{k-1}^h$, and let $f_\cL$ (resp.\ $f_\cR$) be the face behind $e_\cL$ (resp.\ $e_\cR$). If $\wt\lambda^{-1}(f_\cL)<\wt\lambda^{-1}(f_\cR)$ (resp.\ $\wt\lambda^{-1}(f_\cL)>\wt\lambda^{-1}(f_\cR)$) define $e^h_k=e_\cR$ (resp.\ $e^h_k=e_\cL$). Then identify $e^h_{k-1}$ with $e_\cL$ (resp.\ $e_\cR$), i.e., glue the two edges together. The common vertex of $e_\cL$ (resp.\ $e_\cR$) and $e^h_{k-1}$ is colored red (resp.\ blue), and we define $\wt\lambda(k)\in\cV(M)$ to be this vertex. We see from this description that the function $\cL$ (resp.\ $\cR$) represents the net change in the length of the left (resp.\ right) frontier of the map, relative to the length at time 0.

We remark that $M_{-\infty,k}$ is not a submap of $M$, since for every step $\Delta \cZ_k=c$ in the sewing procedure we identify two edges and two vertices, and vertices and edges which are identified in the sewing procedure at times $>k$ are \emph{not} identified in $M_{-\infty,k}$. However, as remarked above, there are natural maps from the set of vertices (resp.\ inner faces, edges) of $M_{-\infty,k}$ to the set of vertices (resp.\ faces, edges) of $M_{-\infty,k+1}$ and $M$. For faces this map is injective, while this is not the case for vertices and edges. We will often identify vertices (resp.\ faces, edges) of the various maps when this identification is well-defined. We remark that for $k\in\BB Z$, the edge $\lambda(k)$ of Remark~\ref{remark-edge-map} is the edge of $M$ corresponding\footnote{Note that the function $\lambda:\Z\to \mcl E(M)$ is \emph{not} the same as the function $\eta_{\op{e}}:\Z\to \mcl E(M)$ considered in \cite{bhs-site-perc}. For all $k\in\Z$, we have $\lambda(k-1)=\eta_{\op{e}}(k)$.} to $e_k^h$.

Given $n\in\N$ we now define the map $\Mn$ introduced in Section \ref{sec-main-results}. We will let $\Mn$ be a particular submap of $M_{-\infty,n}$. Let the vertex set $\cV(\Mn)\subset\cV( M_{-\infty,n})$ of $\Mn$ be the set of all vertices which are either an endpoint of some edge $e_{k-1}^h$ or $e_k^h$ or on the boundary of some face $\wt\lambda(k)\in\cF(M_{-\infty,n})$ for $k\in[-n,n]_{\Z}$. Let $\cE_1\subset \cE(M_{-\infty,n})$ be the set of edges $e=\{v_1,v_2\}$ which are on the boundary of some face $\wt\lambda(k)\in\cF(M_{-\infty,n})$. Let $\cE_2=\{e_{k-1}^h\,:\,k\in[-n,n]_{\Z},\Delta \cZ_k=c \}\subset \cE( M_{-\infty,n})$ be the set of edges which were glued to another edge in the sewing procedure. Define $\cE(\Mn)=\cE_1\cup \cE_2$. We let $\iota_n : \Mn \rta M$ be the function which sends each vertex, face, and edge of $\Mn$ to the corresponding vertex, face, or edge of $M$.
	We will define the mappings $\phi_n$ and $\psi_n$ of~\eqref{eqn-peano-functions} in~\eqref{eqn-percolation-phi-psi} just below. 

\begin{figure}
	\centering
	\includegraphics[scale=1]{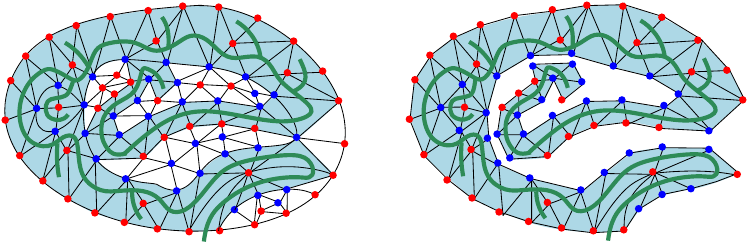}
	\caption[The planar maps $\Mn$ in the case of the UIPT]{The figure shows $\Mn$ embedded into $M$ (left) and $\Mn$ (right) in light blue. Notice that vertices which are adjacent in $M$ may correspond to vertices of $\Mn$ which are far apart. The green path is part of the depth-first search (DFS) tree of the dual map $M^*$ as defined in \cite[Section 4.1]{bhs-site-perc} (this DFS tree is not used in the present paper). The DFS tree gives an ordering of the faces of $M$ that is consistent with the ordering given by $\wt\lambda$.}
	\label{fig4}
\end{figure}

\subsubsection{Proofs of main theorems in the UIPT case}
\label{sec-percolation-proof}
To prove the theorems from Section \ref{sec-intro} in case~\ref{item-kappa6}, we will make small local changes in the map $M$ and the walk $\mcl Z$ defined above and observe that the resulting map and walk are related by Mullin's bijection as described in Section \ref{sec-kappa8}. In particular, as we will see, if we replace each of the $c$ steps of the walk $\cZ$ by a $-a$ step followed by a $-b$ step and each of the corresponding ``gluing steps" in the sewing procedure by a step which adds a certain subgraph consisting of two triangles to the map, then the resulting modified walk and planar map are related via Mullin's bijection (Section~\ref{sec-kappa8-bijection}). This will allow us to reduce the proof of Theorem~\ref{thm-map-count} in the UIPT case to the Mullin bijection case, which we have already treated. It is also possible to treat case~\ref{item-kappa6} directly, without reference to the Mullin bijection, but the argument given here is shorter and simpler. 

We first define a modified walk $\wh\cZ$ with steps in $\{-a,-b,a,b\}$ by replacing each step $c$ in $\cZ$ by one step $-a$ followed by one step $-b$. More precisely, $\wh\cZ$ and the maps $\alpha_n$ and $\wh\alpha_n$ describing the associated time changes are defined as follows. First define $\alpha:\Z\to \Z$ and $\wh\alpha:\Z\to\Z$ by
\eqbn
\begin{split}
	\wh\alpha(k)&=
	\begin{cases}
		k+\#\{ j\in\{0,\dots,k-1\}\,:\,\Delta \cZ_j=c  \} \quad&\text{for}\,\,k\geq 0\\
		k-\#\{ j\in\{k,\dots,-1\}\,:\,\Delta \cZ_j=c  \} \quad&\text{for}\,\,k<0,
	\end{cases}\\
	\alpha(k)&=\sup\{ j\in\Z\,:\,\wh\alpha(j)\leq k \}.
\end{split}
\eqen
See Figure \ref{fig3} for an illustration. Let $\wh\alpha_n=\wh\alpha|_{[-n,n]_\Z}$, and let $n_-=-\wh\alpha_n(-n)$. If $\Delta \cZ_n\neq c$ set $n_+=\wh\alpha_n(n)$, and if $\Delta \cZ_n=c$ set $n_+=\wh\alpha_n(n)+1$. Define $\alpha_n:[-n_-,n_+]_\Z\to[-n,n]_\Z$ by $\alpha_n=\alpha|_{[-n_-,n_+]_\Z}$, and define $\wh\cZ$ by
\eqbn
\wh\cZ_k=
\begin{cases}
	\cZ_{\alpha(k)}\quad & \text{if\,\,}\cZ_{\alpha(k)}\in\{ a,b\},\\
	-a& \text{if\,\,}\cZ_{\alpha(k)}=c\text{\,\,and\,\,}\alpha(k)=\alpha(k+1),\\
	-b& \text{if\,\,}\cZ_{\alpha(k)}=c\text{\,\,and\,\,}\alpha(k)\neq\alpha(k+1).
\end{cases}
\eqen
\begin{figure}
	\centering
	\includegraphics[scale=1]{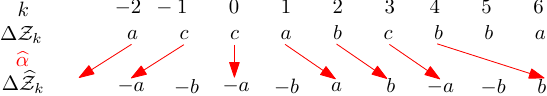}
	\caption[Time change relation $\wh{\cZ}$ and $\cZ$]{Illustration of $\wh\alpha:\Z\to\Z$, which describes the time change of $\wh{\cZ}$ relative to $\cZ$.}
	\label{fig3}
\end{figure}

Since the walk $\wh\cZ$ has nearest-neighbor steps, it determines an infinite triangulation $\wh M$ (analogous to the triangulation $\mcl Q\cup T\cup T_*$ from Section~\ref{sec-kappa8-bijection}), a map $\wh{\frk t} : \BB Z\rta \mcl F(\wh M)$, and an infinite triangulations with boundary $\wh M_{-\infty,n}$ for $n\in\BB N$ (which is a submap of $\wh M$) via the Mullin bijection sewing procedure described in Section~\ref{sec-mullin-inverse}. 

More precisely, for steps $\Delta\wh\cZ_k\in\{a,b \}$ the sewing procedure is as before and $\wh{\frk t}(k)$ is the triangle which we add. If $\Delta \cZ_k=-a$ (resp.\ $\Delta \cZ_k=-b$) we instead add a triangle $\wh{\frk t}(k)$ adjacent to $e_{k-1}^h$, we glue the right (resp.\ left) side of $\wh{\frk t}(k)$ to the edge immediately to its right (resp.\ left) on the frontier, and we define $e_{k}^h$ to be the (only) edge of $\wh{\frk t}(k)$ which is still on the frontier of the map. 

We define $\wh M_n$ be the (finite) submap of $\wh M_{-\infty,n}$ consisting of the vertices and edges on the boundaries of the $2n+1$ triangles added at times in $[-n,n]_{\BB Z}$, as in Section~\ref{sec-kappa8-proof}.

\begin{figure}
	\centering
	\includegraphics[scale=1]{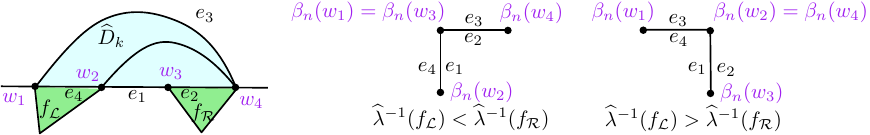}
	\caption[Contraction used to recover the UIPT]{Contraction of $\wh D_k$ (light blue), as defined before the statement of Lemma \ref{prop6}, and illustration (in purple) of the map $\beta_n:\cV(\wh M_n)\to\cV(\Mn)$ defined after Lemma \ref{prop6}. }
	\label{fig2}
\end{figure}

Given $k\in[-n,n]_{\cZ}$ such that $\Delta \cZ_k=c$, let $\wh D_k$ be the submap of $\wh M_n$ of boundary length 4 consisting of the two faces $\wh{\frk t}_n(\wh\alpha_n(k)  )$ and $\wh{\frk t}_n(\wh\alpha_n(k)+1)$ (which correspond to consecutive steps of the $\wh \cZ$ of the form $-a$ and $-b$) and the vertices and edges on their boundaries. We say that we \emph{contract} $\wh D_k$ if we modify the map $\wh M_n$ by the following procedure. (See Figure \ref{fig2} for an illustration.) Let $e_1,e_2,e_3,e_4$ denote the edges on the boundary of $\wh D_k$ in counterclockwise order, such that $e_1=\wh e^h_{\wh\alpha_n(k)-1}$ and $e_3=\wh e^h_{\wh\alpha_n(k)+1}$. A contraction means that we remove the two faces  $\wh{\frk t}_n(\wh\alpha_n(k)  )$ and $\wh{\frk t}_n(\wh\alpha_n(k)+1)$ from the map, and then we identify edges pairwise along the boundary of $\wh D_k$. The edge between the two faces is removed, while the other faces on the boundary of the two faces are kept. Let $f_{\cL}$ (resp. $f_{\cR}$) be the face of $\wh M_{-\infty,n}$ which has $e_4$ (resp.\ $e_2$) on its boundary, but is not equal to  $\wh{\frk t}_n(\wh\alpha_n(k))$ or $\wh{\frk t}_n(\wh\alpha_n(k)+1)$. If $\wt\lambda^{-1}(f_\cL)< \wt\lambda^{-1}(f_\cR)$ (resp.\ $\wt\lambda^{-1}(f_\cL)>\wt\lambda^{-1}(f_\cR)$) then we identify $e_1$ and $e_4$ (resp.\ $e_1$ and $e_2$), and $e_2$ and $e_3$ (resp.\ $e_3$ and $e_4$); notice that the two edges we identify will always have one common vertex and one not common vertex, and when we identify the edges we choose the orientation of the edges in the natural way, such that the common (resp.\ not common) vertices are identified. %In the special case when $k=n$ and $f_{\cR}$ (resp.\ $f_{\cL}$) is not a face of $\wh M_n$ then we also remove the edge obtained when identifying $e_2$ (resp.\ $e_4$) and $e_3$. 

\begin{lem}
Let $\wh\cH_n$ denote the map with vertex set $[-n_-,n_+]_\Z$ and adjacency condition defined by \eqref{eqn-walk-adjacency} with the walk $\wh\cZ$ in place of $\cZ$.

\begin{itemize}
	\item[(i)] The map $k\mapsto\wh{\frk t}(k)$ is a graph isomorphism modulo multiplicity (recall Definition~\ref{def-iso-mod-multiplicity}) from $\wh\cH_n$ to the dual of $\wh M_n$.
	\item[(ii)] If we contract all the submaps $\wh D_k$ of $\wh M_n$ for $k\in[-n,n]_\Z$ such that $ \Delta \cZ_k=c$, then we get a map which is isomorphic to $\Mn$. 
\end{itemize}
	\label{prop6}
\end{lem}
\begin{proof}
	(i) This is immediate by Proposition \ref{prop-kappa-8-tri}, and since $\wh\cZ$ and $\wh M_n$ (in the limit as $n\rta\infty$) are related by Mullin's bijection.

	(ii) Consider two walks $\cZ'$ and $\cZ''$ with steps $\Delta \cZ'$ and $\Delta \cZ''$, respectively, in $\{a,b,c,-a,-b \}$. Assume $\cZ''$ is obtained from $\cZ'$ by choosing a $k_0\in\Z$ such that $\Delta \cZ'_{k_0}=c$, and replacing this step $c$ by two steps $-a$ and $-b$. In other words, we have $\Delta \cZ''_k=\Delta \cZ'_k$ for $k<k_0$, $\Delta \cZ''_{k_0}=-a$, $\Delta \cZ''_{k_0+1}=-b$, and $\Delta \cZ''_k=\Delta \cZ'_{k-1}$ for $k>k_0+1$. Let $M'$ (resp.\ $M''$) denote the map we get when applying the sewing procedure described above to $\cZ'$ (resp.\ $\cZ''$), and let $D''$ be the submap of $M''$ of boundary length 4 corresponding to the faces added in steps $k_0$ and $k_0+1$ of the sewing procedure. Since we can proceed by induction, in order to conclude it is sufficient to show that when we contract $D''$, we get the map $M'$.  This property is immediate by the sewing procedure applied to $\cZ'$ and $\cZ''$, respectively. 
\end{proof}

From Lemma \ref{prop6}(ii) we see that there exists a natural map $\beta_n:\cV(\wh M_n)\to\cV(\Mn)$ given by contracting all of the $\wh D_k$'s. See Figure~\ref{fig2}. We define $\wh\beta_n:\cV(\Mn)\to \cV(\wh M_n)$ in an arbitrary way such that %$\wh\beta_n(v)\in \{w\in\cV(\Mn)\,:\,\beta_n(w)=v\}$, so that 
$\beta_n \circ \wh\beta_n$ is the identity map $\cV(\Mn)\rta \cV(\Mn)$, and such that $\wh\beta_n$ sends $\bdy \Mn$ to $\bdy \wh M_n$. Note that it is possible to find a map $\wh\beta_n$ satisfying the latter property by the construction in the proof of Lemma \ref{prop6}(ii). Let $\innF(\wh M_n)$ denote the set of bounded faces of $\wh M_n$, i.e., $\cF(\wh M_n)\setminus \innF(\wh M_n)$ consist of the single unbounded face of $\wh M_n$. Let $\wh\phi_n:\cV(\wh M_n)\to\innF(\wh M_n)$ map $v\in \cV(\wh M_n)$ to some arbitrarily chosen face of $\wh M_n$ which has $v$ on its boundary. Let $\wh\psi_n:\innF(\wh M_n)\to \cV(\wh M_n)$ map $f\in \innF(\wh M_n)$ to some arbitrarily chosen vertex on its boundary.
We also require that $\wh\phi_n,\wh\psi_n,\wh\beta_n$ are defined such that the maps $\phi_n$ and $\psi_n$ defined just below satisfy the requirements on the root in \eqref{eqn-peano-functions} (note that by the definition of the contraction operation, it is possible to satisfy both this requirement and the second requirement on $\wh\beta_n$ specified right above).  By Lemma \ref{prop6}(i), we may identify $\innF(\wh M_n)$ with $[-n_-,n_+]_\Z$ in a natural way, so we may view $\wh\psi_n$ and $\wh\phi_n$ as maps  
$\wh\psi_n:[-n_-,n_+]_\Z\to \cV(\wh M_n)$ and
$\wh\phi_n:\cV(\wh M_n)\to [-n_-,n_+]_\Z$. Define 
\eqb 
\phi_n:\cV(\Mn)\to [-n,n]_\Z \: \text{by} \: \phi_n :=\alpha_n\circ\wh\phi_n\circ\wh\beta_n \quad \op{and} \quad
 \psi_n:[-n,n]_\Z\to \cV(\Mn) \: \text{by} \: \psi_n :=\beta_n\circ\wh\psi_n\circ\wh\alpha_n.
 \label{eqn-percolation-phi-psi}
\eqe

\begin{proof}[Proof of Lemma \ref{lem-ball-layers} in case~\ref{item-kappa6}]
By the definition of the maps $\phi_n$ it is sufficient to show the following to conclude the proof of the lemma:
\begin{enumerate}[label=(\roman{enumi})]
\item \label{item-i} $\wh\beta_{n}$ sends $\bdy \Mn$ to $\bdy \wh M_n$;  	
\item \label{item-ii} $\alpha_{n}$ sends the complement of $[-m_-,m_+]_\Z$ to the complement of $[-m,m]_\Z$; 	
\item \label{item-iii} if \eqref{eqn-walk-inf-condition0} holds then \eqref{eqn-walk-inf-condition0} also holds with $\wh \cL,\wh\cR, n_\pm,m_\pm$ instead of $\cL,\cR,n,m$; 
\item \label{item-iv} the map $\wh M_n$ satisfies Lemma \ref{lem-ball-layers}, with \eqref{eq:bdy-disjoint} replaced by $\wh\phi_{n}(\bdy \wh M_n)\cap[-m_-,m_+]_{\Z}=\emptyset$.
\end{enumerate}
		It follows from the definition of $\wh\beta_n$ that~\ref{item-i} holds, and it follows from the definition of $\alpha$ and $\wh{\cZ}$ that~\ref{item-ii} and~\ref{item-iii}, respectively, hold. %We obtain~\ref{item-iii} by deforming $\Mn$, $M_m$, and $M$ step by step as described in the proof of Lemma \ref{prop6}(ii) and observing that the considered property is unchanged in each step. 
		Property~\ref{item-iv} follows from the argument for case \ref{item-kappa8} since the sewing procedure for $\wh M_n$ and the definition of $\wh\phi_n$ are the same as for case \ref{item-kappa8}. %we first recall that there is a triangulation $\wh M$ such that each map $\wh M_k$ for $k\in\N$ is a submap of $\wh M$ (this property does not hold for $M_k$ due to the gluing of edges associated with a $c$ step). This allows us to view $\bdy \wh M_n$ and $\bdy \wh M_{m}$ as closed disjoint curves in $\wh M$ such that \eqref{eq:bdy-inclusion} holds without the use of $\iota_n$, i.e., $\wh M_{m} \subset \wh M_n \setminus \bdy \wh M_n$. Also recall that the faces of $\wh M$ are in bijection with $\Z$, such that the interior faces of $\wh M_k$ correspond to $[-k_-,k_+]_{\Z}$ for each $k\in\N$. By definition of $\wh\phi_{n}$ and the fact that $\bdy \wh M_{m}$ is surrounded by and disjoint from $\bdy \wh M_n$, we see that $\wh\phi_{n}(v)$ corresponds to a triangle of $\wh M$ which is \emph{not} surrounded by $\bdy \wh M_{m}$. Since the triangles of $\wh M$ corresponding to $[-m_-,m_+]$ are surrounded by $\bdy\wh M_{m}$  we see that $\wh\phi_{n}(\bdy \wh M_n)\cap[-m_-,m_+]_{\Z}=\emptyset$. 
\end{proof}

\begin{proof}[Proof of Theorem \ref{thm-map-count} in case~\ref{item-kappa6}]
Define the planar maps $\mcl H$ and $\{\cH_n\}_{n\in\BB N}$ as in~\eqref{eqn-walk-adjacency} and the discussion just after with $\mcl Z$ the encoding walk for $M$. Throughout the proof, we fix $A>0$ and $n\in\BB N$ and we couple $\mcl Z$ with the correlated Brownian motion $Z$ in such a way that the conclusion of Theorem~\ref{thm-sg-map-dist} is satisfied for this choice of $\cH$. 

	To prove condition~\ref{item-map-count-G} from the theorem statement, we will argue that for sufficiently large $C>1$, depending only on $A$, the following is true for each $n\in\BB N$. 
	\begin{itemize}
		\item[(i)] For any $v_1\sim v_2$ in $\cV(\Mn)$ we can find a path $P_{v_1,v_2}^{\wh M_n}$ from $\wh\beta_n(v_1)$ to $\wh\beta_n(v_2)$ in $\wh M_n$ of length at most 3, and each edge of $\wh M_n$ is on at most 10 of these paths. (Notice that when we bound the number of paths hitting each $i\in\cV(\cG_n)$ below, it is sufficient to bound the number of paths hitting each edge, rather than each vertex, in this step.)
		\item[(ii)] With probability at least $1-O_n(n^{-A})$, for any $\wh v_1\sim\wh v_2$ in $\cV(\wh M_n)$ we can find a path $P_{\wh v_1,\wh v_2}^{\wh{\mcl H}_n}$ from $\phi_n(\wh v_1)$ to $\phi_n(\wh v_2)$ in $\wh{\mcl H}_n$ of length at most $C\log n$, and each $i\in [-n_-,n_+]_\Z$ is hit by at most $C\log n$ of the paths $P_{\wh v_1,\wh v_2}^{\wh{\cH}_n}$.
		\item[(iii)] For any $i_1\sim i_2$ in $\wh\cH_n$ we can find a path $P_{i_1,i_2}^{\cH_n}$ from $\alpha_n(i_1)$ to $\alpha_n(i_2)$ in $\cH_n$ of length at most 2, and each vertex $i$ of $\cH_n$ is hit by at most 10 of the paths $P_{i_1,i_2}^{\cH_n}$.
		\item[(iv)] With probability at least $1-O_n(n^{-A})$, for any $i_1\sim i_2$ in $\cH_n$ we can find a path $P_{i_1,i_2}^{\cG_n}$ from $i_1$ to $i_2$ in $\cG_n$ of length at most $C(\log n)^3$, and each $i\in [-n,n]_\Z$ is hit by at most $C(\log n)^6$ of the paths $P_{i_1,i_2}^{\cG_n}$.
	\end{itemize}
	Before we prove (i)-(iv), we will argue that they imply condition~\ref{item-map-count-G} upon replacing $C$ by $100C^2$. Indeed, if the events in (i)-(iv) occur (which happens with probability at least $1-O_n(n^{-A})$) and $v_1\sim v_2$ in $\cV(\Mn)$ we can construct the path $P_{v_1,v_2}^{\mcl G_n}$ from $\phi_n(v_1)$ to $\phi_n(v_2)$ in $\mcl G$ by considering the path $P_{v_1,v_2}^{\wh M_n}$ from (i), replacing each step $(\wh v_1,\wh v_2)$ by the path $P_{\wh v_1,\wh v_2}^{\wh{\mcl H}_n}$ from (ii), replacing each step of the resulting path by the appropriate path from (iii), then replacing each step of this last path by the appropriate path from (iv). It is immediate from (i)-(iv) that the length of the path will be at most $3\cdot C\log n\cdot 2\cdot C(\log n)^3=6C^2(\log n)^4$, and that each $i$ is hit by at most $10\cdot C\log n\cdot 10\cdot C(\log n)^6=100C^2(\log n)^7$ of the paths $P_{v_1,v_2}^{\wh M_n}$.
	
	The property (i) follows from Lemma \ref{prop6}(ii), while  (iii) is immediate from the adjacency condition~\eqref{eqn-walk-adjacency} for $\mcl H_n$ and $\wh{\mcl H}_n$. It follows from Theorem \ref{thm-sg-map-dist} that we can couple so that (iv) holds. Since the degree of the root of the UIPT has an exponential tail~\cite[Lemma 4.2]{angel-schramm-uipt}, we know that for large enough $C$ with probability at least $1-O_n(n^{-A})$, 
	\eqb
	\op{deg}(v; M_{ n})\leq C\log n,\qquad\forall v\in \cV(M_{ n}).
	\label{eq2}
	\eqe
The property (ii) follows from Lemma \ref{prop6}(i) and \eqref{eq2}, since the definition of $\phi_n$ implies that the face of $\wh M_n$ associated with $\phi_n(\wh v_1)$ (resp.\ $\phi_n(\wh v_2)$) has $\wh v_1$ (resp.\ $\wh v_2$) on its boundary, so we can find a path of faces for $\wh M_n$ from $\phi_n(\wh v_1)$ to $\phi_n(\wh v_2)$ of length at most $C\log n$ (c.f.\ the proof of condition~\ref{item-map-count-G} of Theorem~\ref{thm-map-count} from Section~\ref{sec-kappa8-proof} for a similar argument).
	
	The proof of condition~\ref{item-map-count-M} is very similar, and we therefore only give a brief sketch. First we prove variants of (i)-(iv) above for the inverse maps. For example, instead of (i) above, we show that for any $\wh v_1\sim \wh v_2$ in $\cV(\wh M_n)$ we can find a path $P_{\wh v_1,\wh v_2}^{\Mn}$ from $\beta_n(\wh v_1)$ to $\beta_n(\wh v_2)$ in $\Mn$ of length at most 3, and each edge of $\Mn$ is on at most 10 of these paths. The new variants of (i)-(iv) are proved as before, and they imply condition~\ref{item-map-count-M} similarly as in the proof of condition~\ref{item-map-count-G}. 
	
	To prove condition~\ref{item-map-count-close} we will only explain how to bound $\op{dist}(\phi_n\circ\psi_n(i),i;\cG_n)$, since the bound for $\op{dist}(\psi_n\circ\phi_n(v),v;\Mn)$ is proven in a similar way. Assume the events in (i)-(iv) and \eqref{eq2} are satisfied, and choose an arbitrary $i\in[-n,n]_\Z$. Define 
	\eqbn
	\wh i=\wh\alpha_n(i), \quad
	\wh v=\wh\psi_n(\wh i),\quad
	v=\beta_n(\wh v),\quad
	\wh v'=\wh\beta_n(v),\quad
	\wh i'=\wh\phi_n(\wh v'),\quad
	i'=\alpha_n(\wh i').
	\eqen 
	By the definition of contraction and since $\wh v'=\wh\beta_n(\beta_n(\wh v))$, we have $\op{dist}(\wh v,\wh v';\wh M_n)\leq 2$. By the triangle inequality, 
	\eqbn
	\op{dist}(\wh i,\wh i';\wh \cH_n)
	\leq 
	\op{dist}(\wh i,\wh\phi_n(\wh v);\wh \cH_n)
	+
	\op{dist}(\wh\phi_n(\wh v),\wh i';\wh \cH_n).
	\eqen
	The first term on the right side is bounded by $C\log n$, 
		since it follows from the definition of $\wh\psi_n$ and $\wh\phi_n$ that the faces $\wh{\frk t}(\wh i)$ and 
		$\wh{\frk t}(\wh\phi_n(\wh\psi_n(\wh i)))$ have a common vertex, so 
		the distance between the two faces in the dual of $\wh M_n$ (equivalently, $\wh\cH_n$) 
		is at most $C\log n$ by \eqref{eq2}. The second term on the right side is bounded by  $2C\log n$, since (ii) and $\op{dist}(\wh v,\wh v';\wh M_n)\leq 2$ imply that $\op{dist}(\wh\phi_n(\wh v),\wh\phi_n(\wh v');\wh\cH_n)\leq 2C\log n$. We conclude that $\op{dist}(\wh i,\wh i';\wh \cH_n)\leq 3C\log n$. Using this bound, (iii), the definition of $\alpha_n,\wh\alpha_n$, and $i'=\alpha_n(\wh i')$, we get further
	\eqbn
	\op{dist}(i,i';\cH_n)
	\leq 
	\op{dist}(i,\alpha_n(\wh i); \cH_n)
	+
	\op{dist}(\alpha_n(\wh i),i'; \cH_n)
	\leq 1+2\cdot 3C\log n.
	\eqen
	Combining this with (iv) concludes the proof:
	$\op{dist}(i,i';\cG_n)
	\leq (1+6C\log n)(\log n)^3\leq 7C(\log n)^4$, where we assume that $n$ is chosen sufficiently large in the last step.
\end{proof}

\begin{proof}[Proof of Theorems~\ref{thm-map-ball},~\ref{thm-map-dist}, and~\ref{thm-6-ball} in case~\ref{item-kappa6}] 
By~\cite[Theorem 1.2]{angel-peeling}, the metric ball $B_n(e_0 ; M)$ of radius $n$ centered at the root edge in the UIPT satisfies 
\eqb \label{eqn-uipt-ball}
\# \mcl V (B_n(e_0 ; M)) = n^{4+o_n(1)} \quad \text{with probability tending to 1 as $n\rta\infty$} , 
\eqe 
which is stronger than the UIPT case of Theorem~\ref{thm-map-ball}. (Observe that the weaker variant stated in Theorem~\ref{thm-map-ball} can also be proved by proceeding as in the UST case considered above.) Theorem~\ref{thm-map-dist} follows from Theorem~\ref{thm-map-count} and~\cite[Theorem 1.15]{ghs-dist-exponent}. Theorem~\ref{thm-6-ball} is immediate from~\eqref{eqn-uipt-ball} together with the UIPT versions of Theorem~\ref{thm-map-count} and Lemma~\ref{lem-ball-iso}.
\end{proof}

%%%%%%%%%%%%%%%%%%%%%%%%%%%%%%%%%%%%%%%%%%%%%%%%%%%%%%%%%%%%%%%%
\subsection{Bipolar-oriented planar maps and Schnyder woods}
\label{sec:bipolar}

In this section we study the case of planar maps decorated by a bipolar orientation or Schnyder wood. We first do the necessary combinatorial arguments for finite bipolar-oriented maps in Section~\ref{subsub:bipolar}. Then we move to the infinite-volume setting in Section~\ref{subsub:ubom}, and explain why the \emph{uniform infinite bipolar-oriented map}, the local limit of finite bipolar-oriented maps, exists (the analogous arguments in the case of the other maps considered in this paper have already been carried out elsewhere). We then prove Theorem~\ref{thm-map-count} in the case of the uniform infinite bipolar-oriented map (case~\ref{item-kappa12}) in Section~\ref{subsub:bi-metric}. Finally in Section~\ref{subsub:Schnyder} we treat the case of bipolar-oriented maps with other face degree distributions (case~\ref{item-kappa>8}) and show that Schnyder wood-decorated maps (case~\ref{item-kappa16}) can be treated as a special case of bipolar-oriented maps.

\subsubsection{The mating-of-trees bijection for bipolar-oriented maps}\label{subsub:bipolar}
An \emph{orientation} on a graph is an assignment of a direction to each edge. A vertex is called a sink (resp. source) if there are no outgoing
(resp. incoming) edges incident to the vertex. Sources and sinks are also called \emph{poles}. A \emph{bipolar orientation} of $G$ with specified source and sink is an acyclic orientation with no source or sink except at the specified poles.  A \emph{bipolar-oriented  planar map} is a planar  map  $\map$, with a marked face $\outf$ and a bipolar orientation $\cO$ on $\map$ whose source and sink are both on the boundary of $\outf$.  Bipolar-oriented  planar maps have  rich combinatorial structures and numerous applications
in algorithms. For an overview of the graph theoretic perspective on bipolar orientations,
we refer to~\cite{fmor-bipolar} and references therein. See also~\cite{fps-counting-bipolar,bbf-bipolar-bijection} and the
references therein for enumerative and bijective results.

The plane $\R^2$ is naturally associated with the notion of east, west, south and north. It is known~\cite{abrams-kenyon-enharmonic} that any bipolar-orientated planar map  can be embedded into $\R^2$ in such a way that every edge is north oriented and the marked face is unbounded. See Figure~\ref{fig:bip1} for an example. Given a bipolar-oriented map $(\map,\outf,\cO)$, we always  embed it this way so that $\cO$ can be read off from the embedding and the source and sink can be identified as the south pole and north pole, respectively, of the map. We therefore denote the source and sink by $\spole(\map)$ and $\npole(\map)$, respectively.  For any edge $e\in \cE(\map)$ directed as in $\cO$, we also write $\spole(e)$ and $\npole(e)$, respectively, for the initial and terminal edge of $e$.
\begin{figure}
	\centering
	\includegraphics[scale=0.7]{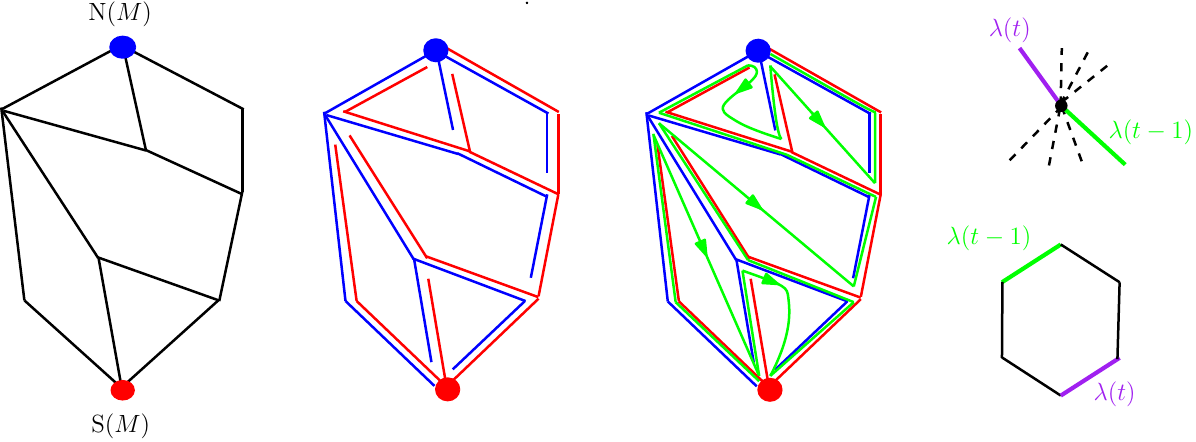}
	\caption[The bijection for bipolar-oriented maps]{
		In the first panel, we embed a bipolar-orientated map $(\map,\outf,\cO)$	so that each edge is pointing to the north direction.   
		We use blue and red vertices to represent their corresponding north and south poles.  The second panel illustrates the (blue) NW tree  and (red) SE tree. On the third panel, there is a path $\path$, drawn in green, as the interface between the NW and SE trees. The path $\path$ induces a total ordering  of $\cE(\map)$ which we still denote by $\path$. The fourth panel lists the two cases in the definition of $\path$ and $\wt\path$. In both cases $\path(t-1)$ is colored green and $\path(t)$ is colored purple. In case~\ref{item:path1} (top figure) $\wt\path(t)$ is the boldfaced vertex and in case~\ref{item:path2} (bottom figure) $\wt\path(t)$ is the face.
	}
	\label{fig:bip1}
\end{figure}

Let $\innV(\map):=\cV(\map)\setminus \{\npole(\map),\spole(\map) \}$ be the set of non-polar vertices of $\map$ 
and $\innF(\map):=\cF(\map)\setminus \{f_0\}$ the set of \emph{bounded} faces of $\map$. 
We call a bipolar-oriented map consisting of a simple cycle (separating a single bounded face from the unbounded face) a \emph{bipolar-oriented face}. 
The edges of a bipolar-oriented face can be divided into those which lie on the west and east boundary arcs connected the two poles, which we call \emph{west edges} and \emph{east edges}, respectively.
We can think of any face in $\innF(\map)$ as a bipolar-oriented face.  We call the clockwise (resp. counterclockwise) arc on $\outf$  from $\spole(\map)$ to $\npole(\map)$ the west (resp. east) boundary of $\map$. An edge $e$ is on both the west and east boundaries of $\map$ if and only if removing $e$ separates $\map$ into two connected components.

We will now review the mating-of-trees bijection for finite bipolar-oriented maps from~\cite{kmsw-bipolar}.
For any $v\in \cV(M)\setminus\{\npole(\map)\}$, there is a unique edge incident to $v$ which is oriented away from $v$ and is the furthest west among all such edges. We call this edge the \emph{NW (northwest) edge} of $v$. 
Similarly, there is a unique edge incident to $v$ which is oriented toward $v$ and is the furthest east among all such edges. We call this edge the \emph{SE (southeast) edge}  of $v$. 
The \emph{NW tree}  $\nwt$ of $(M,f_0,\cO)$ is the directed planar  tree  that can be drawn as follows: each NW
edge is entirely in $\nwt$; for each other edge, a segment containing the
head of the edge is in $\nwt$. 
Similarly, the SE tree $\set$ can be drawn as follows: each SE
	edge is entirely in $\set$; for each other edge, a segment containing the
	tail of the edge is in $\set$. 
Once $\nwt$ and $\set$ are drawn, we can draw an \emph{interface} path $\path$ winding between them
from $\spole(\map)$ to $\npole(\map)$.  The path $\path$ introduces an ordering on $\cE(\map)$, i.e. a mapping (still denoted by $\path$) from $[0,\#\cE(\map) -1]_{\BB Z}$ to $\cE(\map)$. It also defines a bijection $\wt\lambda:[1,\#\cE(\map) - 1]_{\BB Z}\to \innV(\map)\cup \innF(\map)$. 
We now give a formal definition of $\path,\wt\path$.
\begin{definition}\label{def:sewing}
	Set $\path(0)$ to be the NW edge of $\spole(\map)$.  Inductively, for $1\le t\le |\cE(\map)|-1$, given $\path(t-1)$:
	\begin{enumerate}
		\item \label{item:path1} If $\path(t-1)$ is the SE edge of  $\npole(\path(t-1))$, let $\path(t)$ be the NW edge of  $\npole(\path(t-1))$ and set $\wt \path(t):=\spole(\path(t))\in \cV(\map)$.
		%	\item $\npole(\path(t))$ is closer to  $\npole(\map)$ than $\npole(\path(t-1))$ in the $\nwt$-tree distance, in which case we let $\wt \path(t):=\spole(\path(t))\in \cV(\map)$.
		\item \label{item:path2} Otherwise, let $\wt\path(t)$ be the unique  face where $\path(t-1)$ is a west edge and $\path(t)$ be the  unique east edge of $\wt\path(t)$ where $\spole(\path(t))=\spole(\wt\path(t))$.
	\end{enumerate}
\end{definition}

See Figure~\ref{fig:bip1} for an illustration of $\nwt, \set,\path,\wt\path$.

%Now we recall the mating-of-trees bijection for  bipolar-oriented maps discovered in~\cite{kmsw-bipolar}.  
Suppose  $(\map,f_0,\cO)$ has $m+1$ edges on the west boundary and  $n+1$ edges on the east boundary. We can associate $(\map,f_0,\cO)$ with a walk $\cZ=(\cL,\cR)$ on $\Z^2$, starting from  $(m,0)$ when $t=0$ and ending at $(0,n)$ when $t= \#\cE(\map) -1$. For $t \in [1 , \# \cE(\map) -1]_{\BB Z}$,  if $\wt\path(t)\in \innV(\map)$, we define the increment $\Delta\cZ_t:=\cZ_t-\cZ_{t-1}$ to be $(-1,1)$.  If $\wt\path(t)\in \innF(\map)$, we set $\Delta\cZ_t=(i,-j)$ where $i+1$ and $j+1$ are the number of the east and west edges of the face $\wt\path(t)$ respectively.
It is shown in~\cite[Theorem 2]{kmsw-bipolar} that this procedure gives a bijection from bipolar-oriented maps to finite-length lattices walks on $\Z^2_{\ge 0}$ starting at the $x$-axis, ending at the $y$-axis, whose steps are in the set 
\begin{equation}\label{eq:step}
\{(-1,1)\} \cup \{(i,-j) : i,j \in \Z_{\ge 0} \}. 
\footnote{The walk in~\cite{kmsw-bipolar} is written as $(\cR,\cL)$, thus the set of steps differs from \eqref{eq:step} by a reflection against $\{x=y\}$.}
\end{equation}

Following \cite[Section~2.2]{kmsw-bipolar},  we now describe the  \emph{sewing procedure} in this case, which is a way to dynamically build a bipolar-oriented map $(\map,f_0,\cO)$ from the lattice path $(\cZ)_{0\le t\le \#\cE(\map) -1}$ of this type.  
For each $0\le k\le \# \cE(\map) -1$, we inductively associate $(\cZ_t)_{0\le t\le k}$ with  a bipolar-oriented map $(\map^k,\outf^k,\cO^k)$ with a marked  edge $e^k$ on its east boundary.  When $k=0$, the map $(\map^0,\outf^0,\cO^0)$  is just a directed edge $e^0$. For $0<k\le \#\cE(\map) -1$, suppose $(\map^{k-1},\outf^{k-1},\cO^{k-1},e^{k-1})$  is constructed.
\begin{enumerate}
	\item \label{item:sew1} If $\Delta\cZ_{k}=(-1,1)$  and $\npole (e^{k-1})\neq \npole (\map^{k-1})$, then  $(\map^{k},\outf^{k},\cO^{k})=(\map^{k-1},\outf^{k-1},\cO^{k-1})$ and $e^k$ is the unique edge on the east boundary of $\map^{k-1}$  with $\spole(e^k)=\npole(e^{k-1})$.
	\item\label{item:sew2} If  $\Delta\cZ_{k}=(-1,1)$  and  $\npole (e^{k-1})=\npole (\map^{k-1})$, then $(\map^{k},\outf^{k},\cO^{k},e^k)$ is obtained by attaching a north going edge $e^{k}$ to $(\map^{k-1},\outf^{k-1},\cO^{k-1})$ at $ \npole (\map^{k-1}) $ so that $\npole (\map^{k})=\npole (e^{k})$. 
	\item \label{item:sew3} If $\Delta\cZ_{k}=(i,-j)$ for some $i,j\ge 0$, to construct $(\map^{k},\outf^{k},\cO^{k},e^k)$, we glue a bipolar-oriented face $f^k$  with $i+1$ east edges and $j+1$ west edges to $\map^{k-1}$,  so that $\npole(f^k)=\npole(e^k)$, the west edges of $f^k$ are east edges of  $\outf^k$, and the east edges of $f^k$ are \emph{not} edges of $\map^{k-1}$. The marked edge $e^k$ is updated to be the unique east edge of $f^k$ satisfying $\spole(e^k)=\spole(f^k)$.
\end{enumerate}

\begin{figure}[h!]
	\centering
	\includegraphics[scale=0.65]{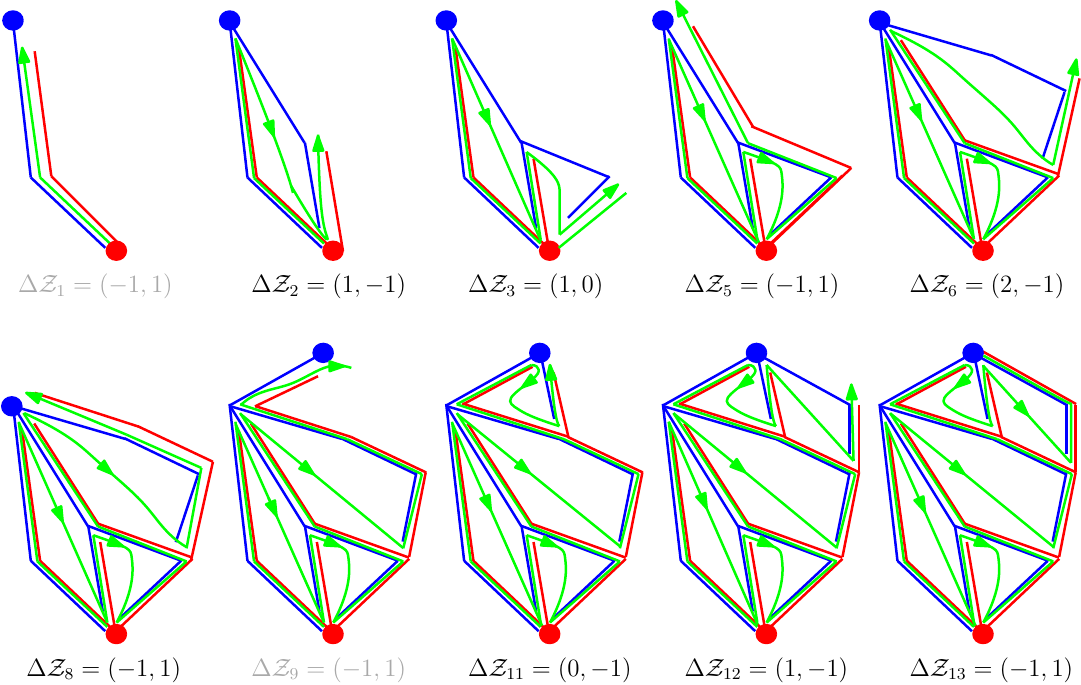}
	\caption[The sewing procedure for bipolar-oriented maps]{We present 10 snapshots of the sewing procedure for the bipolar-oriented map in Figure~\ref{fig:bip1} where  $k\in \{1,2, 3,5,6,8,9, 11,12,13\}$. The ten figures represent these $(\map^k,\outf^k,\cO^k)$'s in order with the corresponding incremental step below. 
		$\Delta\cZ=(-1,1)$ corresponds to a Type~\ref{item:sew1} or Type~\ref{item:sew2} move, i.e. adding an edge. Here the Type~\ref{item:sew2} moves are displayed in gray .  
		$\Delta\cZ\neq (-1,1)$ corresponds to a Type~\ref{item:sew3} move., i.e. adding a face.
		The marked edge $e^k$ of $(\map^k,\outf^k,\cO^k)$ is always the current edge visited by the (green) path $\path$, i.e. $\path(k)$.
		The final map $(\map^{13},\outf^{13},\cO^{13})$ recovers  Figure~\ref{fig:bip1}.}
	\label{fig:sewing}
\end{figure}

From the construction, $\{(\map^k,\outf^k,\cO^k)\}_{0\le k\le \# \cE(\map) -1}$ is an increasing family of bipolar-oriented maps with $\spole(\map)$ as its common south pole. See Figure~\ref{fig:sewing} for an illustration.  It is explained in~\cite{kmsw-bipolar} (and can be easily checked by induction) that the final map $(\map^{\# \cE(\map) -1},\outf^{\# \cE(\map) -1},\cO^{\# \cE(\map) -1})$ is the bipolar-oriented map  $(\map,f_0,\cO)$ producing $\cZ$,  thus each $(\map^k,\outf^k,\cO^k)$ can be considered as  a submap of $(\map,f_0,\cO)$.
Moreover, the number of edges from  $\spole (e^k)$ to $\spole(\map)$ (resp. $\npole (e^k)$ to $\npole(\map)$) on the east  boundary of $(\map^k,\outf^k,\cO^k)$ is equal to $\cR_k$ (resp. $\cL_k$). 
From \ref{item:path1} and~\ref{item:path2} in the Definition~\ref{def:sewing}  we observe  that
\begin{align*}
\path(k)&=e^k &&\textrm{for all} \,\, k,\\
\wt \path(k)&=f^k &&\textrm{whenever}\,\,  \wt \path(k)\in \innF(\map),\\ 
\wt \path(k)&= \spole (e^k) &&\textrm{whenever}\,\,  \wt \path(k)\in \innV(\map).
\end{align*}

By the definition of $\cZ$, $\lambda$ and the sewing procedure, we get the following.
\begin{lem}\label{lem:NW}
For two integers $i,j$, the edge $\lambda(i)$ is the NW edge of $\npole(\lambda(j))$ if and only if $i=\inf\{t>j: \cL_t=\cL_j-1 \}$.
Similarly, $\lambda(i)$ is the SE edge of $\spole(\lambda(j))$ if and only if $i=\sup\{t<j: \cR_t=\cR_j-1 \}$.
\end{lem}

The  following observation will allow us to apply Theorem~\ref{thm-sg-map-dist} in the proof of Theorem~\ref{thm-map-count} in the bipolar-oriented map case. 
In its statement we extend $\cZ$ to $[-1,\#\cE(\map)]_{\BB Z}$ and $\wt \lambda$ to $[0, \#\cE(\map)]_{\BB Z}$ by setting 
\begin{equation}\label{eq:ext}
\cZ_0-\cZ_{-1}=  \cZ_{\#\cE(\map) }-\cZ_{\# \cE(\map) -1} =(-1,1), \quad  \wt\lambda (0)=\spole(\map), \quad \textrm{and}\quad \wt\lambda(n)=\npole(\map).
\end{equation}

\begin{lem}\label{lem:LR}
	For $s,t\in [0,\#\cE(\map)]_{\BB Z}$ with $s<t$, we have that
	\begin{enumerate}
		\item \label{item:R} 
		$\wt\path(s)$ is the tail of a west edge on the boundary of $\wt \path(t)$ if and only if 
		\begin{equation}\label{eq:R}
		\cR_{s-1} \vee (\cR_{t}-1) < \min_{s\le j<t} \cR_j
		\end{equation} 
		\item\label{item:L} 
		$\wt\path(t)$ is the head of an east edge on the boundary of $\wt \path(s)$  if and only if 
		\begin{equation}\label{eq:L}
		\cL_{t} \vee (\cL_{s-1}-1) < \min_{s-1<j\le t-1} \cL_j
		\end{equation}
	\end{enumerate}
\end{lem}

\begin{figure}
	\centering
	\includegraphics[scale=0.65]{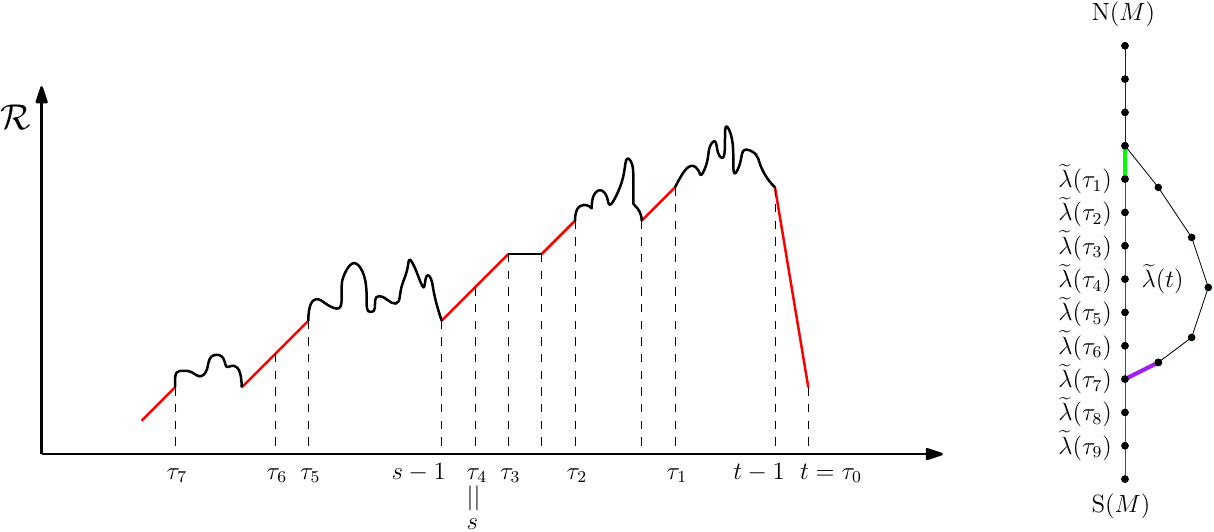}
	\caption[Illustration of the proof of Lemma~\ref{lem:LR}]{Illustration of the proof of Lemma~\ref{lem:LR}.  The left panel is an example of $\cR$ and $\tau^{t-1}_{\bullet}$ when $\Delta\cZ_{t}\in \Z_{\ge 0}\times \Z_{\le 0}$. We write $\tau^{t-1}_{\bullet}$ as $\tau_{\bullet}$ for simplification. In the right panel, the straight line represents the east  boundary of $\map^{t-1}$. The green edge is $\path(t-1)$ and the purple edge is $\path(t)$. The vertices to the south of $\npole(\path(t-1))$ are $\wt\path(\tau_1),\wt\path(\tau_2),\dots$ in order. At time $t$ we add a face $\wt\path(t)$ with west boundary length $\ell=7$. We see that  $\spole(\wt\path(t))=\wt\path(\tau_7)$. Moreover, in the left panel for $t,\cR$ fixed, the time $s$ satisfies the inequality \eqref{eq:R} if and only if $s\in \{\tau_i \}_{1\le i\le 7}$, which proves Assertion~\ref{item:R}.
	}
	\label{fig:R}
\end{figure}

\begin{proof}
	See Figure~\ref{fig:R} for an illustration of the proof.
	For all $k\in [0,\#\cE(\map)]_\Z$, let $\tau^k_0=k$. For $i\ge 0$ and $\tau^k_i\ge 0$, we set $\sup \emptyset=-\infty$ by convention and inductively define
	\eqbn
	\tau^k_{i+1}=\sup\{t<\tau^k_i:  \Delta \cZ_t=(-1,1) \;\textrm{and}\; \cR_t=\cR_{\tau^k_i}-1\}.
	\eqen
	We first prove the following claim by an induction on $k$:
	\begin{align}\label{eq:bdy}
	&\textrm{$\spole(e^k)=\wt \path(\tau^k_1)$ and the vertices on the east boundary of $\map^k$  between $\spole(e^k)$ and} \notag \\
	&\textrm{$\spole(\map^k)$ are $\wt \path(\tau^k_1),\wt \path(\tau^k_2),\cdots,\wt\path(\tau^k_{\cR_k+1})$ in order.}
	\end{align}
	For $k=0$, \eqref{eq:bdy} holds by \eqref{eq:ext}. 
	For $1\le k\le \#\cE(\map)$, assume \eqref{eq:bdy} holds for $k-1$. 
	If $\Delta \cZ_{k}=(-1,1)$,  then $\tau^k_i = \tau^{k-1}_{i-1}$ for $i \geq 1$ and by the definition of the sewing procedure we add an edge at time $k$ (so we do not cover up any boundary vertices), which shows that~\eqref{eq:bdy} holds for $k$.
	If $\Delta \cZ_{k}\in \Z_{\ge 0}\times \Z_{\le 0}$, then at time $k$ we add a face $\wt\path(k)$ with west boundary length $\ell_k:=\cR_{k}-\cR_{k-1}$. 
	Therefore $\tau^{k}_i=\tau^{k-1}_{\ell_k+i-1}$ for $i \ge 1$ and $\spole(e^k)=\wt \path(\tau^{k-1}_{\ell_k})=\wt \path(\tau^{k}_1)$. Thus \eqref{eq:bdy} holds for $k$.
	This proves~\eqref{eq:bdy}.
	Note that $s,t$ satisfies \eqref{eq:R} if and only if $s\in\{\tau^{t-1}_i: 1\le i\le \ell_{t-1}\}$. Now Assertion~\ref{item:R} follows from \eqref{eq:bdy}. 
	%Now viewing $\eqref{eq:R}$ as a statement about $t$, we simultaneously proved \eqref{eq:bdy} and \eqref{eq:R} by induction.  
	
	By rotating $\R^2$ by 180 degrees, we reverse the direction of each edge on $(\map,\outf,\cO)$ and swap the role of south and north as well as west and east. We can also reverse the path $\path$ and the walk $\cZ$. Then the role of $\cR$ is played by the time reversal of $\cL$ and vice versa.
	Here we note that after the time reversal, $(s-1,s,t-1,t)$ becomes $(t,t-1,s,s-1)$. This time reversal trick gives Assertion~\ref{item:L} from Assertion~\ref{item:R}.
\end{proof}

\begin{remark}\label{rmk:reversal}
	The time reversal trick in the proof of Lemma~\ref{lem:LR} will be used frequently.
\end{remark}

\subsubsection{ The uniform infinite bipolar-oriented map}\label{subsub:ubom}

	\begin{figure}
		\centering
		\includegraphics[scale=0.7]{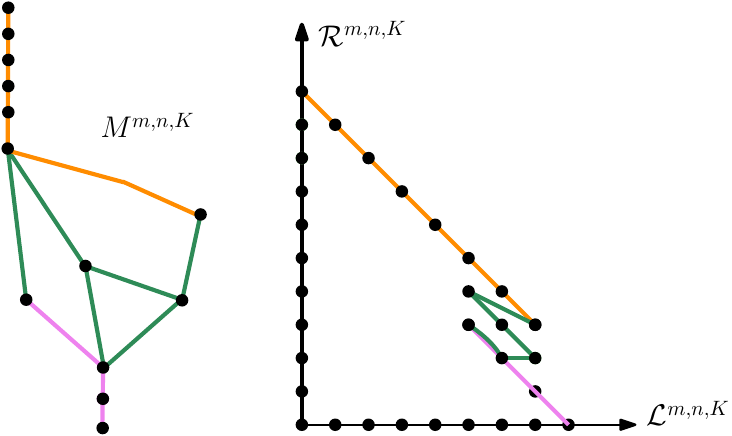}
		\caption{The walk $\cZ^{m,n,K}$ for $m=1$, $n=6$, $K=4$ and its corresponding bipolar-oriented map $(M^{m,n,K}, f^{m,n,K}_0,\cO^{m,n,K})$, adjusted from the fifth figure in Figure~\ref{fig:sewing}. The green edges are $\lambda^{m,n}(t)$ for $t\in [m,n]_\Z$. An edge of a certain color in the map corresponds to a step of the same color in the walk. }
		\label{fig:translate}
	\end{figure}
	In this section, we construct an  infinite planar map which will be the local limit of the uniform bipolar-oriented maps as the size tends to $\infty$ (see Proposition~\ref{prop:BS} and Definition~\ref{def:UIBOM}.).
	Consider a bi-infinite simple random walk $\cZ=(\cL,\cR)$ with $\cZ_0=0$ and the step distribution $\nu$ supported on \eqref{eq:step} such that
	\begin{equation}\label{eq:step1}
	\textrm{$\nu(i,-j) = 2^{-i-j-3}$ for $i,j\ge 0$ or $i=j=-1$}.
	\end{equation}
	As above, we write $\Delta \cZ_t=\cZ_t-\cZ_{t-1}$ for all $t\in \Z$. 
	For two integers $m<n$, we choose $K\in\N$ large enough so that $\cZ([m,n]_{\BB Z}) +(K,K)\subset (0,\infty)^2$.
	Let $x= \cL_{m}+\cR_{m}+2K$ and $y= \cL_{n}+\cR_{n}+2K$
	so that the straight line segment between $\cZ_{m}+(K,K)$ (resp., $\cZ_n+(K,K)$) and $(x,0)$ (resp., $(0,y)$) has slope $-1$.
	By concatenating the trajectory of $\cZ|_{[m,n]_{\BB Z}}  +(K,K)$ with these two line segments of slope $-1$, each traversed by steps of the form $(-1,1)$,
	we get a walk $\cZ^{m,n,K}$ with steps in \eqref{eq:step}, starting from $(x ,0)$ and ending at $(0,y )$.  
	Applying the sewing procedure in Section~\ref{subsub:bipolar} to $\cZ^{m,n,K}$ we obtain bipolar-oriented map $(M^{m,n,K}, f^{m,n,K}_0,\cO^{m,n,K})$. 
	We define $\lambda^{m,n,K}$ as $\lambda$ in Definition~\ref{def:sewing} with $(M^{m,n,K}, f^{m,n,K}_0,\cO^{m,n,K})$
	in place of $(M,f_0,\cO)$. 
	See Figure~\ref{fig:translate} for an illustration.

	Since $\cZ_t=\cZ^{m,n,K}_{\cR_m+K+t-m}$ for each $t\in[m,n]_\Z$, the map $M^{m,n,K}$ can be identified as a submap of $M^{m,n,K+1}$ with compatible edge orientations. In this identification,  for each $t\in[m,n]_{\BB Z}$,
	$\lambda^{m,n,K}(\cR_m+K+t-m)$ and $\lambda^{m,n,K+1}(\cR_m+K+1+t-m)$ represent  the same edge, which we denote by $\lambda^{m,n}(t)$.   This edge corresponds to the step $\Delta \cZ^{m,n,K}_{\cR_m+K+t-m}=\Delta \cZ_t$.
	If $M^{m,n,K}$ is removed from $M^{m,n,K+1}$, we are left with two disjoint edges, one oriented from $\npole(M^{m,n,K})$ to  $\npole(M^{m,n,K+1})$, the other from $\spole(M^{m,n,K+1})$ to 
	$\spole(M^{m,n,K})$.
	Taking the union $\bigcup_K M^{m,n,K}$, we obtain an infinite planar map $M^{m,n}$ with an edge orientation $\cO^{m,n}$ where  $\cO^{m,n}|_{\cE(M^{m,n,K})}=\cO^{m,n,K}$ for large enough $K$.
	Moreover, $M^{m,n,K}$ is a submap of $M^{m,n}$ such that $\cE(M^{m,n})\setminus \cE(M^{m,n,K})$ form two disjoint rays emanating from $\npole(M^{m,n,K})$ and $\spole(M^{m,n,K})$, respectively.  Concatenating  the ray from  $\npole(M^{m,n,K})$ to the NW tree of $M^{m,n,K}$, we obtain an infinite planar tree which we call the NW tree of $M^{m,n}$ that does not depend on $K$.  Similarly, by concatenating  the ray from  $\spole(M^{m,n,K})$ to  the SE tree of $M^{m,n,K}$, we obtain the SE tree of $M^{m,n}$.
	
	By the time reversal trick in the proof of Lemma~\ref{lem:LR}  (see also Remark~\ref{rmk:reversal}),  we have the following.
		\begin{lem}\label{lem:rev}
			For each $t\in\Z$, let $\ol\cZ_t=(\cR_{-t}, \cL_{-t})$.  For $m<n\in\Z$,  let $\wt \cO^{m,n}$ be the orientation obtained by reversing each arrow in $\cO^{m,n}$.  Then 
			$(M^{m,n}, f^{m,n}_0, \wt\cO^{m,n})$ is isomorphic to  $(\ol M^{-n,-m}, \ol f^{-n,-m}_0, \ol \cO^{-n,-m})$,  which is defined in the same way as $(M^{-n,-m}, f^{-n,-m}_0, \cO^{-n,-m})$ with $\cZ$ replaced by $\ol\cZ$.
		\end{lem}
		Recall that $M^{m,n}$ can be embedded so that $\cO^{m,n}$ is pointing upward for each edge. By  Lemma~\ref{lem:rev}, $(\ol M^{-n,-m}, \ol \cO^{-n,-m})$ is obtained by rotating  this picture  by 180 degrees.
	
	For $\ell\in [m,n]_{\BB Z}$,   let  $\{b_{\ell}^{m,n}(i)\}_{i\in\Z}\subset\cE(M^{m,n})$ be such that $b_\ell^{m,n}(0)=\lambda^{m,n}(\ell)$, $b^{m,n}_{\ell}(i+1)$ is the NW edge of $\npole(b^{m,n}_\ell(i))$ for $i\ge 0$, and   $b^{m,n}_{\ell}(i-1)$ is the SE edge of $\spole(b^{m,n}_\ell(i))$ for $i\le 0$.  
	Intuitively, we can think of $\{b^{m,n}_{\ell}(i)\}_{i>0}$ as the branch of the NW tree of $M^{m,n}$ from $\npole(\lambda^{m,n}(\ell))$ to infinity, and $\{b^{m,n}_{\ell}(i)\}_{i<0}$ as the branch of the SE tree of $M^{m,n}$ from $\spole(\lambda^{m,n}(\ell))$ to infinity.	
	\begin{lem}\label{lem:bdy}
		For $m'<m<n<n'$, $M^{m,n}$ is isomorphic to the  submap of $M^{m',n'}$ bounded by the bi-infinite lines $b^{m',n'}_{m}$ and $b^{m',n'}_n$ in $M^{m,n}$. Under this identification, $\cO^{m,n}|_{\cE(\map^{m,n})}=\cO^{m,n}$.  For each $t\in [m,n]_\Z$, we have  $\lambda^{m,n}(t)=\lambda^{m',n'}(t)$ and $b^{m',n'}_t=b^{m,n}_t$.
	\end{lem}
	\begin{proof}
		This can be checked straightforwardly from the definition of the sewing procedure in the finite volume setting.
		We leave the details to the reader.
	\end{proof}
	
	Now for each $n\in\N$ we let $e^0=\lambda^{-n,n}(0)$ be the root edge of $M^{-n,n}$. Taking the union $\bigcup_{n\in\N}M^{-n,n}$ , we obtain an infinite planar map $M$ with the root edge $e^0$ and an edge orientation $\cO$ such that $\cO|_{\cE(M^{-n,n})} =\cO^{-n,n}$ for each $n\in\N$.
	For each $t\in \Z$, let $\lambda(t)=\lambda^{-n,n}(t)$ for $n>|t|$. This defines a function $\lambda: \Z\to \cE(M)$. By definition $\lambda$ is an injection. Lemma~\ref{lem:face} below will show that $\lambda$ is in fact a bijection.

\begin{figure}
	\centering
	\includegraphics[scale=0.65]{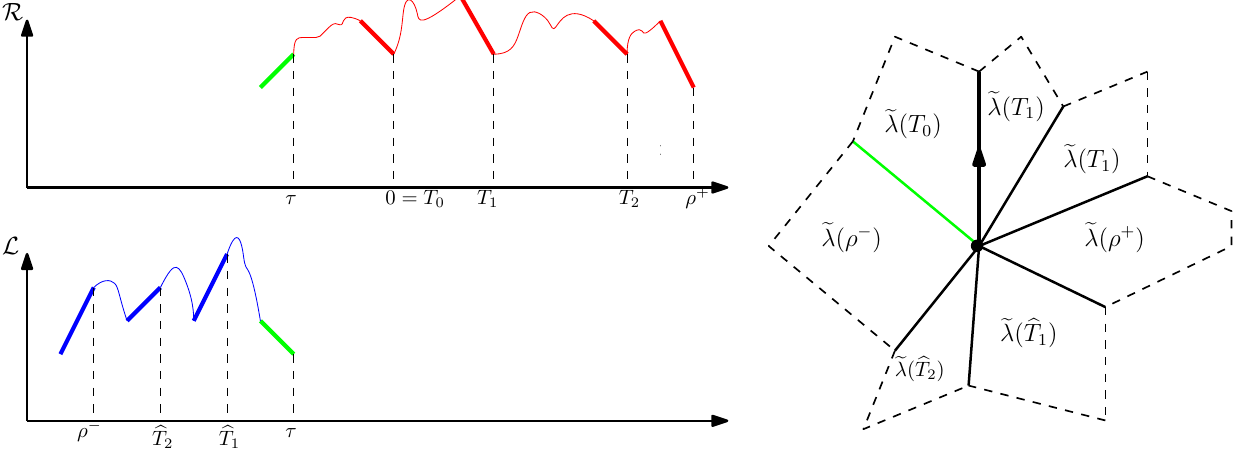}
	\caption[Illustration of the proof of Lemma~\ref{lem:inner}]{ 
		We use the notations in the proof of Lemma~\ref{lem:inner}.	\textbf{Left:} An illustration of the times $\tau,\rho^{\pm}, T_{\bullet},  \wh{T}_{\bullet}$.  The green step is $\Delta\cZ_\tau$.  \textbf{Right:} The corresponding picture for $\nb_0$. The arrowed edge is $e^0$. The green edge is  $\path(\tau)$. The faces of $\nb_0$ correspond to the blue and red steps in the left panel.  From Assertion~\ref{item:R} in Lemma~\ref{lem:LR}, the time $\tau$ is first time when $\spole(\path(\tau))=\spole(e^0)$.   Moreover, $\rho^+$ (resp. $\rho^{-}$) is the unique time $t$ such that $\wt\path(t)$ is a face and $\spole(e^0)$ is on the west (resp. east) boundary of $\wt\path(t)$ but is not a pole of $\wt\path(t)$.  Finally,  $\{\wt\path(T_i)\}_{\xi^-<i\le \xi^+}$ (resp. $\{\wt\path(\wh T_i)\}_{1 \le i\le \zeta}$) describes the set of faces where $\spole(e^0)$ is the south (resp. north) pole.
	}
	\label{fig:nb}
\end{figure}

So far there is no probability involved. 
	Now we prove a few probabilistic lemmas ensuring that almost surely  $M$ does not have certain  pathological behaviors.
	We will use the following notation.  For integers $m<n$, we call the union of the  east (resp., west) boundary of $M^{m,n,K}$ for all large enough $K$'s the  west (resp. east) boundary   of $M^{m, n}$.
	\begin{lem}\label{lem:inner}
		With probability $1$, there exists $n\in\N$ such that  $\spole(e^0)$, i.e., the tail of $e^0$, is not on the east or west boundary of $M^{-n,n}$.
\end{lem}
\begin{proof}
	Consider the following special times for $\cZ$:
	\allb
	&\tau :=\max\{ t\le 0:  \Delta\cZ_t=(-1,1)\;\textrm{and}\; \cR_t=\cR_0=0\} , \notag \\
	&\rho^+ :=\inf\{t\ge 0:\cR_t<0\} , \quad \op{and}\quad \rho^- :=\sup\{t<\tau: \Delta\cZ_t \not= (-1,1) \;\textrm{and}\; \cL_{t-1}\le \cL_{\tau-1} \} .
	\alle
	These times are illustrated in Figure~\ref{fig:nb}. Note that all three times are finite almost surely.  
	
	Let $K$ be so large that $M^{\rho^--1,\rho^+,K}$  is well defined and let $(M',f_0',\cO'):=(M^{\rho^--1,\rho^+,K}$, $f_0^{\rho^--1,\rho^+,K}$, $\cO^{\rho^--1,\rho^+,K})$
		Now we define a function $\wt\lambda: [\rho^-,\rho^+]_{\Z} \to \innV(\map')\cup \innF(\map')$ with $\wt\lambda$ as in Definition~\ref{def:sewing} but with a different interval of definition.
		For $t\in [\rho^-,\rho^+]_\Z$, if $\path(t-1)$ is the SE edge of $\npole(\path(t-1))$, set $\wt \path(t):=\spole(\path(t))$.
		Otherwise, let $\wt\path(t)$ be the unique  face where $\path(t-1)$ is a west edge. Here the notion of east/west/south/north is with respect to $(M',f_0',\cO')$.
		We define $\wt\lambda$ this way so that  Lemma~\ref{lem:LR}  can be adapted to this situation straightforwardly.
		
		Let $\nb_0$ be the subgraph  of $\map'$ consisting of all faces of $\map'$ with $\spole(e^0)$ on their boundaries,  and all vertices and edges of $\map'$ that are on the boundaries of such faces. 
		We claim that faces in $\nb_0$ all belong to $\innF(\map')$, which will imply Lemma~\ref{lem:inner} with $n=|\rho^-|\vee |\rho^+|$. 
	
	Let $\xi^+=\#\{t\in (0,\rho^+): \cR_t=0\}$  and $\xi^-=-\#\{t\in[\tau,0) : \cR_t=0 \}$.  Let   $\{T_i\}_{\xi^-  \le i  \le \xi^+}$ be the set of times in $[\tau,\rho^+]$ for which $\mcl R_t =0$, in increasing order. Note that $\tau=T_{\xi^-}$ and $T_0=0$. 
	Let $\zeta:=\#\{t\in (\rho^{-},\tau-1): \cL_t= \cL_{\tau-1} \}$ and let $\{\wh T_i\}_{1\le i  \le \zeta}$ represent this time  set in decreasing order.  
	Recall Lemma~\ref{lem:LR} and its proof. We make the following observations which imply the claim in the previous paragraph. We will justify them just below. 
	\begin{enumerate}
		\item \label{item:n1} $\{\path(T_i)\}_{\xi^-  \le i  \le \xi^+}$ is equal to the set of outgoing edges in $\nb_0$ in clockwise order.
		In particular,  $\path(\tau)$ is the NW edge of $\spole(\path(0))$.
		\item \label{item:n2}The cardinality of the set of  faces in $M'$ with $\spole(e^0)$  as south pole equals $\xi^+ -\xi^-$. When the set is nonempty, it equals $\{\wt\path(T_i)\}_{\xi^-<i\le \xi^+ }$.
		\item  \label{item:ne} 
		$\wt\path(\rho^+)$ is the unique face in $M'$ with $\spole(e^0)$ on its west boundary but not as one of its poles.
		\item \label{item:s} The cardinality of the set of faces  in $M'$ with $\spole(e^0)$ as  north pole equals $\zeta$. When the set is nonempty, it equals $\{\wt\path(\wh T_i)\}_{1\le i\le \zeta}$.
		\item  \label{item:w} 
		$\wt\path(\rho^-)$ is the unique face in $M'$ with $\spole(e^0)$ on its east boundary but not as one of its poles.
	\end{enumerate}
	See Figure~\ref{fig:nb} for an illustration of observations \ref{item:n1}-\ref{item:w}. Note that if $\Delta\cZ_0=(1,-1)$, then $\tau=0$. Otherwise $\tau=\inf\{s: \textrm{ $s$ and $0$ satisfy \eqref{eq:R}} \}$. In both cases $\spole(\path(\tau)) =\spole(e^0)$. Since $\Delta\cZ_\tau=(1,-1)$,  condition~\ref{item:path1} in the definition of $\path$ implies that $\path(\tau)$ is the NW edge of $\spole(\path(0))=\spole(e^0) $. Now $\{T_i\}_{\xi^-<i\le \xi}$ is the set of times $t$ satisfying $\tau=\inf\{s: \textrm{ $s$ and $t$ satisfy \eqref{eq:R}} \}$. Moreover, $\rho^+$ is the unique time $>\tau$ such that  $\tau$ and $\rho^+$ satisfy \eqref{eq:R} and $\cR_{\rho^+}<\cR_\tau$. This combined with Assertion~\ref{item:R} in Lemma~\ref{lem:LR} proves observations \ref{item:n1}-\ref{item:ne}. Recall Remark~\ref{rmk:reversal} and Lemma~\ref{lem:rev}. If we replace $\tau$ by $\tau-1$ and $\cR$ by the time reversal of $\cL$ in the above argument, we obtain observations~\ref{item:s}-\ref{item:w}. 
\end{proof}

\begin{lem}\label{lem:regular}
		The following occurs with probability 1. 
		Each vertex in $M$ has finite degree.  For each $r$, there exists $n=n(r)\in\N$ large enough such that $B_r(\spole(e^0);M)$ is a subgraph of $M^{-n,n}$ and has no vertices on the east or west boundary of $M^{-n,n}$.
	\end{lem}
	\begin{proof}
		By Lemma~\ref{lem:inner}, the degree of $\spole(e^0)$ is finite. Note that $e^0=\spole(\lambda(0))$.
		By the stationarity of $\cZ$,  Lemma~\ref{lem:inner} remains true if $\spole(e^0)$ is replaced by $\spole(\lambda(k))$ for each $k\in\Z$. 
		Let $v$ be an adjacent vertex of $\spole(e^0)$. Then by the proof of Lemma~\ref{lem:inner} there exists $k\in \Z$ such that $v=\spole(\lambda(k))$ or $v=\npole(\lambda(k))$. In the later case, we must  have $v=\spole(\lambda(k'))$
		where $k'=\inf\{t>k: \cL_t=\cL_k-1 \}$. Therefore we can always find $k\in\Z$ such that $v=\spole(\lambda(k))$. Thus there almost surely  exists $n_v\in\N$ such that $v\in \cV(M^{-n_v,n_v})$ and $v$ is not on the east or west boundary of $M^{-n_v,n_v}$. 
		Let $n(1)$ be the maximum over all such $n_v$'s. Then $n(1)$ satisfies the desired property for 
		$B_1	(\spole(e^0);M)$. Iterating this argument for each vertex in $B_1	(\spole(e^0);M)$ gives the desired result for $B_2(\spole(e^0);M)$. The general statement follows from further iterations.
	\end{proof}
	By Lemma~\ref{lem:regular}, given each $v\in\cV(M)$, by considering  $M^{-n,n}$ with large enough $n$, 
	we can talk about the SE and NW edges of $v$ as in  the finite volume case.
	\begin{lem}\label{lem:face}
		For each edge $e$ on $M$, there exists $n\in\N$ such that $e\in \cE(M^{-n,n})$ and $e$ has no endpoints on the east or west boundary of $M^{-n,n}$. In particular, there exists $k\in[n,n]_\Z$ such that $e=\lambda (k)$ and the two faces with $e$ on their boundaries belong to $\innF(M^{-n,n})$.
	\end{lem}
	\begin{proof}
		The existence of $M^{-n,n}$ in the first statement immediately follows from Lemma~\ref{lem:regular}.  By the definition of east and west boundaries, the two faces in $\cF(M^{-n,n})$ with $e$ on their boundaries belong to $\innF(M^{-n,n})$. Choose $K$ large enough so that $M^{-n,n,K}$ is well defined.
		By Lemma~\ref{lem:LR} and  the definition of $\cZ^{-n,n,K}$, for each $t<\cR_{-n}+K$ there is no face in $\innF(M^{-n,n,K})$ with  $\lambda^{-n,n,K}(t)$ being an edge. Therefore,  $\lambda^{-n,n,K}(t)$ must be on the west boundary of $M^{-n,n,K}$.
		Similarly, $\lambda^{-n,n,K}(t)$ must be on the east boundary of $M^{-n,n,K}$ for each $t>\cR_{-n}+K+2n$. Therefore there exists $k\in[-n,n]_\Z$ such that $e=\lambda^{-n,n}(k)=\lambda(k)$.
	\end{proof}
	By Lemma~\ref{lem:face}  $\lambda$ is a bijection between $\Z$ and $\cE(M)$. 
	Moreover, the degree of each face of $M$ is bounded. Therefore, given a face of $M$ we can define its poles and east/west boundaries as in the finite map case.
	Now for each $t\in \Z$, if $\path(t-1)$ is the SE edge of $\npole(\path(t-1))$, set $\wt \path(t):=\spole(\path(t))$.
	Otherwise, let $\wt\path(t)$ be the unique  face where $\path(t-1)$ is a west edge.  Then $\wt\lambda:\Z\to \cV(M)\cup\cF(M)$ is a bijection.  For each $t\in\Z$, both Items~\ref{item:path1} and~\ref{item:path2} in Definition~\ref{def:sewing} hold in this infinite volume setting. The following fact is an immediate consequence of this extension of definition.
	\begin{lem}\label{lem:inn-face}
		For $n\in\N$ and $K$ large enough such that $M^{-n,n,K}$ is well defined,  we have that $\wt \lambda([-n,n]_\Z)\cap \cF(M)=\innF(M^{-n,n,K})$.
	\end{lem}

We adapt the notion of Benjamini-Schramm convergence
\cite{benjamini-schramm-topology} to the setting of oriented rooted graphs. Given planar maps $G$ and $G'$ with oriented root edges $e$ and $e'$, respectively, the \emph{Benjamini-Schramm distance}~\cite{benjamini-schramm-topology} from $(G,e)$ to $(G',e')$ is defined by 
\[
d_{\mathrm{loc}} = \inf\{2^{-r}:r\in \N, B_r(o;G)\simeq B_r(o';G')  \} ,
\]
where $o$ and $o'$ are the tail of $e$ and $e'$ respectively and  $B_r(e;G)\simeq B_r(e';G')$ means that the two metric balls are isomorphic as oriented rooted planar maps.   Then $d_{\op{loc}}$ defines a metric on the space of finite rooted planar maps. Let $\mathfrak{M}$ be the Cauchy completion of the space of finite oriented rooted planar maps under $d_{\mathrm{loc}}$. Elements in $\mathfrak M$ which are not finite  are  infinite rooted oriented planar maps. 
The following proposition asserts that $M$ is the Benjamini-Schramm limit of the  uniform finite bipolar-oriented maps viewed from a typical edge.
\begin{prop}\label{prop:BS}
	Let $(\map(n), \cO(n))$ be a uniformly chosen bipolar-oriented map with $n$ edges, and let $e(n)$ be an edge  uniformly chosen in $\cE(\map(n))$. Then as $n\rta\infty$, $(\map(n),e(n),\cO(n) )$ converges in law to $(\map,e^0,\cO)$ defined above in the  Benjamini-Schramm topology. 
	
\end{prop}

\begin{proof}
	Let $(\xi_1,\xi_2)$ be independent of $\cZ$ so that $\P[(\xi_1,\xi_2)=(i,j)]=2^{-i-j-2}$ for all $(i,j)\in \Z_{\ge 0}\times \Z_{\ge 0}$.
	Let $E_n$ be the event that $\{(\xi_1,0)+\cZ_t\}_{0\le t\le n-1}$ is a walk starting at $(\xi_1,0)$, ending at $(0,\xi_2)$ and staying in $ \Z_{\ge 0}\times \Z_{\ge 0}$ during $[0,n-1]$. 
	Then  conditioning on $E_n$ and applying the sewing procedure to $\{(\xi_1,0)+\cZ_t\}_{0\le t\le n-1}$, we obtain a sample of $(\map(n), \cO(n))$ (see \cite[Remark 2]{kmsw-bipolar}).  
	Let $\path_n$ be the edge-valued function associated with this sewing procedure and let $U_n$ be a uniform sample from $[0,n-1]_{\Z}$. Then $(\map(n), \path(U_n), \cO(n))$ is a sample of $(\map(n), e(n),\cO(n))$.
	
	Now let $\cZ^n(\bullet)$ be sampled from the conditional law of $\cZ(\bullet+U_n)-\cZ(U_n)$ given $E_n$.  Standard results for random walk in cones (see, e.g., \cite[Theorem 4]{dw-limit} and\cite[Theorem 1]{dw-cones})  yield that as long as  $\P[E_n]>0$, it holds that $\P[E_n]$ decays sub-exponentially in $n$.  Therefore, Cramer's large deviation theorem implies that for all $N>0$ the walks  $\cZ^n$ and $\cZ$ can be coupled so that their restrictions to $[-N,N]$ agree with probability $1-o_n(1)$ (see \cite[Section 4.2]{shef-burger} for a similar application of Cramer's theorem). Now Proposition~\ref{prop:BS} follows from Lemma~\ref{lem:regular}.\qedhere
\end{proof}

	For each $\ell\in \Z$,  let $M_{\ell,\infty}$ be the submap of $M$ obtained by taking the union $\bigcup_{n>\ell}M^{\ell,n}$, rooted at the edge $\lambda(\ell)$. Then $M_{\ell,\infty}$ has a bi-infinite boundary, which coincides with the west boundary of  $M^{\ell,n}$ for each integer $n>\ell$.  List these boundary edges as $\{b_\ell(i)\}_{i\in\Z}$ such that $b_\ell(0)=\lambda(\ell)$ and $\npole(b_\ell(i))=\spole(b_\ell(i+1))$ for each $i\in\Z$. 
	By  Lemma~\ref{lem:bdy}, we have the following. 
	\begin{lem}\label{lem:half}
		In the setting right above,  almost surely,  $b_\ell(i+1)$ is the NW  edge of $\npole(b_\ell(i))$ for each $i\ge 0$, and  $b_\ell(i-1)$ is the SE  edge of  $\spole(b_\ell(i))$ for each $i\le 0$.
	\end{lem}
	For each $\ell\in\Z$, we can also consider $M_{-\infty, \ell}=\bigcup_{n<\ell}M^{n,\ell}$, where $\lambda(\ell)$ is assigned to be the root edge. Then $M_{-\infty,\ell}$ and $M_{\ell,\infty}$ share the same boundary.
	\begin{definition}\label{def:UIBOM}
		We call $(\map,e^0,\cO)$ the \emph{uniform infinite bipolar-oriented map} ($\UBOM$). We call $\{M_{-\infty,\ell}\}_{\ell\in\Z}$ and $\{M_{\ell,\infty}\}_{\ell\in\Z}$  the forward and backward \emph{sewing procedure} of $(\map,e^0,\cO)$, respectively.
	\end{definition}
	\begin{prop}\label{prop:bijection}
		For each $\ell\in\Z$, $\{\cZ_k-\cZ_\ell\}_{k\le \ell}$ and $(M_{-\infty,\ell}, \lambda(\ell),\cO|_{\cE(M_{-\infty,\ell})})$ almost surely determine each other. The same statement holds for $M_{\ell,\infty}$ with $\{\cZ_k-\cZ_\ell\}_{k\le \ell}$  replaced by $\{\cZ_k-\cZ_\ell\}_{k\ge\ell}$. 
		In particular, $\cZ$ and $(M,e^0,\cO)$ almost surely determine each other.
	\end{prop}
	\begin{proof}
		For $n>\ell$, by definition $M^{ \ell , n}$ is determined by $\{\cZ_k-\cZ_\ell\}_{k\in [\ell,n]_\Z}$.  Therefore $\{\cZ_k-\cZ_\ell\}_{k\ge \ell}$ determines
		$(M_{\ell,\infty}, \lambda(\ell),\cO|_{\cE(M_{\ell,\infty})})$
		Now given  $(M_{\ell,\infty}, \lambda(\ell),\cO|_{\cE(M_{\ell,\infty})})$,
		let $f$ be the face  with $\lambda(\ell)$ on its west boundary. According to whether $\npole(\lambda(\ell))=\npole(f)$, we can decide $\lambda(\ell+1)$.   By Lemma~\ref{lem:half}, $(M_{\ell+1,\infty}, \lambda(\ell+1),\cO|_{\cE(M_{\ell+1,\infty})})$
		is determined by $(M_{\ell,\infty}, \lambda(\ell),\cO|_{\cE(M_{\ell,\infty})})$. By induction, $(M_{k,\infty}, \lambda(k),\cO|_{\cE(M_{k,\infty})})$ is determined by $(M_{\ell,\infty}, \lambda(\ell),\cO|_{\cE(M_{\ell,\infty})})$ for all $k> \ell$. This allows us to recover $\{\Delta\cZ_t\}_{t>\ell }$ hence 	$\{\cZ_k-\cZ_\ell\}_{k\in [\ell,n]_\Z}$.  Using the time reversal trick in Lemma~\ref{lem:rev},  we obtain the same result for $M_{-\infty,\ell}$ and $\{\cZ_k-\cZ_\ell\}_{k\le \ell}$.
	\end{proof}
	
	Given $v\in \cV(\map)$, we define $\nb_v$ to be the subgraph of $\map$ consisting of all faces of $\map$ with $v$ on their boundaries and all vertices and edges of $\map$ on the boundaries of such faces. In the proof of Lemma~\ref{lem:inner}, if $M'$  is identified as a subgraph of $M$, then $\nb_{\spole(e^0)}$ equals $\nb_0$ defined there. 
For $f\in \cF(\map)$, let $\nb_f$ be subgraph of $\map$  given by the union of $\nb_v$'s  where $v$ is a  boundary vertex of $f$, i.e., $\nb_f$ is the set of all faces of $\map$ which share a vertex with $f$ and all vertices and edges which lie on the boundaries of such faces.
Let  
\eqbn
|\nb_x|:=\#\cV(\nb_x)+\#\cF(\nb_x)+\#\cE(\nb_x) ,\quad \forall x\in \cV(\map)\cup\cF(\map).
\eqen
We  conclude this section by a degree-bound type estimates for the $\UBOM$.  
\begin{lem} \label{lem-12-deg}
	For all fixed $i\in \Z$, set $\nb_i:=\nb_{\spole(\path(i))}$ and  $\wt \nb_i:=\nb_{\wt\path(i)}$. Then the laws of both $|\nb_i|$ and $| \wt \nb_i| $ do not depend on $i$ and have exponential tails. In particular, the degrees of the two endpoints of the root edge of the $\UBOM$ have exponential tails.
\end{lem}
For the proof of Lemma~\ref{lem-12-deg}, we need the following elementary estimate.
\begin{lem}\label{lem:LD}
	Suppose $\{X_i\}_{i\ge 1}$ is a sequence of i.i.d.\ nonnegative random variables such that 
	$\E [ e^{bX_1} ]<\infty$ for some $b>0$. 
	Let $\xi$ be a nonnegative  random variable (not necessarily independent from the $X_i$'s) with an exponential tail. Then there exists $c>0$ only depending on the law of $X_1$ and $\xi$ so that
	\[
	\P\Big[ \sum_{i=1}^\xi X_i \ge n\Big] \le e^{-c n} .
	\] 
\end{lem}
\begin{proof}
	By a union bound, for $\delta > 0$, 
	\begin{align*}
	\P\Big[ \sum_{i=1}^\xi  X_i \ge n\Big] \le \P [\xi\ge \delta n ]+ \P\Big[ \sum_{i=1}^{\lfloor \delta n \rfloor} X_i \ge n\Big].
	\end{align*}
	By choosing $\delta$ small enough according to the law of $X_1$,  the lemma follows from Cramer's Large Deviation Theorem.  
\end{proof}

\begin{proof}[Proof of Lemma~\ref{lem-12-deg}]
	The stationarity of $\{\nb_i \}_{i\in \Z}$ and  $\{\wt \nb_i \}_{i\in \Z}$ follows from the stationarity of $\cZ$.
	So it suffices to focus on $i=0$ in the rest of the proof.
	
	To deal with $\nb_0$, we retain the notation used in the proof of Lemma~\ref{lem:inner}. By the Markov property of $\cZ$, all of $\xi^-,\xi^{+}$ and $\zeta$ have geometric distributions. Moreover, conditioning on $\xi^-\neq 0$, we have that $\{\Delta\cZ_{T_i} \}_{\xi^- <i\le 0}$ are i.i.d and have exponential tails.  The same also holds for  $\{\Delta\cZ_{T_i} \}_{1\le i\le \xi_+}$ and $\{\Delta\cZ_{\wh T_i} \}_{1\le i\le \zeta}$. (However, the laws of $\Delta\cZ_{T_{-1}}$ and $\Delta\cZ_{T_1}$ are not the same.)  Now  observations \ref{item:n1}-\ref{item:w} in the proof of Lemma~\ref{lem:inner} combined with Lemma~\ref{lem:LD} and ~\eqref{eq:step1} imply that $|\nb_0|$ has an exponential tail.
	We obtain the  same result for $\{\nb_{\npole(\path(i))}\}_{i\in \Z}$ by reversing the time direction and apply Lemma~\ref{lem:rev}.
	
	It remains to treat $\wt\nb_0$.
		Set $t_0=0$. For a positive integer $k$, we inductively define $t_k=\inf\{ t>t_{k-1} : \cL_t=\cL_{t_{k-1}}-1 \}$.
		Then $\{t_k\}_{k\ge 0}$ are stopping times for the filtration $\{\cF_\ell=\sigma(\{\cZ_i\}_{i\le \ell})\}_{\ell\in\Z}$. Moreover, $\lambda(t_k)$ is the NW edge of $\lambda(t_{k-1})$ for $k\ge 1$.
		Let $\nb^{+}_k=(\nb_{\spole(\path(t_{k-1}))} \cup \nb_{\spole(\path(t_k)} )\cap \map^{t_{k-1},t_k}$.
		Then $\{\nb^{+}_k\}_{k\ge 1}$ is a sequence of independent identically distributed random variables.
		Recall the $M_{0,\infty}$ in Definition~\ref{def:UIBOM}.   On the event $v\in\cV(\map)$, $\wt \nb_0=\nb_0$ is handled in the previous paragraph.
		On the event  $\wt\path(0)\in \cF(\map)$, let $\ell$ be  the number of edges on the east boundary of $\wt\path(0)$, which has an exponential tail.
		Then   $\{\nb^+_{k} \}_{1\le k\le \ell}$ and $\nb_{\spole(\lambda(t_\ell))}\cap M_{0,\infty}$ cover $\wt \nb_{0}\cap \map_{0,\infty}$.  
		Let $|\nb^{+}_k|=\#\cV(\nb^+_k)+\#\cF(\nb^+_k)+\#\cE(\nb^+_k)$. Since $|\nb^{+}_1|\le |\nb_{\spole(e^0)}|+|\nb_{\npole(e^0)}|$, we see that   $|\nb^{+}_k|$ has an exponential tail.  On the other hand, $\nb_{\spole(\lambda(t_\ell))}\cap M_{0,\infty}$ has the same law as $\nb_{0}\cap M_{0,\infty}$, whose size has an exponential tail.  By Lemma~\ref{lem:LD} and  reversing the time as in Lemma~\ref{lem:rev},  we can bound $|\wt\nb_0|$ by the sum of two random variables with exponential tails, thus proving Lemma~\ref{lem-12-deg}.
\end{proof}

\begin{remark}
	A bipolar orientation on a map can be associated with a bipolar orientation on its dual map via rotating the orientation of each edge by 90 degrees.  Therefore the $\UBOM$ and its dual map agree in law. Thus bounds on $|\wt\nb_0|$ immediately follow from bounds on $|\nb_0|$.   However, our proof of Lemma~\ref{lem-12-deg} is more robust and works for the more general setting in Section~\ref{subsub:Schnyder}.
\end{remark}

\subsubsection{Proof of Theorem~\ref{thm-map-count} for the UIBOM}\label{subsub:bi-metric}
From the discussion in Section~\ref{subsub:ubom},
Assertions~\ref{item:R} and~\ref{item:L} in Lemma~\ref{lem:LR} hold for the $\UBOM$ for all $s,t\in \Z$ with $s<t$. We still refer to this fact as Lemma~\ref{lem:LR} for the $\UBOM$. Let $\cQ :=\cQ(\map)$ be the radial map associated with $\map$, i.e., the quadrangulation obtained from $\map$ by connecting the center of each face of $\map$ with its adjacent vertices and removing all edges of $\map$. We identify $\cV(\map)\cup\cF(\map)$ with $\cV(\cQ)$  so that $\wt \path$ is a bijection between $\Z$ and $\cV(\cQ)$.
By Lemma~\ref{lem:LR} for the $\UBOM$, we have the following.
\begin{prop}\label{prop:bipolar}
	The graph $\cQ$ is isomorphic via $i \leftrightarrow \wt\path(i)$ to the graph $\wh \cH$ defined as follows. The vertex set of $\wh\cH$ is $\Z$. For $s,t\in \cV(\wh\cH)$ with $s <t$,  the vertices $s$ and $t$ are connected by an edge in $\wh\cH$ if and only if either 
	\eqb \label{eq:walk-adjacency3}
	(\cL_{s-1}-1) \vee \cL_{t} < \min_{s-1< j\le t-1} \cL_j \quad \op{or} \quad
	\cR_{s-1} \vee (\cR_{t}-1) < \min_{s\le j<t} \cR_j.
	\eqe
\end{prop}
\begin{remark}\label{rmk:Q}
	Although it is not needed later, we note that $\cQ$ is a quadrangulation with boundary which can be oriented as follows. Each edge $e\in \cE(\cQ)$ must be of the form $\{v,f\}$ with $v\in \cV(\map)$ and $f\in \cF(\map)$. Orient $e$ from $f$ to $v$ if $v=\npole(f)$ or $\spole(f)$ and from $v$ to $f$ otherwise. This combinatorial object is called a 2-orientation on quadrangulation (see e.g. \cite[Section 5.2]{Felner}) because every vertex has out-degree  2.  This constructions gives a general bijection between bipolar orientations and 2-orientations. When $\map=\Z^2$, the 2-orientation on $\cQ$ is called a 6-vertex configuration. This special case is considered  in~\cite{kmsw-6vertex}, based on which a connection between $\SLE_{12}$ and the 6-vertex model is found. 
\end{remark} 
To use Theorem~\ref{thm-sg-map-dist} and Remark~\ref{remark-other-adjacency} to prove Theorem~\ref{thm-map-count} , we define a planar map $\mcl H = \mcl H(\mcl Z)$ via the following discrete analogue of \eqref{eqn-bm-inf-adjacency}.  
The vertex set of $\mcl H$ is $\cV(\cH) =\Z$, and for $s,t\in \cV(\cH)$ with $s <t$, we declare that $s$ and $t$ are connected by an edge in $\mcl H$ if and only if either 
\eqb \label{eq:walk-adjacency2}
(\cL_{s-1}-1) \vee \cL_{t} < \min_{s-1< j\le t-1} \cL_j \quad \op{or} \quad
\cR_{s-1} \vee (\cR_{t}-1) < \min_{s\le j< t} \cR_j \quad \op{or} \quad t-s= 1.
\eqe

By Proposition~\ref{prop:bipolar}, the radial map $\cQ$ can be identified with a submap of $\cH$ via  $i \leftrightarrow \wt\path(i)$. Moreover, $\cH$ can be recovered from $\wh\cH$ in Proposition~\ref{prop:bipolar} by adding edges $\{t,t+1\}$ for all $t\in \Z$.  Recall Definition~\ref{def:sewing} of $\path$ and $\wt\path$ and its infinite volume extension in Section~\ref{subsub:ubom}. We have that $\cH$ is isomorphic via $\wt\path$ to the graph $\ol\cQ$ such that $\cV(\ol\cQ) = \cV(\cQ)$ and $v,w\in \cV(\ol\cQ)$ are connected in $\ol\cQ$ if and only if one of the following three possibilities occurs:
\begin{enumerate}
	\item $v$ and $w$ are connected in $\cQ$; 
	\item\label{item:2vertice} $v,w$ are two vertices of $\map$ and are connected by an edge $e\in \cE(\map)$ so that $e$ is the NW edge of $\spole(e)$ and the SE edge of $\npole(e)$
	(see Item \ref{item:path1} in Definition~\ref{def:sewing});
	\item \label{item:2face} $v,w$ are two faces of $\map$ which share an edge $e$ so that $\spole(e)$ is the south pole of one face and $\npole(e)$ is the north pole of the other face (see Item \ref{item:path2} in Definition~\ref{def:sewing}).
\end{enumerate}

For $n\in\BB N$, let $\mcl H_{n}$ be the subgraph of $\cH$ whose vertex set is $[-n,n]_{\BB Z}$, with two vertices connected by an edge in $\cH_{n}$ if and only if they are connected by an edge in $\cH$.  Let $\ol\cQ_n$  be the image of $\cH_n$ under the isomorphism $\wt\path$. Let $M_{n}$ be the smallest connected subgraph of $M$ containing each vertex in $\{\wt\path(t)\}_{t\in[-n,n]_\Z}$ and all vertices and edges on the boundary of each face in $\{\wt\path(t)\}_{t\in[-n,n]_\Z}$.

One can give a more direct description of $M_n$. Recall Definition~\ref{def:UIBOM}. For $n\in\N$, let	$\{b_n(i)\}_{i\in\Z}$ and $\{b_{-n}(i)\}_{i\in\Z}$  be as in Lemma~\ref{lem:half} with $\ell=n$ and $\ell=-n$, respectively.	Define
	\begin{align*}
	\tau_n&=\inf\{t\ge n: \lambda(t)=  b_{-n}(i)  \textrm{ for some }i>0 \} \qquad &&\textrm{and}\qquad v_{+,n}=\spole(\lambda(\tau_n));\\ 
	\tau_{-n}&=\sup\{t\le -n: \spole(\lambda(t))   =  \spole (b_n(i))   \textrm{ for some }i\le 0 \} \qquad &&\textrm{and}\qquad v_{-,n}=\spole(\lambda(\tau_{-n})).
	\end{align*}
	Both $v_{+,n}$ and $v_{-,n}$ are on the boundary of both $M_{-n,\infty}$ and $M_{-\infty,n}$. 
	The segments between   $v_{+,n}$ and $v_{-,n}$ on these two boundaries  enclose a bipolar oriented map $(\wh M_n,\wh f_n,\wh \cO_n)$. In other words.
	$\wh M_n$  is the unique finite  connected component of $M^{-n,n}$ after removing $\{v_{+,n} , v_{-,n} \}$.
	Moreover, $\cO|_{\cE(\wh \map_n )} =\wh\cO_n$,  $v_{+,n}=\npole(\wh M_n)$ and $v_{-,n}=\spole(\wh M_n)$. 
	\begin{lem}\label{lem:submap}
		$M_n$ is  isomorphic to  $\wh M_n$ when $\wh M_n$ is viewed as a graph.  
	\end{lem}
	\begin{proof}
		We first show that $\wh M_n$  is a subgraph of $M_n$. 
			Since both $\wh M_n$ and $M_n$ are connected subgraphs of $M$, it suffices to show that $\cE(\wh M_n) \subset \cE(M_n)$.
		Choose $K$ large enough so that $M^{n,n,K}$ is well defined. Given $e\in \cE(\wh M_n)$, if $e$ is on the boundary of a  face $f\in\innF(M^{n,n,K}) $, by Lemma~\ref{lem:inn-face} there exists $i\in[-n,n]_\Z$ such that $\wt\lambda(i)=f$. Therefore $e\in\cE(M_n)$. 
		Now we suppose $e$ is not on the boundary of any face in $\innF(M^{-n,n,K})$. 
		Then $e$ must be on the boundary of both $M_{-n,\infty}$ and $M_{-\infty,n}$.
		Let $k\in\Z$ be such that $\lambda(k)=e$. 
		If $k\ge n$, by the minimality of $\tau_n$, we must have $\tau_n\le k$.  This contradicts that $v_{+,n}=\spole(\lambda(\tau_n))$ is the north pole of $\wh M_n$. Therefore $k<n$. Similarly, if $k<-n$, we must have $\tau_{-n}\ge k$ which contradicts that $v_{-,n}=\spole(\lambda(\tau_{-n}))$ is the south pole of $\wh M_n$.   Therefore $k\ge -n$. Let $i$ be such that $b_{-n}(i)=e$. Then we must have $i\ge 0$. 
		Therefore $e$ is the NW edge of $\spole(e)$ and hence $\spole(e)=\wt \lambda(k)\in \cV(M_n)$. On the other hand, $\lambda(k+1)$ must be the NW edge of $\npole(e)$. Namely, $\lambda(k+1)=b_{-n}(i+1)$. Otherwise $\wt\lambda(k+1)$ would be a face with $e$ on its boundary. Therefore $\npole(e)\in \cV(M_n)$, which  concludes the proof that $\cE(\wh M_n)\subset\cE(M_n)$.
		
		It remains to show that $M_n$ is a subgraph of $\wh M_n$. By the definition of $M_n$ above, it suffices to show that $\{\wt\path(t)\}_{t\in[-n,n]_\Z}\subset \cV(\wh M_n)\cup \innF(\wh M_n)$. By Lemma~\ref{lem:inn-face}, $\{\wt\path(t)\}_{t\in[-n,n]_\Z}\cap\cF(M)=\innF(M^{-n,n})= \innF(\wh M_n)$.  
		Now suppose $j\in\Z$ is such that $v=\wt \lambda(j) \in \cV(M^{-n,n,K}) \setminus \cV(\wh M_n)$. Then  $v$ must be in one of the two infinite connected of $M^{-n,n}$ after removing $\{ v_{+,n} , v_{-,n}\}$. Therefore $j>\tau_n\ge n$ or $j< \tau_{-n}\le -n$. This gives  $\{\wt\path(t)\}_{t\in[-n,n]_\Z}\cap \cV(M)\subset \cV(\wh M_n)$ and concludes the proof.\qedhere
	\end{proof}
	
	From now on we identify the graph $M_n$ with the   submap $\wh M_n$ of $M$ in Lemma~\ref{lem:submap}.
Let $\iota_n : M_n\rta M$ be the inclusion map.   
Define $\psi : \BB Z\rta \cV (\map)$ by $\psi(i)=\spole(\path(i))$
and set $\psi_n := \psi|_{[-n,n]_{\BB Z}}$. 
Define 
\[
\phi_n(v):= \wt\path^{-1}(v) \1_{v\neq \spole(e^0)}\qquad \qquad\qquad\qquad\qquad\qquad\textrm{for $v\in \cV(\map_n)$ with $ |\wt\path^{-1}(v)|\le n$.}
\] 
To define $\phi_n(v)$ for the case when $|\wt\path^{-1}(v)|> n$ we proceed as follows. 
Suppose $v\in \cV(M_n)$ and  $\cF(\nb_v)\cap \cF(\map_n)=\emptyset$.  Then $v$ cannot be on any inner face of $M_n$.
	Here recall that $\cF(\nb_v)$ is a subset of $\cF(M)$ and note that the outer face of $M_n$ is not a face in $M$. So $\cF(\nb_v)\cap \cF(M_n)$ is just $\cF(\nb_v)\cap \innF(M_n)$. 
	By the first paragraph in the proof of Lemma~\ref{lem:submap}, $v\in \{\wt\path(t)\}_{t\in [-n,n]_\Z}$.   Therefore, for $v\in \cV(\map_n)$ with $|\wt\path^{-1}(v)|>n$, we must have $\cF(\nb_v)\cap \cF(\map_n)\neq \emptyset$. In this case, 
	we let  $\phi_n(v)$  be the integer with the smallest absolute value for which
	$\wt \lambda(\phi_n(v)) \in  \cF(\nb_v) \cap \cF(M_n)$  (in the case of a tie, we choose  $\phi_n(v)$ to be positive).  Now we have a function $\phi_n:\cV(M_n)\to[-n,n]_\Z$.
	Since $\wt\path(0) \in \cF(\nb_0)\cap \cF(\map_n)$ when $\wt\path(0)\in \cF(\map)$, we see that $\phi_n(\spole(e^0)) = 0$. Therefore $\psi_n,\phi_n,\map_n$ satisfy \eqref{eqn-peano-functions}.
\begin{proof}[Proof of Lemma~\ref{lem-ball-layers} in case~\ref{item-kappa12}]
	We first prove that under the condition~\eqref{eqn-walk-inf-condition0} of Lemma~\ref{lem-ball-layers}, there does not exist  $v\in\cV(\bdy M_n)$ such that $k=\phi_n(v)\in[-m,m]_\Z$ and $|\wt\lambda^{-1}(v)|\le n$. 
	Suppose by way of contradiction that there exists such a $v$. If $v$ is on the west boundary of $M_n$, by the definition of  $\phi_n$, we have $v=\wt\lambda(k)$ and $\lambda(k)$ is the NW edge of $v$. If $v$ is on the west boundary of $M_n$ then there exists $i>0$ such that $\lambda(k)=b_{-n}(i)$. (Recall $b_{-n}$ in Lemma~\ref{lem:half}.)  
	Therefore, by Lemma~\ref{lem:NW} and Lemma~\ref{lem:half},
		 $k=\inf\{t>-n: \cL_t=\cL_{-n}-i\}$. Since $k\in [-m,m]_{\BB Z}$, this contradicts  the assumption~\eqref{eqn-walk-inf-condition0} for $\cL$. 
	Similarly, if $v$ is on the east boundary of $M_n$, 
	by the time reversal trick in Remark~\ref{rmk:reversal} and Lemma~\ref{lem:rev},  there exists $i<0$ such that $\lambda(k-1)=b_{n}(i)$ and  $k-1=\sup\{t<n: \cR_t=\cR_n+i\}$.  This contradicts  the assumption~\eqref{eqn-walk-inf-condition0} for $\cR$.  
	
It remains to show that there exists no  $v\in\cV(\bdy M_n)$ such that $k=\phi_n(v)\in[-m,m]_\Z$ 	and $|\wt\lambda^{-1}(v)|>n$. Suppose by way of contradiction that there exists such a $v$. Then  $\wt \lambda(k)$ is a face in $\innF(M_m)$ with $v$  on its boundary. Suppose $\wt\lambda^{-1}(v)>n$. Then $v$ must be on the east boundary of $M_n$. Thus the north pole of $\wt\lambda(k)$, which is $\npole(\lambda(k-1))$, must be  on the east boundary of $M_n$. Let $\ell=\inf\{t>k-1: \cL_t=\cL_{k-1}-1 \}$. By Lemma~\ref{lem:NW}, $\lambda(\ell)$ is the NW edge of $\npole(\lambda(k-1))$. Since $\wt\lambda^{-1}(v)>n$, we must have $\ell>n$. Therefore $\cL_{k-1}= \min_{t\in[k-1,n]_\Z}\cL_t$,  which  contradicts~\eqref{eqn-walk-inf-condition0} for $\cL$.  Suppose $\wt\lambda^{-1}(v)<-n$. Then similarly,  $v$ and $\spole(\lambda(k))$ must be on the east boundary of $M_n$. This gives $\sup\{t<k: \cR_t=\cR_{k}-1\}<-n$, which contradicts~\eqref{eqn-walk-inf-condition0} for $\cR$.
\end{proof}
 
\noindent Recall the notations above Lemma~\ref{lem-12-deg}.  The following fact  follows from Lemma~\ref{lem:inner} and the definition of $\ol\cQ_n$. 
\begin{lem}\label{lem:ad}
	Suppose $v\in \cV(\map_n)$ and $f_1,f_2\in \cF(\nb_v)\cap \cF(M_n)$. Then $f_1,f_2$ are connected in $\ol\cQ_n$ only using vertices and faces in $\nb_v$.  (Recall the identification $\cV(\map)\cup\cF(\map)=\cV(\ol\cQ)$.)
\end{lem}
\begin{proof}
	We leave the reader to check the following fact. Write $\cF(\nb_v)$ as $\{\wt \path(i_k): 1\le k \le \#\cF(\nb_v)\}$  where $i_1<\cdots < i_{\#\cF(\nb_v)}$.
	Then for any $1\le k \le  \#\cF(\nb_v)-1$,   if $\wt \path(i_k)$ and $\wt \path(i_{k+1})$ are both in $\cF(\map_n)$ but do not satisfy condition~\ref{item:2face} in the definition of $\ol\cQ$, then they are connected in $\cQ_n$ only using vertices and faces in $\nb_v$. These facts together with the definition of $\ol\cQ_n$ imply the statement of the lemma.
\end{proof}

From here on the proof of Theorem~\ref{thm-map-count} in this case  is very similar to the case of the infinite spanning tree-decorated map (Section~\ref{sec-kappa8-proof}) but there are minor technical differences so we give the complete argument.

Let $C_1  >  0$ be the constant from Theorem~\ref{thm-sg-map-dist} for $A$ in Theorem~\ref{thm-map-count} and for $\mcl Z$ above.  As in the spanning tree case, by the isomorphism (modulo multiplicity) $\wt\path$ between $\cH$ and $\ol\cQ$ plus Remark~\ref{remark-other-adjacency}, we see that with probability at least $1-O_n(n^{-A})$, the following is true 
\begin{enumerate}[label=(\Alph{enumi})]
	\item For each $i_1,i_2 \in  [-n , n]_{\BB Z}$ with $\wt\path(i_1) \sim \wt\path(i_2)$ in $\ol\cQ$, there is a path $\wt P_{i_1,i_2}^{\mcl G}$ from $i_1$ to $i_2$ in $\mcl G _{n}$ with $| \wt P_{i_1,i_2}^{\mcl G}| \leq C_1 (\log n)^3$; and each $j \in [-n , n]_{\BB Z}$ is contained in at most $ C_1 (\log n)^6 $ of the paths $\wt P_{i_1,i_2}^{\mcl G}$. \label{item-12-path-G}
	\item For each $i_1,i_2 \in  [-n , n]_{\BB Z}$ with $i_1 \sim i_2$ in $\mcl G$, then there is a path $\wt P_{i_1,i_2}^{\ol\cQ}  $ from $\wt\path(i_1)$ to $\wt\path(i_2)$ in $\ol\cQ_{n}$ with $| \wt P_{i_1,i_2}^{\ol\cQ}| \leq C_1 (\log n)^3$; and each $j  \in [-n , n]_{\BB Z}$ is contained in at most $ C_1 (\log n)^6 $ of the paths $\wt P_{i_1,i_2}^{\ol\cQ}$. \label{item-12-path-T}
\end{enumerate}
By basic properties of Brownian motion, $\deg(i;\cG)$  has an exponential tail. See \cite[Lemma~2.2]{gms-tutte} for a proof. Combined with Lemma~\ref{lem-12-deg}, a union bound implies that  there exists $C_2=C_2(A)$ such that  with probability at least $1-O_n(n^{-A})$,
\begin{enumerate}[label=(\Alph{enumi})]
	\setcounter{enumi}{2}
	\item\label{item-12-deg}  
	$\max_{x\in  \cV(\map_n)\cup\cF(\map_n)} |\nb_x| \le C_2\log n$ and $\max_{i\in [-n,n]_\Z} \deg(i;\cG) \le C_2\log n.$
\end{enumerate}
From here on we work on the event that \ref{item-12-path-G}-\ref{item-12-deg} hold.
\medskip

\noindent\textit{Proof of condition~\ref{item-map-count-G}.}  
For $v_1,v_2\in \cV(\map_{n})$ such that $v_1\sim v_2$ in $\map_n$, let $e$ be the edge connecting them. If both of the two faces of $\map$ with $e$ on their boundaries are not in $\map_n$, then one can check that $v_1,v_2$ satisfy condition~\ref{item:2vertice} in the definition of $\ol\cQ$   and $v_1,v_2\in \cV(\cQ_n)$.  Indeed, if $\{v_1,v_2\}$ is not on a face, it must be of the form $\path(t)$ for some $t\in [-n,n)_\Z$ by the definition of $\map_n$. This means that it is the NW edge of its tail. If it is not the SE edge of its head, then the face on its right will be visited at $t+1$.
By the definition of $\phi_n$, we have that $\phi_n(v_i)=\wt\path^{-1}(v_i)$ for $i=1,2$.  In this case we let $P^{\cG}_{v_1,v_2}=P^{\cG}_{\wt\path^{-1}(v_1),\wt\path^{-1}(v_2)}$.

If there exists $f\in \cF(\map_n)$ with $e$ on its boundary, then  $f\in \cF(\nb_{v_1}) \cap \cF(\nb_{v_2})$.
Now Lemma~\ref{lem:ad}  yields that  $\wt\path(\phi_n(v_i))$ and $f$ are connected by a path $P_i$ in $\ol\cQ_n$ using only vertices and faces in $\nb_{v_i}$ for $i=1,2$.
Let $P^\cG_{v_1,v_2}$ be the  concatenation of  paths  $\wt P^\cG_{\wt\path^{-1}(v),\wt\path^{-1}(w)}$ where $\{v,w\}$ is an edge on $P_1$ or $P_2$. Now 
\begin{equation}\label{eq:fl}
|P_{v_1,v_2}^{\mcl G}|  \overset{\ref{item-12-path-G}}{\le}   2\max_{f\in  \cF(\map_n)}|\nb_f| \cdot C_1(\log n)^3 \overset{\textrm{\ref{item-12-deg}}}{\le} 2C_1C_2(\log n)^4.
\end{equation}

Given $f\sim v$ in $\cQ_n$, if  $\{f,v\}$ is an edge on $P_1$ or $P_2$ defined above for some $v_1,v_2\in \cV(\map)$, we must have $v_1,v_2\in \cV(\nb_f)$.
Therefore for each $j \in [-n , n]_{\BB Z}$, the number of paths of the form $P_{v_1,v_2}^{\mcl G}$ hitting $j$ is at most 
\begin{equation}\label{eq:multiple}
\max_{f\in  \cF(\map_n)} |\nb_f| \cdot C_2(\log n)^6   \overset{\textrm{\ref{item-12-deg}}}{\le}  C_1C_2 (\log n)^7.
\end{equation}
\medskip 

\noindent\textit{Proof of condition~\ref{item-map-count-M}.} 
Given a path $P$ on $\ol\cQ$ of length $m$, we write it as $\{\wt\path(k_i)\}_{1\le i\le m}$. In addition, if $\wt\path(k_1)\in \cF(\map)$ (resp. $\wt\path(k_m)\in \cF(\map)$), we add $\psi(k_1)$ (resp. $\psi(k_m)$) at the beginning (resp. end) of $P$. Denote this possibly augmented path by $\ol{P}$. 
Now we define a path $\Pi(P)$ on $\map_n$ in two steps:
\begin{enumerate}
	\item for any two consecutive faces $f_1,f_2$ of $\map$, insert a vertex $v\in \cV(\map_n)$ which is an endpoint of the edges shared by $f_1,f_2$.  Denote this path by $\ol P'$.
	\item For any three consecutive vertices $v_1,f,v_2$ on $\ol{P}'$ so that $v_1,v_2\in \cV(\map)$ and $f\in \cF(\map)$, replace the segment $[v_1,f,v_2]$ with an arbitrary arc (say, clockwise) on $f$ between $v_1$ and $v_2$. 
\end{enumerate}
For each $i_1,i_2 \in  [-n , n]_{\BB Z}$ with $i_1 \sim i_2$ in $\mcl G$,  set  $P_{i_1,i_2}^{\map}:=\Pi(\wt P_{i_1,i_2}^{\ol\cQ})$. Then $P_{i_1,i_2}^{\map}$ connects $\psi_n(i_1)$ and $\psi_n(i_2)$ and
\begin{equation}\label{eq:dist-bd}
|P_{i_1,i_2}^{\map}| \overset{\textrm{\ref{item-12-path-T}}}{\le}  \max_{f\in \cF(M_{n})} (\#\nb_f) \cdot  2C_1(\log n)^3  \overset{\textrm{\ref{item-12-deg}}}{\le} 2C_1C_2(\log n)^7.
\end{equation} 
For each $v\in \cV(\map_{n} )$, in order for $v$ to be hit by some $P^\map_{i_1,i_2}$,  there must exist a face in $\nb_v$ that is hit by some path $\wt P_{i_1,i_2}^{\ol\cQ}$. 
By \ref{item-12-path-T}, the number of paths of the form $P^\map_{i_1,i_2}$ hitting $v$ is bounded by
\begin{equation}\label{eq:num}
|\nb_v|\cdot C_1(\log n)^6 \overset{\textrm{\ref{item-12-deg}}}{\le}  C_1C_2(\log n)^7.
\end{equation}
\medskip

\noindent\textit{Proof of condition~\ref{item-map-count-close}.} 
For all $v\in \cV(\map_n)$, we can check case by case that 
$v$ and $\psi(\phi_n(v))$  are both on  $\nb_{\wt\path(\phi_n(v))}$.
Therefore by \ref{item-12-deg} $\op{dist}\left(\psi(\phi_n(v)) ,v ; \map_{n} \right)  \le  \max_{f\in \cF(\map_n)} |\nb_f| \le  C_2\log n$.

Now we proceed to bound $ \op{dist}\left(\phi(\psi(i)) , i ; \mcl G_{n} \right)$ for $i\in[-n,n]_\Z$. If $\wt\path(i)\in \cV(\map_n)$, then \ref{item-12-deg} implies  that $\op{dist}\left(\wt\path( \phi_n(\psi(i))) ,\wt\path(i) ; \ol\cQ_{n} \right) =\op{dist}\left(\wt\path(\phi_n(\wt\path(i))) ,\wt\path(i); \ol\cQ_{n} \right) \le C_2 \log n$. 
If $\wt\path(i)\in \cF(\map_n)$, then $\wt\path(\phi_n(\psi (i)  ))$ and $\wt\path(i)$ are both in $\nb_{\psi(i)}\cap \cF(\map_n)$. By Lemma~\ref{lem:ad} and \ref{item-12-deg}, we have  $\op{dist}\left(\wt\path( \phi_n(\psi(i))) ,\wt\path(i) ; \ol\cQ_{n} \right) \le C_2\log n$.
Now  \ref{item-12-path-G} implies that   $\op{dist}\left(\phi_n(\psi(i)) , i ; \mcl G_{n} \right)\le C_2\log n\cdot  C_1(\log n)^3=C_1C_2(\log n)^4$.
\qed

\subsubsection{Bipolar orientations with other weights and Schnyder woods}
\label{subsub:Schnyder}

The only properties of $\cZ=(\cL,\cR)$ that we used in Sections~\ref{subsub:ubom} and~\ref{subsub:bi-metric} are as follows:
\begin{align}\label{eq:step2}
&\textrm{$\{\Delta\cZ_i\}_{i\in \Z}$ are i.i.d with $\E [\Delta\cZ_i]=(0,0)$ and the correlation of $\cL$ and $\cR$ is in $(-1,1)$ }\\
&\textrm{The support of  $\Delta\cZ_1$ is  contained in \eqref{eq:step} and $\E[e^{b |\Delta Z_1 |}] <\infty$ for some $b>0$.}\label{eq:step3}
\end{align}
In particular, as long as  $\cZ$ satisfies \eqref{eq:step2}-\eqref{eq:step3}, we can construct an infinite planar map associated with $\cZ$, analogous to the $\UBOM$, via the sewing procedure. This gives a large family of measures on infinite bipolar-oriented maps. By \cite[Theorem 6, Remark 2]{kmsw-bipolar} and the argument for Proposition~\ref{prop:BS}, this family of infinite bipolar-oriented maps includes the Benjamini-Schramm limits of the uniform measures on size $n$ bipolar-oriented $k$-angulations\footnote{The infinite bipolar-oriented triangulation has been considered in \cite{ghs-bipolar}, where the authors shows that the triangulation and its dual map jointly converge  to two $\SLE_{12}$'s coupled in imaginary geometry.} for any $k\geq 3$.
More generally, one can consider Boltzmann-type distributions where each face has a certain weight. 
We note that the above family includes maps encoded by walks with any correlation in $(-1,0)$, which corresponds to $\kappa>8$ and $\gamma<\sqrt{2}$.

The exact same argument as in Section~\ref{subsub:bi-metric} shows that Theorem~\ref{thm-map-count} still holds for this family of maps (with $M_n$, $\phi_n$, and $\psi_n$ defined in exactly the same way and $\gamma$ determined by $\op{corr}(\mcl L_1, \mcl R_1) = -\cos(\pi\gamma^2/4)$).
We remark that in Section~\ref{subsub:bi-metric} we sometimes  rotate the plane 180 degree and switch $\cZ_t=(\cL_t,\cR_t)$ to $(\cR_{-t},\cL_{-t})$. But we never need the fact  that 
$(\cL_t,\cR_t)\overset{d}{=}(\cR_{-t},\cL_{-t})$. All we need is that the walk $(\cR_{-t},\cL_{-t})$ also satisfies \eqref{eq:step2}-\eqref{eq:step3}, so we do not have to impose this kind of symmetry condition on $\Delta\cZ$ to ensure our generalization of Section~\ref{subsub:bi-metric}.
Thus, we have established Theorem~\ref{thm-map-count} in case~\ref{item-kappa>8}.

As an example, suppose $\cZ$ is a bi-infinite  random walk satisfies \eqref{eq:step2}-\eqref{eq:step3} with step distribution 
\begin{equation}\label{eq:step4}
\nu(-1,1) = 2^{-1} \qquad \textrm{and} \qquad \nu(1,-j) = 2^{-j-2} \quad \textrm{for}\; j\ge 0.
\end{equation} Let $\map^{\cZ}$ be the infinite map produced by $\cZ$ via the sewing procedure. 
For any positive integer $n$ we apply the finite sewing procedure  to the walk $\{\cZ_t \}_{0\le t\le n-1}$. We condition on $\{\cZ_t\}_{0\le t\le n-1}$ starting at (0,0), ending at  $(0,1)$, and staying in $\Z_{\ge 0}\times\Z_{\ge 0}$. This produces a uniform  sample of an $n$-edge bipolar-oriented map with the following property:
\begin{equation}\label{eq:proptery}
\textrm{The  west (resp. east) boundary  length equals 1 (resp. 2), and each bounded face has exactly 2 east edges.}
\end{equation}
By the same argument as in Proposition~\ref{prop:BS},  the Benjamini-Schramm limit as $n\to\infty$ is $\map^{\cZ}$. 

We choose to elaborate on the example of $\map^{\cZ}$ because of its close relation to Schnyder wood-decorated random planar maps. Following the notations in \cite{lsw-schnyder-wood}, consider a simple plane triangulation (i.e. without self-loops or multiple edges) $\map$ with unbounded face $\triangle \Ab \Ar \Ag$, where the \emph{outer vertices} $\Ab$, $\Ar$, and $\Ag$ are arranged in counterclockwise order. We denote by $\overline{\Ar\Ag}$ the edge connecting $\Ar$ and $\Ag$, and similarly for $\overline{\Ag\Ab}$ and $\overline{\Ab\Ar}$.  Vertices, edges, and faces of $M$ other than $\Ab,\Ar,\Ag,\overline{\Ar\Ag},\overline{\Ag\Ab},\overline{\Ab\Ar}$ and $\triangle \Ab\Ar\Ag$ are called \textit{inner} vertices, edges, and faces respectively.
A \textit{3-orientation} on $M$ is an orientation on inner edges where every inner vertex has out-degree $3$. 	Given a 3-orientation $\cO$ on $M$, the following algorithm, gives a  way of coloring the inner edges in blue, red and green:
\begin{enumerate}
	\item Color $\Ab,\Ar,\Ag$ blue, red, and green.
	\item For an inner edge $e$, set $e_1=e$, and inductively for $k \geq 1$ do the following: if the head of $e_k$ is an outer vertex, we  stop. Otherwise, let $e_{k+1}$ be the second outgoing edge encountered when clockwise (or equivalently, counterclockwise) rotating $e_k$ about its head. This procedure always terminates in finitely many steps and yields a path	$\cP=[e_1,e_2,\cdots e_\ell]$.
	\item Assign to $e$ the color of the outer vertex at which $\cP$ terminates.
\end{enumerate} 
It is elementary to check (see e.g. \cite[Introduction]{lsw-schnyder-wood})  that the set of blue edges forms a tree $\Tb$ spanning all inner vertices plus $\Ab$, and
similar statements hold for red and green edges, where we denote the corresponding trees by $\Tr$ and $\Tg$ respectively.  $\{\Tb,\Tr,\Tg \}$ is called the \emph{Schnyder wood} associated with the 3-orientation $\cO$.  
In this paper we treat the notion of 3-orientation and Schnyder wood interchangeably  and call $(\map, \cO)$ a \emph{Schnyder wood-decorated triangulation}, or simply a wooded triangulation. See Figure~\ref{fig:sch} for an illustration.

The Schnyder wood was first introduced by Walter Schnyder \cite{schnyder1989planar} to prove a characterization of graph planarity. He later used Schnyder woods to describe an efficient algorithm for embedding planar graph in such a way that its edges are straight lines and its vertices lie on a grid \cite{Schnyder}. Schnyder's celebrated construction has continued to play an important role in graph theory and enumerative combinatorics \cite{bernardi2007catalan,felsner2008schnyder}.  It is proved in \cite{lsw-schnyder-wood} that uniformly wooded triangulation converges to a triple of $\SLE_{16}$ curves on Liouville quantum gravity with $\gamma=1$ in the peanosphere topology. Although inspired by \cite{lsw-schnyder-wood}, arguments in this paper do not rely on any ingredient from \cite{lsw-schnyder-wood}, thus can be read independently.

Suppose $\map$ has $n$ inner vertices.  Let $\map'$ be the submap of $M$ with all of the green edges removed. Reverse the orientation of each red edge, and give the outer edges the orientations $\Ar\Ab$, $\Ar\Ag$ and $\Ag\Ab$. It is explained in \cite[Proposition 7.1]{Felner} that 
this operation gives a bipolar orientation $(\map',\cO')$ and is a bijection between the set of wooded triangulations with $n$ inner vertices  and the set of bipolar-oriented maps with $2n+3$ edges satisfying \eqref{eq:proptery}. Moreover, to recover $\map$ from $\map'$, one just needs to connect  an edge between any pair vertices $v,w$ where 
\begin{equation}\label{eq:green}
\textrm{$v,w$ lie on opposite sides of a bounded face $f$ of $\map'$ so that neither $v$ nor $w$ is a pole of $f$.}
\end{equation}
These edges form the green tree $\Tg$ that we deleted. Then we can reverse the lower right edge of each face to recover the red tree. In light of this bijection, given a sample of $\map^{\cZ}$, (deterministically) connecting all such pairs in $\cV(\map^{\cZ})$ as in \eqref{eq:green}, we obtain a sample of the Benjamini-Schramm limit of a uniform wooded triangulation with $n$ inner vertices, which we call the \emph{uniform infinite wooded triangulation} or $\UIWT$.\footnote{
	If we take a Benjamini-Schramm limit of a uniform wooded triangulation rooted at typical edge, then the rooted edge is colored blue, red or green with equal probability. But if we first sample $\map^{\cZ}$ then obtain $\UIWT$ by adding edges as in \eqref{eq:green}, the rooted edge must to blue or red. However, if we think of the $\UIWT$ as a vertex rooted map whose vertex is the tail of the root edge, since each vertex has a unique outgoing edge of each color, these two ways of generating $\UIWT$ are equivalent.
}
This provides an alternative proof of the existence of $\UIWT$ other than \cite[Proposition 2.9]{lsw-schnyder-wood}.
\begin{figure}
	\centering
	\includegraphics[scale=0.7]{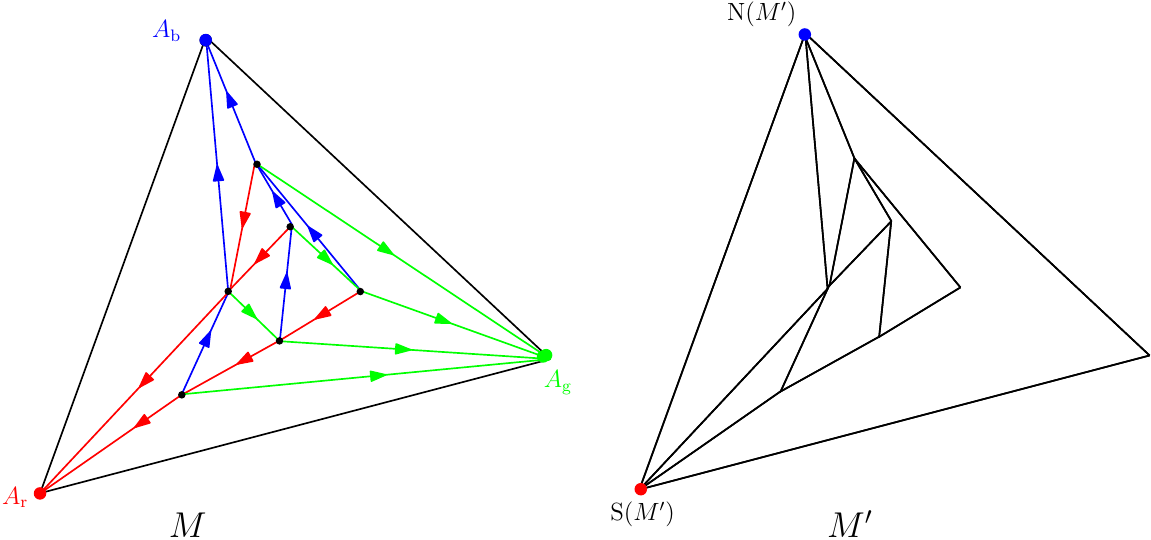}
	\caption[The relationship between bipolar-oriented maps and Schnyder wood-decorated maps]{
		On the left  is the output of the coloring algorithm for a 3-orientation on a simple triangulation, $(\map,\cO)$. On the right is the bipolar-oriented map  $(\map',\cO')$ obtained by deleting all green edges and reversing the direction of red edges. We see that $(\map',\cO')$ satisfies \eqref{eq:proptery}.}
	\label{fig:sch}
\end{figure}
\begin{proof}[Proof of Theorem~\ref{thm-map-count} and Lemma~\ref{lem-ball-layers} for the $\UIWT$]
	Define $\psi_{n},\phi_{n},\wt\path,\nb_{\bullet}, \map^\cZ_n$ in the same way as in Section~\ref{subsub:ubom}-\ref{subsub:bi-metric} with $\map$ replaced by $\map^{\cZ}$.
	Since $\cZ$ satisfies \eqref{eq:step2}-\eqref{eq:step3}, the same argument for case~\ref{item-kappa12} of  Lemma~\ref{lem-ball-layers},  Lemma~\ref{lem-12-deg}, and Theorem~\ref{thm-map-count} yields  the analogs of these results for $\map^{\cZ}$.  Now let $M$ 
	be a sample of $\UIWT$ generated by adding green edges of $\map$ to $\map^{\cZ}$ as in \eqref{eq:green}, with $\map^{\cZ}$ in place of $\map'$.  
	Let $\map_{n}$ be obtained from $M^{\cZ}_{n}$  by adding edges as in \eqref{eq:green} with $M^{\cZ}_{n}$  in place of $\map'$.  
	Since  $\cV(\map)=\cV(\map^{\cZ})$, we see that $\psi_n$ maps $[-n,n]_\Z$ to $\cV(\map_n)$ and $\phi_n$ maps $\cV(\map_n)$ to $[-n,n]_\Z$. 
	
	Since $\cV(\partial M^{\cZ}_{n})=\cV(\partial M_{n})$ for each $n\in \N$, Lemma~\ref{lem-ball-layers} in case~\ref{item-kappa16} follows from the same statement for $\map^{\cZ}$.
	Moreover, conditions~\ref{item-map-count-M} and~\ref{item-map-count-close} in Theorem~\ref{thm-map-count} hold for  $M^{\cZ}_{n}$  implies that the same holds for  $M_{n}$. So it only remains to prove condition~\ref{item-map-count-G}. In this case we work on the event that \ref{item-12-path-G}-\ref{item-12-deg} in Section~\ref{subsub:bi-metric} hold, with $\map^{\cZ}$ in place of the $\UBOM$. In this case the only additional case to consider is when $v_1,v_2$ are connected  by a green edge in $\map_n$.  By \eqref{eq:green} we can find a face $f$ incident to $v_1,v_2$  so that $|\wt\path^{-1}(f)|\le n$.  Therefore the exact same proof of condition~\ref{item-map-count-G} for the $\UBOM$ still works.
\end{proof}

\begin{proof}[Proof of Theorems~\ref{thm-map-ball} and \ref{thm-map-dist} in cases~\ref{item-kappa12} through~\ref{item-kappa16}] 
	This follows from Theorem~\ref{thm-map-count}, Lemma~\ref{lem-ball-iso}, and the results of~\cite{ghs-dist-exponent}, exactly as in case~\ref{item-kappa8}.
\end{proof}

%%%%%%%%%%%%%%%%%%%%%%%%%%%%%%%%%%%%%%%%%%%%%%%%%%%%%%%%%%%%%%%%

\section{Open questions}
\label{sec-questions}

The proofs of our main results extend to any random planar map which can be encoded by a peanosphere-type bijection for which the encoding random walk can be sufficiently closely approximated by Brownian motion. This leads us to ask whether there are any more random planar map models, aside from the ones listed in Section~\ref{sec-main-results}, for which this is the case.

\begin{ques} \label{ques-kmt}
Is there an analogue of the strong coupling theorem~\cite{zaitsev-kmt,kmt} for random walk and Brownian motion for the non-Markovian random walk appearing in Sheffield's hamburger-cheeseburger bijection~\cite{shef-burger} for $p\not=0$, or for the non-Markovian random walk appearing in the generalized hamburger-cheeseburger model of~\cite{gkmw-burger}?
\end{ques}

An affirmative answer to Question~\ref{ques-kmt} would enable us to generalize the results of this paper to critical Fortuin-Kasteleyn planar maps for $q\in (0,4)$ which are encoded by Sheffield's hamburger cheeseburger bijection and belong to the $\gamma$-LQG universality class for $\gamma \in (\sqrt 2, 2)$; and to the planar maps decorated by active spanning trees with bending energy which are encoded by the model of~\cite{gkmw-burger} and which belong to the $\gamma$-LQG universality class for $\gamma \in (0,\sqrt 2)$. 

The results of the present paper can be easily generalized to other random planar map models for which the corresponding encoding walk has i.i.d.\ increments since the needed strong coupling result comes from~\cite{zaitsev-kmt}. 
 
\begin{ques} \label{ques-other-peano}
What random planar map models, other than the ones listed in Section~\ref{sec-main-results}, can be encoded by a mating-of-trees bijection where the random walk has i.i.d.\ increments?
\end{ques}
 
It is natural to ask for a bijective encoding of a planar map decorated by a variant of the Ising model (or some other decorated random planar map model in the $\sqrt 3$-LQG universality class) by a random walk with i.i.d.\ increments, since this model is in some sense the simplest statistical mechanics model for which no such encoding is currently known. Note that the hamburger-cheeseburger bijection encodes a planar map decorated by an instance of the FK Ising model, but the walk in this case does not have i.i.d.\ increments. 

A potential candidate for the encoding walk in such a hypothetical bijection is the \emph{Gessel walk}~\cite{bkr-gessel-walk,kkz-gessel-walk,mbm-gessel-walk}, whose increment distribution is uniform on $\{ (-1,0) , (1,0) , (1,1) , (-1,-1)\}$. This walk is exactly solvable and its coordinates have correlation $1/\sqrt 2 = -\cos(\pi (\sqrt 3)^2/4)$, which is the same as the correlation of the Brownian motion used to define the mated-CRT map when $\gamma = \sqrt 3$. 

We also note that the results of the present paper provide additional motivation to prove stronger estimates for distances in the mated-CRT map $\mcl G$ (improving on the results of~\cite{ghs-dist-exponent,dg-lqg-dim,gp-lfpp-bounds,ang-discrete-lfpp}) since these results can be immediately transferred to other random planar map models by means of Theorem~\ref{thm-map-coupling}.

\bibliography{cibiblong,cibib,ref} 

\def\cprime{$'$}
\begin{thebibliography}{GKMW18}

\bibitem[AB14]{ambjorn-budd-lqg-dist}
J.~{Ambj{\o}rn} and T.~G. {Budd}.
\newblock {Geodesic distances in Liouville quantum gravity}.
\newblock {\em Nuclear Physics B}, 889:676--691, December 2014,
  \arxiv{1405.3424}.

\bibitem[AK15]{abrams-kenyon-enharmonic}
A.~{Abrams} and R.~{Kenyon}.
\newblock {Fixed-energy harmonic functions}.
\newblock {\em ArXiv e-prints}, May 2015, \arxiv{1505.05785}.

\bibitem[AK16]{andres-heat-kernel}
S.~Andres and N.~Kajino.
\newblock Continuity and estimates of the {L}iouville heat kernel with
  applications to spectral dimensions.
\newblock {\em Probab. Theory Related Fields}, 166(3-4):713--752, 2016,
  \arxiv{1407.3240}. \MR{3568038}

\bibitem[Ald91a]{aldous-crt1}
D.~Aldous.
\newblock The continuum random tree. {I}.
\newblock {\em Ann. Probab.}, 19(1):1--28, 1991. \MR{1085326 (91i:60024)}

\bibitem[Ald91b]{aldous-crt2}
D.~Aldous.
\newblock The continuum random tree. {II}. {A}n overview.
\newblock In {\em Stochastic analysis ({D}urham, 1990)}, volume 167 of {\em
  London Math. Soc. Lecture Note Ser.}, pages 23--70. Cambridge Univ. Press,
  Cambridge, 1991. \MR{1166406 (93f:60010)}

\bibitem[Ald93]{aldous-crt3}
D.~Aldous.
\newblock The continuum random tree. {III}.
\newblock {\em Ann. Probab.}, 21(1):248--289, 1993. \MR{1207226 (94c:60015)}

\bibitem[Ang03]{angel-peeling}
O.~Angel.
\newblock Growth and percolation on the uniform infinite planar triangulation.
\newblock {\em Geom. Funct. Anal.}, 13(5):935--974, 2003, \arxiv{0208123}.
  \MR{2024412}

\bibitem[{Ang}05]{angel-uihpq-perc}
O.~{Angel}.
\newblock {Scaling of Percolation on Infinite Planar Maps, I}.
\newblock {\em ArXiv Mathematics e-prints}, December 2005,
  \arxiv{math/0501006}.

\bibitem[Ang19]{ang-discrete-lfpp}
M.~Ang.
\newblock {Comparison of discrete and continuum Liouville first passage
  percolation}.
\newblock {\em ArXiv e-prints}, Apr 2019, \arxiv{1904.09285}.

\bibitem[AS03]{angel-schramm-uipt}
O.~Angel and O.~Schramm.
\newblock Uniform infinite planar triangulations.
\newblock {\em Comm. Math. Phys.}, 241(2-3):191--213, 2003. \MR{2013797
  (2005b:60021)}

\bibitem[BB09]{bernardi2007catalan}
O.~Bernardi and N.~Bonichon.
\newblock Intervals in {C}atalan lattices and realizers of triangulations.
\newblock {\em J. Combin. Theory Ser. A}, 116(1):55--75, 2009.

\bibitem[BBMF11]{bbf-bipolar-bijection}
N.~Bonichon, M.~Bousquet-M{\'e}lou, and {\'E}.~Fusy.
\newblock Baxter permutations and plane bipolar orientations.
\newblock {\em S\'em. Lothar. Combin.}, 61A:Art. B61Ah, 29, 2009/11,
  \arxiv{arXiv:0805.4180}. \MR{2734180 (2011m:05023)}

\bibitem[BDFG04]{bdg-bijection}
J.~Bouttier, P.~Di~Francesco, and E.~Guitter.
\newblock Planar maps as labeled mobiles.
\newblock {\em Electron. J. Combin.}, 11(1):Research Paper 69, 27, 2004,
  \arxiv{math/0405099}. \MR{2097335 (2005i:05087)}

\bibitem[Ber07a]{bernardi-dfs-bijection}
O.~Bernardi.
\newblock Bijective counting of {K}reweras walks and loopless triangulations.
\newblock {\em J. Combin. Theory Ser. A}, 114(5):931--956, 2007.

\bibitem[Ber07b]{bernardi-maps}
O.~Bernardi.
\newblock Bijective counting of tree-rooted maps and shuffles of parenthesis
  systems.
\newblock {\em Electron. J. Combin.}, 14(1):Research Paper 9, 36 pp.
  (electronic), 2007, \arxiv{math/0601684}. \MR{2285813 (2007m:05125)}

\bibitem[BHS18]{bhs-site-perc}
O.~{Bernardi}, N.~{Holden}, and X.~{Sun}.
\newblock {Percolation on triangulations: a bijective path to Liouville quantum
  gravity}.
\newblock {\em ArXiv e-prints}, July 2018, \arxiv{1807.01684}.

\bibitem[BKR17]{bkr-gessel-walk}
A.~Bostan, I.~Kurkova, and K.~Raschel.
\newblock A human proof of {G}essel's lattice path conjecture.
\newblock {\em Trans. Amer. Math. Soc.}, 369(2):1365--1393, 2017,
  \arxiv{1309.1023}. \MR{3572277}

\bibitem[BLR17]{blr-exponents}
N.~Berestycki, B.~Laslier, and G.~Ray.
\newblock Critical exponents on {F}ortuin-{K}asteleyn weighted planar maps.
\newblock {\em Comm. Math. Phys.}, 355(2):427--462, 2017, \arxiv{1502.00450}.
  \MR{3681382}

\bibitem[BM16]{mbm-gessel-walk}
M.~Bousquet-M\'elou.
\newblock An elementary solution of {G}essel's walks in the quadrant.
\newblock {\em Adv. Math.}, 303:1171--1189, 2016, \arxiv{1503.08573}.
  \MR{3552547}

\bibitem[BS01]{benjamini-schramm-topology}
I.~Benjamini and O.~Schramm.
\newblock Recurrence of distributional limits of finite planar graphs.
\newblock {\em Electron. J. Probab.}, 6:no. 23, 13 pp. (electronic), 2001,
  \arxiv{0011019}. \MR{1873300 (2002m:82025)}

\bibitem[Che17]{chen-fk}
L.~Chen.
\newblock Basic properties of the infinite critical-{FK} random map.
\newblock {\em Ann. Inst. Henri Poincar\'e D}, 4(3):245--271, 2017,
  \arxiv{1502.01013}. \MR{3713017}

\bibitem[dFdMR95]{fmor-bipolar}
H.~de~Fraysseix, P.~O. de~Mendez, and P.~Rosenstiehl.
\newblock Bipolar orientations revisited.
\newblock {\em Discrete Appl. Math.}, 56(2-3):157--179, 1995.
\newblock Special Issue: Fifth Franco-Japanese Days (Kyoto, 1992). \MR{1318743
  (96i:05073)}

\bibitem[DG16]{ding-goswami-watabiki}
J.~{Ding} and S.~{Goswami}.
\newblock {Upper bounds on Liouville first passage percolation and Watabiki's
  prediction}.
\newblock {\em {C}ommunications in {P}ure and {A}pplied {M}athematics}, to
  appear, 2016, \arxiv{1610.09998}.

\bibitem[DG18]{dg-lqg-dim}
J.~{Ding} and E.~{Gwynne}.
\newblock {The fractal dimension of {L}iouville quantum gravity: universality,
  monotonicity, and bounds}.
\newblock {\em {C}ommunications in {M}athematical {P}hysics}, to appear, 2018,
  \arxiv{1807.01072}.

\bibitem[DMS14]{wedges}
B.~{Duplantier}, J.~{Miller}, and S.~{Sheffield}.
\newblock {Liouville quantum gravity as a mating of trees}.
\newblock {\em ArXiv e-prints}, September 2014, \arxiv{1409.7055}.

\bibitem[DW15a]{dw-cones}
D.~Denisov and V.~Wachtel.
\newblock Random walks in cones.
\newblock {\em Ann. Probab.}, 43(3):992--1044, 2015, \arxiv{1110.1254}.
  \MR{3342657}

\bibitem[DW15b]{dw-limit}
J.~{Duraj} and V.~{Wachtel}.
\newblock {Invariance principles for random walks in cones}.
\newblock {\em ArXiv e-prints}, August 2015, \arxiv{1508.07966}.

\bibitem[DZZ18]{dzz-heat-kernel}
J.~{Ding}, O.~{Zeitouni}, and F.~{Zhang}.
\newblock {Heat kernel for Liouville Brownian motion and Liouville graph
  distance}.
\newblock {\em {C}ommunications in {M}athematical {P}hysics}, to appear, 2018,
  \arxiv{1807.00422}.

\bibitem[FFNO11]{Felner}
S.~Felsner, E.~Fusy, M.~Noy, and D.~Orden.
\newblock Bijections for {B}axter families and related objects.
\newblock {\em J. Combin. Theory Ser. A}, 118(3):993--1020, 2011. \MR{2763051}

\bibitem[FK72]{fk-cluster}
C.~M. Fortuin and P.~W. Kasteleyn.
\newblock On the random-cluster model. {I}. {I}ntroduction and relation to
  other models.
\newblock {\em Physica}, 57:536--564, 1972. \MR{0359655 (50 \#12107)}

\bibitem[FPS09]{fps-counting-bipolar}
{\'E}.~Fusy, D.~Poulalhon, and G.~Schaeffer.
\newblock Bijective counting of plane bipolar orientations and {S}chnyder
  woods.
\newblock {\em European J. Combin.}, 30(7):1646--1658, 2009,
  \arxiv{arXiv:0803.0400}. \MR{2548656 (2010j:05047)}

\bibitem[FZ08]{felsner2008schnyder}
S.~Felsner and F.~Zickfeld.
\newblock Schnyder woods and orthogonal surfaces.
\newblock {\em Discrete Comput. Geom.}, 40(1):103--126, 2008.

\bibitem[GH18]{gh-displacement}
E.~{Gwynne} and T.~{Hutchcroft}.
\newblock {Anomalous diffusion of random walk on random planar maps}.
\newblock {\em ArXiv e-prints}, July 2018, \arxiv{1807.01512}.

\bibitem[GHS16]{ghs-bipolar}
E.~{Gwynne}, N.~{Holden}, and X.~{Sun}.
\newblock {Joint scaling limit of a bipolar-oriented triangulation and its dual
  in the peanosphere sense}.
\newblock {\em ArXiv e-prints}, March 2016, \arxiv{1603.01194}.

\bibitem[GHS19]{ghs-dist-exponent}
E.~{Gwynne}, N.~{Holden}, and X.~{Sun}.
\newblock {A distance exponent for Liouville quantum gravity}.
\newblock {\em {Probability Theory and Related Fields}}, 173(3):931--997, 2019,
  \arxiv{1606.01214}.

\bibitem[GKMW18]{gkmw-burger}
E.~Gwynne, A.~Kassel, J.~Miller, and D.~B. Wilson.
\newblock Active {S}panning {T}rees with {B}ending {E}nergy on {P}lanar {M}aps
  and {SLE}-{D}ecorated {L}iouville {Q}uantum {G}ravity for {$\kappa > 8$}.
\newblock {\em Comm. Math. Phys.}, 358(3):1065--1115, 2018, \arxiv{1603.09722}.
  \MR{3778352}

\bibitem[GM17]{gm-spec-dim}
E.~{Gwynne} and J.~{Miller}.
\newblock {Random walk on random planar maps: spectral dimension, resistance,
  and displacement}.
\newblock {\em ArXiv e-prints}, November 2017, \arxiv{1711.00836}.

\bibitem[GM19]{gm-uniqueness}
E.~Gwynne and J.~Miller.
\newblock Existence and uniqueness of the {L}iouville quantum gravity metric
  for {$\gamma \in (0,2)$}.
\newblock {\em ArXiv e-prints}, May 2019, \arxiv{1905.00383}.

\bibitem[GMS17]{gms-tutte}
E.~{Gwynne}, J.~{Miller}, and S.~{Sheffield}.
\newblock {The Tutte embedding of the mated-CRT map converges to Liouville
  quantum gravity}.
\newblock {\em ArXiv e-prints}, May 2017, \arxiv{1705.11161}.

\bibitem[GMS19]{gms-burger-cone}
E.~Gwynne, C.~Mao, and X.~Sun.
\newblock Scaling limits for the critical {F}ortuin--{K}asteleyn model on a
  random planar map {I}: {C}one times.
\newblock {\em Ann. Inst. Henri Poincar\'{e} Probab. Stat.}, 55(1):1--60, 2019,
  \arxiv{1502.00546}. \MR{3901640}

\bibitem[GP19a]{gp-lfpp-bounds}
E.~{Gwynne} and J.~{Pfeffer}.
\newblock {Bounds for distances and geodesic dimension in Liouville first
  passage percolation}.
\newblock {\em {E}lectronic {C}ommunications in {P}robability}, 24:no. 56, 12,
  2019, \arxiv{1903.09561}.

\bibitem[GP19b]{gp-dla}
E.~{Gwynne} and J.~{Pfeffer}.
\newblock {External diffusion limited aggregation on a spanning-tree-weighted
  random planar map}.
\newblock {\em ArXiv e-prints}, January 2019, \arxiv{1901.06860}.

\bibitem[GP19c]{gp-kpz}
E.~{Gwynne} and J.~{Pfeffer}.
\newblock {KPZ formulas for the Liouville quantum gravity metric}.
\newblock {\em {T}ransactions of the {A}merican {M}athematical {S}ociety}, to
  appear, 2019.

\bibitem[KKZ09]{kkz-gessel-walk}
M.~{Kauers}, C.~{Koutschan}, and D.~{Zeilberger}.
\newblock {Proof of Ira Gessel's lattice path conjecture}.
\newblock {\em Proceedings of the National Academy of Science},
  106:11502--11505, July 2009, \arxiv{0806.4300}.

\bibitem[KMSW16]{kmsw-6vertex}
R.~{Kenyon}, J.~{Miller}, S.~{Sheffield}, and D.~B. {Wilson}.
\newblock {The six-vertex model and Schramm-Loewner evolution}.
\newblock {\em ArXiv e-prints}, May 2016, \arxiv{1605.06471}.

\bibitem[KMSW19]{kmsw-bipolar}
R.~Kenyon, J.~Miller, S.~Sheffield, and D.~B. Wilson.
\newblock Bipolar orientations on planar maps and {${\rm SLE}_{12}$}.
\newblock {\em Ann. Probab.}, 47(3):1240--1269, 2019, \arxiv{1511.04068}.
  \MR{3945746}

\bibitem[KMT76]{kmt}
J.~Koml\'os, P.~Major, and G.~Tusn\'ady.
\newblock An approximation of partial sums of independent {RV}'s, and the
  sample {DF}. {II}.
\newblock {\em Z. Wahrscheinlichkeitstheorie und Verw. Gebiete}, 34(1):33--58,
  1976. \MR{0402883}

\bibitem[{Le }05]{legall-tree-survey}
J.-F. {Le Gall}.
\newblock Random trees and applications.
\newblock {\em Probab. Surv.}, 2:245--311, 2005, \arxiv{math/0511515}.
  \MR{2203728 (2007h:60078)}

\bibitem[{Le }13]{legall-uniqueness}
J.-F. {Le Gall}.
\newblock Uniqueness and universality of the {B}rownian map.
\newblock {\em Ann. Probab.}, 41(4):2880--2960, 2013, \arxiv{1105.4842}.
  \MR{3112934}

\bibitem[LL10]{lawler-limic-walks}
G.~F. Lawler and V.~Limic.
\newblock {\em Random walk: a modern introduction}, volume 123 of {\em
  Cambridge Studies in Advanced Mathematics}.
\newblock Cambridge University Press, Cambridge, 2010. \MR{2677157
  (2012a:60132)}

\bibitem[LSW17]{lsw-schnyder-wood}
Y.~{Li}, X.~{Sun}, and S.~S. {Watson}.
\newblock {Schnyder woods, SLE(16), and Liouville quantum gravity}.
\newblock {\em ArXiv e-prints}, May 2017, \arxiv{1705.03573}.

\bibitem[Mie13]{miermont-brownian-map}
G.~Miermont.
\newblock The {B}rownian map is the scaling limit of uniform random plane
  quadrangulations.
\newblock {\em Acta Math.}, 210(2):319--401, 2013, \arxiv{1104.1606}.
  \MR{3070569}

\bibitem[MRVZ16]{mrvz-heat-kernel}
P.~Maillard, R.~Rhodes, V.~Vargas, and O.~Zeitouni.
\newblock Liouville heat kernel: regularity and bounds.
\newblock {\em Ann. Inst. Henri Poincar\'e Probab. Stat.}, 52(3):1281--1320,
  2016, \arxiv{1406.0491}. \MR{3531710}

\bibitem[MS15]{lqg-tbm1}
J.~{Miller} and S.~{Sheffield}.
\newblock {Liouville quantum gravity and the Brownian map I: The QLE(8/3,0)
  metric}.
\newblock {\em Inventiones Mathematicae}, to appear, 2015, \arxiv{1507.00719}.

\bibitem[MS16a]{lqg-tbm2}
J.~{Miller} and S.~{Sheffield}.
\newblock {Liouville quantum gravity and the Brownian map II: geodesics and
  continuity of the embedding}.
\newblock {\em ArXiv e-prints}, May 2016, \arxiv{1605.03563}.

\bibitem[MS16b]{lqg-tbm3}
J.~{Miller} and S.~{Sheffield}.
\newblock {Liouville quantum gravity and the Brownian map III: the conformal
  structure is determined}.
\newblock {\em ArXiv e-prints}, August 2016, \arxiv{1608.05391}.

\bibitem[MS17]{ig4}
J.~Miller and S.~Sheffield.
\newblock Imaginary geometry {IV}: interior rays, whole-plane reversibility,
  and space-filling trees.
\newblock {\em Probab. Theory Related Fields}, 169(3-4):729--869, 2017,
  \arxiv{1302.4738}. \MR{3719057}

\bibitem[Mul67]{mullin-maps}
R.~C. Mullin.
\newblock On the enumeration of tree-rooted maps.
\newblock {\em Canad. J. Math.}, 19:174--183, 1967. \MR{0205882 (34 \#5708)}

\bibitem[Sch89]{schnyder1989planar}
W.~Schnyder.
\newblock Planar graphs and poset dimension.
\newblock {\em Order}, 5(4):323--343, 1989.

\bibitem[Sch90]{Schnyder}
W.~Schnyder.
\newblock Embedding planar graphs on the grid.
\newblock In {\em Proceedings of the First Annual ACM-SIAM Symposium on
  Discrete Algorithms}, SODA '90, pages 138--148, Philadelphia, PA, USA, 1990.
  Society for Industrial and Applied Mathematics.

\bibitem[Sch97]{schaeffer-bijection}
G.~Schaeffer.
\newblock Bijective census and random generation of {E}ulerian planar maps with
  prescribed vertex degrees.
\newblock {\em Electron. J. Combin.}, 4(1):Research Paper 20, 14 pp.\
  (electronic), 1997. \MR{1465581 (98g:05074)}

\bibitem[She16]{shef-burger}
S.~Sheffield.
\newblock Quantum gravity and inventory accumulation.
\newblock {\em Ann. Probab.}, 44(6):3804--3848, 2016, \arxiv{1108.2241}.
  \MR{3572324}

\bibitem[Wat93]{watabiki-lqg}
Y.~Watabiki.
\newblock {Analytic study of fractal structure of quantized surface in
  two-dimensional quantum gravity}.
\newblock {\em Progr. Theor. Phys. Suppl.}, (114):1--17, 1993.
\newblock Quantum gravity (Kyoto, 1992).

\bibitem[Zai98]{zaitsev-kmt}
A.~Y. Zaitsev.
\newblock Multidimensional version of the results of {K}oml\'os, {M}ajor and
  {T}usn\'ady for vectors with finite exponential moments.
\newblock {\em ESAIM Probab. Statist.}, 2:41--108, 1998. \MR{1616527}

\end{thebibliography}
\bibliographystyle{hmralphaabbrv}

\end{document}